\documentclass[11pt,a4paper,reqno]{amsart}
\usepackage{amsmath,amsfonts,amssymb,amsthm, color}
\usepackage[T1]{fontenc}
\allowdisplaybreaks

\advance\textwidth by 2.5cm %%3.6
\advance\oddsidemargin by -1.6cm
\advance\evensidemargin by -1.6cm

\newcommand\toup{\nearrow}
\newcommand\todown{\searrow}

\newcommand\eps{\varepsilon}
\newcommand{\del}[1]{}
\input colordvi
\usepackage{amscd}
\usepackage{mathrsfs}
\usepackage{verbatim}
\theoremstyle{plain}
\newtheorem{theorem}{Theorem}[section]
\theoremstyle{remark}
\newtheorem{remark}[theorem]{Remark}

\theoremstyle{plain}
\newtheorem{corollary}[theorem]{Corollary}
\newtheorem{lemma}[theorem]{Lemma}
\newtheorem{proposition}[theorem]{Proposition}
\newtheorem{definition}[theorem]{Definition}

\newtheorem{assumption}[theorem]{Assumption}

\numberwithin{equation}{section}

\def\R{{\mathbb R}\,}

\def\C{\mathbb{C}\,}
\def\E{\mathbb{E}\,}
\newcommand{\RR}{{\mathbb R}}
\newcommand{\tg}{\tilde{g}}

\def\cR{\mathcal{R}\,}

\newcommand{\rH}{{\rm H}}
\newcommand{\rK}{{\rm K}}

\newcommand{\rE}{{\rm E}}
\newcommand{\rR}{{\rm R}}

\renewcommand{\Re}{\mathrm{Re}\,}
\newcommand{\n}{\,\hbox{\vrule width 1pt height 8ptdepth 2pt}\,}

\newcommand{\embed}{\hookrightarrow}

\def\old#1{}

\def\text#1{{\rm #1}}
\def\newold#1{}

\def\dim{{\rm dim}\,}

\def\tr{\mathrm{tr}\,}

\begin{document}
%\begin{document}

\baselineskip 12pt%13pt

\title[Stochastic Strichartz estimates]
{On the stochastic Strichartz estimates and the stochastic nonlinear Schr{\"o}dinger
equation on a compact riemannian  manifold}
\author[Z. Brze{\'z}niak]{Z. Brze{\'z}niak}
\address{Department of Mathematics\\
The University of York\\
Heslington, York YO10 5DD, UK} \email{zb500@york.ac.uk}
\author[A. Millet]{A. Millet}
\address{SAMM, EA  4543 (and PMA, UMR 7599)\\
 Universit\'e  Paris 1\\
90 Rue de Tolbiac, 75634 Paris Cedex 13, France} \email{amillet@univ-paris1.fr}
\date{\today}

%\thanks{}

%\keywords{}

%\subjclass[2000]{Primary: ; Secondary: }

\begin{abstract}
We prove the existence and the uniqueness of a solution to the stochastic NSLEs on a
two-dimensional compact riemannian manifold. Thus we generalize (and improve) a recent work by Burq et all
 \cite{Burq+G+T_2003} and a  series of papers by de Bouard and Debussche,
 see  e.g. \cite{deBouard+Deb_1999,deBouard+Deb_2003} who have examined similar questions
 in the case of the flat euclidean space.\\
We prove the existence and the uniqueness of a local maximal solution to stochastic nonlinear Schr\"odinger equations
 with multiplicative noise on a compact $d$-dimensional  riemannian manifold.
Under more regularity on the noise, we prove that the solution is global when the nonlinearity is of defocusing or of focusing type,  $d=2$
and the initial data belongs to the finite energy space. Our  proof is based on improved stochastic Strichartz inequalities.
\end{abstract}

\maketitle

\section{Introduction}\label{sec-intro}

The  aim is this paper is twofold. The first one is to  generalise Theorem 2 from Burq et all \cite{Burq+G+T_2004} on the global existence
 of nonlinear Schr\"odinger equation (NLS) to a stochastic setting. The second one is to prove a general version of a Strichartz type
 inequality for stochastic convolutions. This inequality is
the main technical improvement of our results as compared to papers by ~de Bouard and ~Debussche
\cite{deBouard+Deb_1999} and
 \cite{deBouard+Deb_2003}.  In the  deterministic case our existence result is comparable with
  \cite[Theorem 2]{Burq+G+T_2004} and our stochastic Strichartz inequality is comparable with
  \cite[Corollary 2.10]{Burq+G+T_2004}.  Some versions of the stochastic Strichartz inequalities were implicitly
 formulated in  \cite{deBouard+Deb_1999}, see Lemma 3.1 and Corollary 3.2 or \cite{deBouard+Deb_2003}.
 The   proof in  \cite{deBouard+Deb_1999} and \cite{deBouard+Deb_2003}  is  based on the dispersive estimates for the
free Schr\"odinger group. However, see \cite[Remark 2.6]{Burq+G+T_2004},
such estimates are not valid in the case of a compact riemannian manifold;
 hence we had to rely on different methods. Thus,  not only the proof but also the
result differs from the corresponding result from  \cite{deBouard+Deb_2003} but
for stochastic integrands our result is stronger than the previous ones.
The local existence of the solution to the stochastic NLS equation
can be obtained with some non linear diffusion coefficient.
  Another extension involves
both the regularity of the noise and the form of the diffusion coefficient which ensure that blow-up does not occur
in finite time.

As in the above cited papers by de Bouard and Debussche, the global existence result which is a consequence of the  conservation of the $L^2(M)$
norm and the control of a certain  Lyapounov function, requires our problem to be in the  Stratonovich form.
To provide a uniform treatment of equations in both the It\^o and the Stratonovich forms, we assume that all
 vector spaces are over the field $\mathbb{R}$ of real numbers. In particular,
the space $L^2(M,\mathbb{C})$ is considered as a real vector space of all (equivalence classes)
 from $M$ to $\mathbb{R}^2$.  As a consequence, by a linear/bilinear  map we understand a linear/bilinear
 map over the field $\mathbb{R}$.   Similarly, we speak only about $\mathbb{R}$-differentiability.
 In general, the spaces of all $\mathbb{R}$-linear, resp. bilinear,  bounded maps from $E$,
 resp. $E\times E$, to $X$,  where $E$ and $X$  are two real Banach spaces, will be denoted
 by $\mathcal{L}(E,X)$, resp. $\mathbb{L}_2(E,X)$.
\smallskip

Let us briefly describe the main results obtained in literature preceding our article. In \cite{deBouard+Deb_1999}, see Theorem 2.1,
de Bouard and Debussche proved the  existence of a continuous  $L^2({\mathbb R}^d)$-valued global solution for the NSLEs   when  the initial data  was an $L^2(\mathbb{R}^d)$-valued random variable.
 In the subsequent paper \cite{deBouard+Deb_2003}, the same authours  proved the existence of a continuous $H^{1,2}(\mathbb{R}^d)$-valued local and global solution  (depending on the nonlinearity)
when  the initial data was  an $H^{1,2}(\mathbb{R}^d)$-valued random variable.

Now let us briefly describe the content of the current article.
 In section \ref{sec-Nemytski} we study properties of the Nemytski operators. In particular we generalize the results from \cite{Brz+Elw_2000}
 which we proved various properties of the Nemytski operators in the Sobolev-Slobodetski spaces $W^{\theta,q}(D)$,
where $D \subset \mathbb{R}^d$, $\theta>\frac{d}{q}$, to the spaces $W^{\theta,q}(D)\cap L^\infty(D)$, without any restriction on $\theta$.
 This is quite important since later on we work with spaces $W^{\theta,2}(D)$, where $d=2$ and $\theta \leq q$.
Section \ref{sec-Strichartz} is devoted to the study of the stochastic Strichartz inequality.
 We recall the homogenous and inhomogeneous Strichartz estimates from \cite{Burq+G+T_2004}.
Then we prove our main result from this section: a  stochastic Strichartz inequality. This inequality is a generalization
of the inhomogeneous Strichartz estimates from \cite{Burq+G+T_2004} and our argument is somehow motivated by the proof of the latter.
 However, we use the Burkholder inequality in the space $M^p(0,T;E)$, where $E$ is a $2$-smooth Banach space, and the Kahane-Khinchin inequality.
 Somehow related results have been obtained in very concrete setting by De Bouard and Debussche in
 \cite{deBouard+Deb_1999,deBouard+Deb_2003} for the stochastic Schr\"odinger equation and by Ondrej\'at \cite{Ondrejat_2010_JDE} for the
 stochastic wave  equation. We believe that our result is new and that the approach has potential other applications.
In the following section \ref{sec-abstract-local} we formulate an abstract result about the existence and uniqueness
 of the  maximal local  solution  for stochastic evolution equations of NLS type, that is
\[ i du(t) + \Delta u(t) = f(u(t)) dt + g(u(t)) dW(t).\]
 As in \cite{deBouard+Deb_1999,deBouard+Deb_2003}
we study the existence and uniqueness of solutions to appropriate approximated problems. The proof
of the main result from this section, Theorem \ref{thm_maximal-abstract}, which states the existence of a maximal solution
is then given in section \ref{sec-local-sol}. Note that even if the lifetime $\tau_\infty$ of the solution is defined in terms of the sum
of two norms, we prove that the $H^{1,2}$ norm of $u(t)$ explodes as $t\toup \tau_\infty <\infty$.
In Section  \ref{sec-Stratonovich} we describe an abstract formulation
of the NLS equation in Stratonovich form. As in \cite{deBouard+Deb_1999,deBouard+Deb_2003} this is needed  to obtain a global solution since
the $L^2$-norm of the solution is preserved in this formulation.

We then restrict the framework as follows. We consider a $2$-dimensional compact riemannian manifold $M$, a regular function  $\tilde{g} : {\mathbb R}\to {\mathbb R}$
 and the diffusion coefficient of the form $g(u)=\tilde{g}(|u|^2) u$.
In  section \ref{sec-local}  we establishe the local existence and uniqueness of the
solution to the specific NLS equation in Stratonovich form
\[ i du(t) + \Delta u(t) = f(u(t)) \,dt + \tilde{g}(|u(t)|^2) u(t) \circ dW(t).\]
 Note that unlike the usual parabolic case, see \cite{Brz+Elw_2000},
the Stratonovich correction term does not contain the derivative of $g$ or of $\tilde{g}$.
This is somehow similar to the case of stochastic  wave equation  where the  Stratonovich correction term is equal to $0$;
see for instance \cite{Millet_Sanz} and \cite{Brz+Ondrejat_2011}.

Finally, in  section \ref{sec-global}  we deal with the existence and uniqueness of a global solution
for  2-dimensional manifold, an initial condition $u_0\in H^{1,2}$,
when the drift and the  diffusion coefficients have the specific form $f(u)=\tilde{f}(|u|^2) u$
and $f(u)=\tilde{g}(|u|^2) u$ respectively (see Theorem \ref{th_global}).
 This is the natural framework to extend the deterministic well-posedeness proved in Burq et al in \cite{Burq+G+T_2004}.     The function $\tilde{f}$ is either defocusing, that is
 a polynomial with a positive leading coefficient or  $\tilde{f}(r)=C r^\sigma$
with $C>0$ and $\sigma \in [ \frac{1}{2},\infty)$, or  $\tilde{f}$ is focusing, that is
$\tilde{f}(r) = -C r^\sigma$ with $\sigma \in [\frac{1}{2},1)$
and $C>0$.
The stochastic integral is defined in Stratonovich form and depending on the
nonlinearity, some more conditions have to be imposed on the noise $W$. More precisely, $W$ has to take values
 in a sub algebra $H^{1,2}(M)\cap H^{1,\alpha}(M)$
of the space $H^{1,2}(M) \cap W^{\hat{s},q}(M)$, with $\hat{s}=1-\frac{1}{p}$ and $2/p+2/q=1$,
for some $p$ chosen from the non linear term $f$.
Note that this extends both the previous results of de Bouard and Debussche  \cite{deBouard+Deb_2003}
 who imposed $\tilde{g}=1$ and some more smoothness on the noise,
and also the deterministic global existence in \cite{Burq+G+T_2004}.

Unlike  \cite{Burq+G+T_2004} in the deterministic case,  we did not try to shift  the regularity of the initial condition
and that of the solution.
Note that several other problems were not considered here, such as the three dimensional global
 existence which was solved in the
deterministic case for a cubic nonlinearity, and the finite time blow-up which was proven by
de Bouard and Debussche \cite{deBouard+Deb_2005}
in the flat
case of ${\mathbb R}^d$ with multiplicative noise. These will  addressed in fortcoming papers.

 \textbf{Notation:} % \label{sec:notation}
Unless otherwise stated, all vector spaces are complex but treated as real vector spaces.
To ease notations, we will denote  $C$ a generic constant which can change from one line to the next.

\section{Nemytski operator}\label{sec-Nemytski}

 This section relies heavily on ideas from reference \cite{Brz+Elw_2000}.
Assume that we are given two real numbers $q\in [1,\infty )$ and $\theta\in (0,1)$.
Let $M$ be  a compact $d$-dimensional
 riemannian manifold (without boundary).
For $x_1,x_2\in M$, $ \vert x_2-x_1\vert$ we will denote the riemannian distance between
 $x_1$ and $x_2$. By $\mu$ we will denote the riemannian volume measure on $M$ but
  the integration with respect to $\mu$ we will denote by $dx$.
By $L^q(M)$, $q\in [1,\infty]$ we will denote the classical real Banach space  of all
[equivalence classes of] $\mathbb{C}\cong\mathbb{R}^2$-valued $q$-integrable functions on $M$,
endowed with the classical norm which will be denoted by $\vert \cdot \vert_q$
(or sometimes, in danger of ambiguity by $\vert \cdot \vert_{L^q}$).

Let us  recall a definition of the classical Sobolev spaces $H^{\theta,q}(M)$
and of the Besov-Slobodetski space $W^{\theta ,q}(M)$. The former space is defined as
 the complex interpolation space\\ $[L^q(M),H^{k,q}(M)]_{\frac{\theta}k}$, where $k$ is a natural
number bigger that $\theta$. It can be shown that  $ H^{\theta,q}(M):=D((-\Delta_q)^{\theta/2})$,
 where $\Delta_q$ is the Laplace-Beltrami operator on $L^q(M)$, i.e. the infinitesimal generator
of the heat semigroup on the space $L^q(M)$. The latter space,  defined by
\begin{equation}
W^{\theta ,q}(M):=\left\{f\in L^q(M): \n f\n_{\theta,q}^q:=\int_M\int_M\frac {
\vert f(x_2)-f(x_1)\vert^q}{\vert x_2-x_1\vert^{d+\theta q}}\, dx_1\,dx_2
<\infty\right\} ,
\label{nem-1}
\end{equation}
%%
%where the integration is with respect to the volume measure on $M$,
is endowed with the norm $\Vert f\Vert_{W^{\theta ,q}(M)}:= \vert f\vert_{L^q(M)}  +  \n f\n_{\theta ,q}$.
Note that  $W^{\theta,2}(M)=H^{\theta,2}(M)$, whereas  $H^{\theta,q}(M)\subset W^{\theta,q}(M)$
(or $H^{\theta,q}(M)\supset W^{\theta,q}(M)$) depending whether $q>2$ (or $q<2$).
Let us recall that  the Besov-Slobodetski spaces with fractional
order $\theta \in (0,1)$ are equal to the  real interpolation
spaces of order $\theta$ between the spaces
$L^q(M)$ and $H^{1,q}(M)=W^{1,q}(M)$; see for instance \cite{Triebel_1978}. We will also use the notation
\[\n f\n_{1,q}^q:= \int_{M} \vert \nabla f \vert^q\, dx.
\]
Thus, $f\in W^{1,q}(M)$ iff $f\in L^q(M)$ and $\n f\n_{1,q}<\infty$.

The space $\mathbb{R}^2$ can be replaced by any real separable Banach space $Y$
for the $L^q$ spaces and similarly, the Besov-Slobodetski
$W^{\theta,q}$ spaces can naturally
be defined for vector valued functions; see  for instance \cite{Simon_1990}.
Given a real Banach separable  space $Y$, we denote by $L^q(M,Y)$ and $W^{\theta,q}(M,Y)$
the corresponding set of functions $f:M\to Y$. When no confusion arises and to ease notations,
we will  denote by $|f|_q$ and $\|f\|_{\theta,q}$ the corresponding $L^q(M,Y)$ and $W^{\theta,q}(M,Y)$
 norms.
However the Sobolev spaces $H^{\theta,q}$ can only be
 defined  for functions with values in complex vector spaces and their treatments is a more delicate.
Whenever we use these spaces in our paper we do this for $\mathbb{C}\cong\mathbb{R}^2$-valued functions.

For two separable Banach spaces $Y$ and $ Y_1$, and  a locally Lipschitz continuous function $f:
Y\to Y_1 $, we will denote by $F$ the corresponding  Nemytski map  defined by
\begin{equation}
F(\gamma ):=f\circ\gamma\;,\qquad\gamma\in W^{\theta
,q}(M,Y).
\label{nem-4}
\end{equation}

We will at first study  the  regularity properties
of the Nemytski operator $F$. Let us first note that in general
$F$ does not  map the space $W^{\theta ,p}(M,Y)$ into $W^{\theta ,q}(M,Y_1)$.
 However, it does if either the function $f$ is  globally Lipschitz
 or  if $\frac dq<\theta <1$,  see \cite{Brz+Elw_2000}.

 Given a real separable Banach space $Y$,  $\theta \in [0,1]$ and $q\in [1,\infty)$, let
\begin{equation}\label{eqn_R(Y)}
\mathcal{R}^{\theta,q}(Y)=W^{\theta,q}(M,Y)\cap L^\infty(M,Y) ,
\end{equation}
 endowed with the norm
$\|u\|_{\mathcal{R}^{\theta,q}(Y)}=\|u\|_{W^{\theta ,q}(M,Y)} + |u|_{L^\infty(M,Y)}$.
Once more to ease notations, let $\mathcal{R}^{\theta,q}=\mathcal{R}^{\theta,q}(M,\mathbb{C})$ where $\mathbb{C}
\equiv {\mathbb R}^2$.

Fix $\theta \in (0,1)$ and $q\in [1,\infty)$.  We will show    $F$
  is map   from a Banach space $\mathcal{R}^{\theta,q}(Y)$ to $\mathcal{R}^{\theta,q}(Y_1)$
if $f:Y\to Y_1$ is locally Lipschitz.
  Note  that if  $\frac dq<\theta <1$,  then by the Sobolev embedding Theorem
 $W^{\theta ,q}(M,Y)\subset L^\infty(M,Y)$ so that in that case the result from \cite{Brz+Elw_2000}
is a special case of the one below.
Let $Y$ and $Y_1$ be separable Banach spaces, $j$ be an integer larger than one
 and $f:Y\to Y_1$ be of class $C^{j-1}$ such that $f^{(j-1)}$ is Lipschitz on balls or everywhere differentiable.  Given
$R>0$ set
\begin{equation} \label{def-K}
K_j(f,R)= \sup_{\vert x \vert, \vert y \vert \leq R} \frac{\vert f^{(j-1)}(y)- f^{(j-1)}(x) \vert }
{\vert y-x\vert }, \;\;\tilde{K}_j(f,R)= \sup_{\vert x \vert \leq R} \vert f^{(j)}(x)\vert.
\end{equation}

\begin{proposition} \label{prop-Nemytski} Fix  $\theta \in (0,1]$ and $q\in [1,\infty)$.
Let  $Y$ and $Y_1$ be real separable Banach spaces and
 $f:Y\to Y_1$ be Lipschitz on balls when $\theta<1$, or, everywhere differentiable  when $\theta=1$.
Then the Nemytski map $F$ corresponding to $f$ (and  defined in (\ref{nem-4})) maps
$\mathcal{R}^{\theta,q}(Y)$ to $\mathcal{R}^{\theta,q}(Y_1)$.
More precisely, when $\theta <1$, then for all $\gamma\in  \mathcal{R}^{\theta,q}(Y)$,
\begin{equation}\label{nem-7}
\Vert F(\gamma)\Vert_{\theta,q} \leq | f(0)| \mathrm{vol}(M)
 +K_1(f, |\gamma|_{\infty}) \Vert \gamma\Vert_{\theta,q}.
 \end{equation}
 When $\theta =1$, inequality \eqref{nem-7} holds with $K_1$ being replaced by $\tilde{K}_1$.
In particular, $F$ is of linear growth either if $f$ is globally Lipschitz, i.e. $K_1(f):=\sup_{R>0}K_1(f,R)$ is finite
 when $\theta<1$, or if $f^\prime$ is bounded, i.e. $\tilde{K}_1(f):=\sup_{R>0}\tilde{K}_1(f,R)<\infty$   when $\theta=1$.  \\
The above statements remain true
if the space $\mathcal{R}^{\theta,q}(Y)$ is replaced by
$\tilde{\mathcal{R}}^{\theta,q}_{s,p}(Y)=H^{1,2}(M)\cap W^{s,p}(M)$,  provided that  $1>s >\frac{d}p$.
\end{proposition}

\begin{proof}
We at first check that we may assume that  $f(0)=0$. %case (i) reduces to case (ii).
Indeed, the function $\tilde{f}:Y\to Y_1$ defined by   $\tilde{f}(y):=f(y)-f(0)$, $y\in Y$,
 satisfies the same assumptions as $f$
and the corresponding  Nemytski operator $\tilde{F}$  is defined by
$\tilde{F}(\gamma)=\tilde{f}\circ \gamma$, i.e.
$\tilde{F}(\gamma)=F(\gamma)-f(0)1_M$.
Since $\mathrm{vol}(M)<\infty$, $1_M$ belongs to   $W^{\theta ,q}(M)$ and
  $\n F(\gamma )\n_{\theta ,q}=\n \tilde{F}(\gamma )\n_{\theta ,q}$ for each $\gamma:M\to Y$.
Hence  the operator
$F$ maps $\mathcal{R}^{\theta,q}(Y)$ into $\mathcal{R}^{\theta,q}(Y_1)$ if and only if $\tilde{F}$ does.

 Suppose that $f(0)=0$  and $\theta <1$. Let  $\gamma\in \mathcal{R}^{\theta,q}(Y)$
and  set  $R:=|\gamma|_{\infty}$.
 By
assumption $f$ is  Lipschitz on the ball $B(0,R):=\{x\in Y: \vert x\vert \leq R\}$ and
\begin{equation}\vert f(y_1)-f(y_2)\vert\leq K_1(f,R)\vert y_1-y_2\vert
,\;\;  y_1,y_2\in B(0,R).
\label{nem-8}\end{equation}

Thus, since $f(0)=0$,  $\vert f(y)\vert\leq K_1(f,R)\vert y\vert$ for $
\vert y\vert\leq R$ and
since $\vert\gamma (x)\vert\leq |\gamma|_\infty$ for a.a.  $x\in {M}$, we infer that
$| F(\gamma )|_{L^q(M,Y_1)}\leq K_1(f,|\gamma|_\infty)|\gamma|_{L^q(M,Y)}$.
Similarly  (\ref{nem-8})  yields
\[\vert f(\gamma (x_1))-f(\gamma (x_2))\vert\leq K_1(f,|\gamma|_\infty )\vert\gamma (
x_2)-\gamma (x_1)\vert ,\quad \mbox{ for a.a. }  x_1,x_2\in M, \]
which yields
$  \n F(\gamma )\n_{\theta ,q}\leq K_1(f,|\gamma|_\infty )\n\gamma\n_{\theta ,q}$.
Since the part of the result corresponding the $L^\infty$ norm is obvious,
the last inequality concludes the proof of \eqref{nem-7} if $\theta <1$.

If $\theta =1$, $\gamma \in {\mathcal R}^{1,q}(Y)$, we have $|f(\gamma(x))|=\big|\int_0^1 f'(s\gamma(x))\gamma(x) ds\big|
 \leq \tilde{K}_1(f,|\gamma|_\infty) |\gamma|_\infty$. Furthermore, $(f\circ \gamma)' = (f'\circ \gamma) \gamma'$, and since
 $|f'\circ \gamma|\leq \tilde{K}_1(f, |\gamma|_\infty)$ we deduce that
\[\int_M |\nabla(f\circ \gamma)|^q dx \leq \tilde{K}_1(f, |\gamma|_\infty) \int_M |\nabla \gamma(x)|^q dx,\]
which proves \eqref{nem-7} when $\theta=1$.
\end{proof}

Let us formulate the following important (but simple) consequence
  of Proposition \ref{prop-Nemytski}.

\begin{corollary}
\label{cor-algebra}
Let    $\theta \in (0,1]$   and $q\in [1,\infty)$.
Then   $ \mathcal{R}^{\theta,q} $
is an   algebra (with pointwise multiplication) and
  there exists a constant $C>0$ such that for $\sigma, \gamma \in W^{\theta,q}(M)\cap L^\infty(M)$,
\begin{eqnarray}
\label{eqn-multiplication-1}
\n \sigma \gamma \n_{\theta ,q}
&\leq&
\Vert \sigma \gamma \Vert_{\theta ,q}%\dela{W^{\theta,q}}}
\leq  \vert \sigma \vert_{L^\infty}
 \Vert \gamma \Vert_{\theta ,q} + \vert \gamma \vert_{L^\infty}
 \Vert \sigma \Vert_{\theta ,q}
\leq  C^\prime  \|
\sigma \|_{W^{\theta,q}\cap L^\infty} \vert \gamma \vert_{W^{\theta,q} \cap L^\infty},
\end{eqnarray}
\end{corollary}

Now we will formulate the promised generalisation of Proposition \ref{prop-Nemytski}.

\begin{theorem} \label{thm-Nemytski-Lip}
Fix $\theta\in (0,1]$ and $q\in [1,\infty)$ and let $Y$, $Y_1$ be  real separable Banach spaces.
Assume that a function $f:Y\to Y_1$  is of ${\mathcal C}^1$ class and its Fr\'echet derivative  $f':Y\to \mathcal{L}(Y,Y_1)$
is Lipschitz on balls.   Then the corresponding
Nemytski map $F$ defined in (\ref{nem-4}) is Lipschitz continuous on
balls in $\mathcal{R}^{\theta,q}(Y) $. More precisely
 for any $K>0$, and all  $\gamma ,\sigma\in \mathcal{R}^{\theta,q}(Y)$ with $\Vert\gamma\Vert_{\mathcal{R}^{\theta,q}(Y)}
\vee \Vert\sigma\Vert_{\mathcal{R}^{\theta,q}(Y)}\leq K$, we have:
\begin{eqnarray}
\vert F(\gamma )-F(\sigma )\vert_q
&\leq&  K_1(f,|\gamma|_{\infty}\vee |\sigma|_{\infty})\; \vert
\gamma -\sigma\vert_q ,
\label{nem-13b}
\\
\nonumber
\n F(\gamma )-F(\sigma )\n_{\theta ,q} &\leq& K_2(f,|\gamma|_{\infty}\vee | \sigma|_{\infty}) \;
 \vert\gamma -\sigma\vert_{\infty} \big[\n \sigma\n_{\theta ,q}
+\frac 12\n\gamma
-\sigma\n_{\theta ,q}\big]\\
&& + K_1(f,|\gamma|_{\infty}\vee |\sigma|_{\infty})\;  \n\gamma -\sigma\n_{\theta ,q}.
\label{nem-14b}
\end{eqnarray}
The above statements remain true
if the space $\mathcal{R}^{\theta,q}(Y)$ is replaced by
$\tilde{\mathcal{R}}^{\theta,q}_{s,p}(Y)=H^{1,2}(M)\cap W^{s,p}(M)$,  provided that  $1>s >\frac{d}p$.
\end{theorem}
\begin{proof}%[Proof of Theorem \ref{thm-Nemytski-Lip}.]
We only consider the case $\theta<1$ since the proof in the case $\theta=1$  is  easier.\\
Take $\gamma ,\sigma\in \mathcal{R}^{\theta,q}(Y)$  %W^{\theta ,p}(M)\cap L^\infty(M)$
and put $R=|\gamma|_{\infty} \vee |\sigma|_{\infty}$.
 For  $x\in M$ the Taylor formula yields
\begin{equation}
f(\gamma (x))-f(\sigma (x))=\int^1_0f^\prime(y_s(x))(\gamma (x)
-\sigma (x))\,ds,\label{nem-11}
\end{equation}
 where we  let
\begin{equation}
y_s(x):=\sigma (x)+s[\gamma (x)-\sigma (x)],\quad s\in [0 ,1].
\label{nem-12}
\end{equation}
\noindent As in the proof of \eqref{nem-7}, observe that since $|\gamma(x)|\vee |\sigma(x)|\leq R$ for a.a. $x$ in  $M$,
\[
\vert F(\gamma )-F(\sigma )\vert_{q} \leq K_1(f,R)
\vert \gamma -\sigma\vert_q ,\quad
| F(\gamma )-F(\sigma )|_{\infty}\leq  K_1(f,R) | \gamma -\sigma| _{\infty}.
\]
\noindent
Thus,  it suffices to check \eqref{nem-14b}.

Let $x_1,x_2\in M$ be such that $|\gamma(x_1)|,|\gamma(x_2)|, |\sigma(x_1)|,|\sigma(x_2)|\leq R$ and let $y_s(x_i)$ be defined in
(\ref{nem-12}). Then $|y_s(x_i)|\leq R$ for all $s\in [0,1]$ and
\[
\left[f(\gamma (x_1))-f(\sigma (x_1))\right]-\left[f(\gamma (x_2))
-f(\sigma (x_2))\right]= I_1(x_1,x_2)+I_2(x_1,x_2),\]
where
\begin{eqnarray*}
I_1(x_1,x_2)&=&\int^1_0\left[f^\prime(y_s(x_1))-f^\prime(y_s(x_2))\right](\gamma (x_1)-\sigma
(x_1))\,ds  , \nonumber\\
I_2(x_1,x_2)&=&\int^1_0f^\prime(y_s(x_2))\left[\gamma (x_1)-\sigma (x_1)-(\gamma (x_2)-\sigma (x_2))\right]\,ds .  %\nonumber\\
\end{eqnarray*}
Therefore, by the Minkowski inequality and \eqref{nem-1} %\eqref{nem-2} %(\ref{nem-3})
we infer that
$\n F(\gamma )-F(\sigma )\n_{\theta ,q} \leq J_1+J_2$,
 where
\[
J_1 =
\left\{\int_M\int_
M\frac {\vert I_1(x_1,x_2)\vert^q}{\vert x_1-x_2\vert^{d +\theta q}}
dx_1d\,x_2\right\}^{\frac 1q} , \quad
J_2 =  \left\{\int_M\int_M\frac {\vert I_2(x_
1,x_2)\vert^q}{\vert x_1-x_2\vert^{d+ \theta q}}dx_1dx_2\right\}^{\frac
1q}.
\]
Thus,  the local Lipschitz property of $f^\prime$ implies
\begin{align*}
\Vert f^\prime(y_s(x_1))&-f^\prime(y_s(x_2))\Vert_{{\mathcal L}(Y,Y_1)} \leq K_2(f,R)\vert y_s(x_1)-y_s(x_2)
\vert\\
&\leq K_2(f,R)\left\{\vert\sigma (x_1)-\sigma (x_2)\vert +s\vert\gamma
(x_1)-\sigma (x_1)-(\gamma (x_2)-\sigma (x_2))\vert\right\},
\end{align*}
and hence
\begin{eqnarray*}
\vert I_1(x_1,x_2)\vert&\leq&
 K_2(f,R)\vert\sigma (x_1)-\sigma (x_2)\vert
 \, \vert\gamma (x_1)-\sigma (x_1)\vert
%\label{nem-17}
\\
&& +  \frac 12K_2(f,R)\vert
(\gamma -\sigma )(x_1)-(\gamma -\sigma )(x_2)\vert\,
\vert\gamma (x_1)-\sigma (x_1)\vert .
\nonumber
\end{eqnarray*}
Since for $\phi\in W^{\theta,q}(M,Y_1)$ we have
  \[
\left\{\int_M\int_M\frac {\vert\phi (x_1)-\phi (x_2)\vert^q\,
\vert\gamma (x_1)-\sigma (x_1)\vert^q}{\vert x_1-x_2\vert^{d+ \theta q }}\,dx_1\,dx_2\right\}^{\frac 1q}\leq\vert\gamma -\sigma
\vert_{\infty}\n\phi\n_{\theta ,q},
\]
   we  infer that
\begin{equation}J_1\leq K_2(f,R)\vert\gamma
-\sigma\vert_{\infty} \Big[ \n\sigma\n_{\theta ,q}+\frac 12\n\gamma -\sigma\n_{\theta ,q}\Big].
\label{nem-18}
\end{equation}
Let us observe that
$
\vert I_2(x_1,x_2)\vert\leq K_1(f,R)\vert (\gamma -\sigma )(x_1)-(\gamma
-\sigma )(x_2)\vert$.
Therefore,
\begin{equation}J_2\leq K_1(f,R)\n\gamma -\sigma\n_{\theta ,q}.\label{nem-19}
\end{equation}
Summing up, the inequalities (\ref{nem-18}) and
(\ref{nem-19})
yield \eqref{nem-14b}, which completes the proof.
\end{proof}

\begin{corollary}
Under the assumptions of Proposition \ref{prop-Nemytski} the
map $F: \mathcal{R}^{\theta,q}(Y) % W^{\theta ,p}(M)\cap L^\infty(M)
\to \mathcal{R}^{\theta,q}(Y_1)$ is measurable.
\end{corollary}
\begin{proof} One can approximate $f$ by a sequence of functions $f_n$
which satisfy the assumptions of Theorem \ref{thm-Nemytski-Lip}. Then each
Nemytski map $F_n$ associated with $f_n$ is continuous and so Borel
measurable. On the other hand, $F_n \to F$ pointwise on $\mathcal{R}^{\theta,q}(Y)$.
\end{proof}

\begin{remark}\label{focusing} Let  $Y=\mathbb{C}\equiv {\mathbb R}^2$ and let   $f: Y\to Y$ be defined by $f(z)=C\, |z|^{2\alpha}\, z$
for some real constants $\alpha \geq \frac{1}{2}$ and  $C$.  Then $f$ is of class ${\mathcal C}^1$  and both $f$ and $f'$ are Lipschitz on balls.
Furthermore, given $\sigma \geq \frac{3}{2}$, $\theta\in (0,1]$ and $q$ such that $\theta d>q$,
 the map $\Phi : \in W^{\theta,q}({\mathbb C}) \to {\mathbb R}$
 defined by
\[ \Phi (u)=\int_M |u(x)|^{2\sigma} dx\]
is of class ${\mathcal C}^2$, and for $u, v_1, v_2\in  W^{\theta,q}({\mathbb C}$, we have
\begin{align*}
\Phi'(u)(v_1)=&\int_M 2 \sigma |u(x)|^{2(\sigma -1)}  Re(u(x) \overline{ v_1(x)}) dx\\
\Phi''(u)(v_1,v_2)=& \int_M \Big[ 4 \sigma (\sigma -1) |u(x)|^{2(\alpha -2)} Re(u(x) \overline{v_1(x)}) Re(u(x)  \overline{ v_2(x)}) \\
&+ 2\sigma |u(x)|^{2(\sigma -1)} Re( v_2(x)  \overline{v_1(x)})\Big] dx.
\end{align*}
\end{remark}

\section{Stochastic Strichartz estimates}\label{sec-Strichartz}

\subsection{Deterministic Strichartz estimates}\label{subsec-Strichartz-det}
We assume the following.

\begin{assumption}\label{ass-spaces}
\begin{trivlist}
\item[(i)] $\rH_0$ is a   separable Hilbert space and $\rE_0$ is a separable  Banach space
 such that
$\rE_0 \cap \rH_0$
 is dense in both $\rE_0$ and $\rH_0$;
\item[(ii)]
There exists a separable Hilbert space ${\mathcal H}_0$ such that    $\rH_0\subset {\mathcal H}_0$ and a ${\mathcal C}_0$   unitary
 group  $\mathbf{U}=\big(U_t\big)_{t\in \mathbb{R}}$ on $\mathcal H_0$   with  the infinitesimal generator $iA$,
where $A$ is self-adjoint in ${\mathcal H}_0$.\\
 Assume that the restriction of $-A$ to $\rH_0$, denoted also by $-A$,  is a positive operator in $\rH_0$.
\item[(iii)] There exists a positive linear  operator $-\tilde {A}$ on the space $\rE_0$ such that
 $D(A)\cap E_0 \subset D(\tilde{A})$, $D(\tilde{A})\cap \rH_0 \subset D(A)$ and $A=\tilde{A}$ on $D(A)\cap D(\tilde{A})$.
 In what follows, unless in a danger of ambiguity,  the operator $\tilde{A}$ will be denoted by $A$.
\item[(iv)]  There exists a number
 $p\in (2,\infty)$ and  a non decreasing
function  $\tilde{C}_p:(0,\infty)\to (0,\infty) $ such that $\tilde{C}_p(0^+)=0$ and,
  for every $v_0\in \rH_0$, there exists a subset $\Gamma(v_0)$ of $\mathbb{R}_+$,
 of full Lebesgue measure, such that for each $t\in \Gamma(v_0)$, $U_t(v_0)\in \rE_0$,  and
 such that for $T>0$, % each bounded interval $I\subset [0,\infty)$ with length $|I|$,
 and  every $v_0\in \rH_0$,  one has:
\begin{equation}
\label{Strichartz-hom-0}
 \Big( \int_0^T  |U_tv_0|_{\rE_0}^p dt \Big)^{\frac{1}{p}}  \leq \tilde{C}_p(T) \vert v_0\vert_{\rH_0}.
\end{equation}
\end{trivlist}
\end{assumption}

Given $f\in L^1(0,T;H^0)$,  set $U\ast f = \int_0^\cdot U_{.-r} f(r) \,dr$; then we have the following inhomogeneous
Strichartz inequalities.
\begin{lemma}\label{lem-inhomogenous-Strichartz}
Assume that Assumption \ref{ass-spaces} holds with  $p\in (2,\infty)$. Then
 for  $T>0$   we have
\begin{eqnarray} \label{Strichartz-inhom}
\Big\vert  U\ast f \Big\vert_{L^p(0,T;\rE_0)} &\leq& C \tilde{C}_p(T)\,  \vert f\vert_{L^1(0,T;\rH_0)}, \; f\in L^1(0,T;\rH_0),
\\
\label{ineq-maximal-H}
\Big\vert  U\ast f \Big\vert_{C([0,T];\rH_0)} &\leq&  C \vert f\vert_{L^1(0,T;\rH_0)},\; f\in L^1(0,T;\rH_0).
\end{eqnarray}
 where $C=\sup_{t\in [-T,T]}  |U_t|_{{\mathcal L}(H_0)} \in (0,\infty)$
%\[U\ast f = \int_0^\cdot U_{.-r} f(r) \,dr. \]
\end{lemma}
\begin{proof}
  Since by assumptions
$\mathcal{U}=\big(U_t\big)_{t\in \mathbb{R}}$ is a ${\mathcal C}_0$   %(by $\sqrt{C}>0$
group on $\rH_0$  we infer that   for $t\in [0,T]$ we have:
\[
[U \ast f](t)
=U_t\Big( \int_{0}^t U_{-r} f(r) \, dr  \Big) .
 \]
Therefore, since the norms of  $U_t$ and  $U_{-\tau}$ are bounded by  $C$,
 the Jensen inequality implies \eqref{ineq-maximal-H}. Following the argument in \cite{Burq+G+T_2004}
 it is easy to deduce  \eqref{Strichartz-inhom}; we sketch the argument for the sake of completeness.
The Minkowski inequality and \eqref{Strichartz-hom-0} imply that
\[ |U*f|_{L^p(0,T;E_0)} \leq \int_0^T \!\! |U_{.-r}f(r)|_{L^p(0,T;E_0)} dr \leq
\tilde{C}_p(T) \int_0^T \!\! |U_{-r}f(r)|_{H_0} dr
\leq C \tilde{C}_p(T) |f|_{L^1(0,T;H_0)}. \]
\end{proof}

In the next lemma we show that once Assumption \ref{ass-spaces} holds for a certain set of objects it holds for a much larger class of sets of objects.

\begin{lemma}\label{lem-spaces+maps} Assume that the spaces
$\rH_0$ and $\rE_0$,   and the operator $A$ satisfy Assumption\ref{ass-spaces} with a number $p$.
Assume that  $\hat{s}\geq 0$ and put $\rH=D((-A)^{\frac{\hat{s}}2}))$, a subspace of  $\rH_0$.
  Assume also that  $\rE \subset E_0$ is a separable Banach space such that  $\rE\supset D((-\tilde{A})^{\frac{\hat{s}}2})$.
 Then  the spaces
$\rH$  and  $\rE$, and the restriction of the operator $A$ to $\rH$ and $\rE$
satisfy Assumption \ref{ass-spaces} with the same number $p$.
\end{lemma}

This yields the following:
\begin{corollary}\label{cor-group} In the framework of Lemma \ref{lem-spaces+maps},  there exists  a non decreasing
function  ${C}_p:(0,\infty)\to (0,\infty) $ such that ${C}_p(0^+)=0$,  and such that
 for every $u_0 \in L^p(\Omega, \rH)$ and $T>0$ the trajectories of the process
$(U_tu_0 , t\in [0,T])$  belong a.s. to $C([0,T];\rE) \cap L^p(0,T;\rH)$;  moreover
\begin{eqnarray}
\label{ineq-initial-0}
\mathbb{E} \Big( \int_0^T \vert U_t u_0\vert_{\rE}^p\, dt
 +\sup_{t \in [0,T]} \vert U_t u_0\vert^p_{\rH}\Big)  &\leq &  C \big(1+C^p_p(T)\big) \mathbb{E} \vert u_0\vert^p_{\rH}.
\end{eqnarray}
\end{corollary}

\subsection{Stochastic Strichartz estimates}\label{subsec-Strichartz}

 The formulation
of the result from this section
is motivated by an abstract approach to time in-homogenous Strichartz estimates proposed by
\cite{Keel+Tao_1998} and \cite{Burq+G+T_2004}.
To ease notations, let $L^q=L^q(M)$, $W^{\sigma,q}= W^{\sigma,q}(M)$ and $H^{\sigma,q}=H^{\sigma,q}(M)$
denote the spaces of complex-valued functions defined on $M$ and which satisfy the corresponding integrability and
regularity assumptions.

\noindent Let us recall
that a   Banach space $E$ is of  martingale type $2$ %, where $p\in [1,2]$ is fixed,   iff
if there exists a constant $L=L_{2}(E)>0$ such that for every
$E$-valued  finite martingale $\{M_{n}\}_{n=0}^N$  the
following  holds:
\begin{equation} \label{eqn-M type p}
\sup_{n} \mathbb{E} | M_{n} | ^{2} \le  L
\sum_{n=0}^N \mathbb{E}  | M_{n}-M_{n-1} | ^{2},
\end{equation}
where, as  usually, we put  $M_{-1}=0$; see for instance \cite{Brz_1995,Brz_1997} and/or
 \cite{Brz+Elw_2000}.
 It is know that it is possible to extend the classical Hilbert-space valued It\^o integral to
the framework of martingale type $2$ Banach spaces; see for instance \cite{Neidhardt_1978} and \cite{Dettweiler_1990}.
\smallskip

 We make the following assumption.
 \begin{assumption}\label{ass-stochastic basis}
Let $(\Omega, \mathcal{F},\mathbb{F},\mathbb{P})$, where $\mathbb{F}=\big(\mathcal{F}_t , t\geq 0\big)$,
is a filtered probability space satisfying the usual assumptions.
 We assume that $W=\big(W(t) , t\geq 0\big)$ is an $\rK$-cylindrical Wiener process on some
real  separable Hilbert space $\rK$; see Definition 4.1 in \cite{Brz+Pesz_2001}.
 \end{assumption}

  Let us recall the definition of an accessible stopping time.
\begin{definition}\label{def_accessible-stopping-time}
 An $\mathbb{F}$-stopping time $\tau$ is called accessible if there exists an
   increasing sequence $(\tau_n)_{n\in\mathbb{N}}$ of $\mathbb{F}$-stopping
   times such that $\tau_n<\tau$ and
   $\lim_{n\rightarrow\infty}\tau_n=\tau$ a.s. Such a
   sequence $(\tau_n)_{n\in\mathbb{N}}$  will be called   an approximating sequence
   for $\tau$.
\end{definition}
Let us recall the following standard notation. Assume that $X$ is a separable Banach space,  $p\in [1,\infty)$
 and $(\Omega,\mathcal{F}, \mathbb{F},\mathbb{P})$
 is a filtered probability space.
  By ${\mathcal M}^p_{\mathrm{loc}}(\mathbb{R}_+,X)$ we denote the space of all
 $\mathbb{F}$-progressively measurable $X$-valued processes $\xi:\mathbb{R}_+ \times \Omega \to X$
for which there exists a sequence $(T_n)_{n\in {\mathbb N}}$ of bounded stopping times such that $T_n\toup \infty$,
 $\mathbb{P}$-almost surely and
$\mathbb{E}\int_0^{T_n} \vert \xi(t)\vert^p\, dt<\infty$ for every $ n\in \mathbb{N}$.

Recall that if ${\mathrm K}$ and $E$ are separable Hilbert and respectively Banach spaces,
 then the space $R({\mathrm K},E)$ of $\gamma$-radonifying operators consists of all bounded operators
 $\Lambda :{\mathrm K}\to E$ such that the series $\sum_{k=1}^\infty \gamma_k \Lambda(e_k)$
 converges  in $L^2(\Omega,E)$ for some (or any)  orthonormal basis $(e_k)_{k\in {\mathbb N}}$
of ${\mathrm K}$ and some (or any)  sequence $(\gamma_k)_{k\in {\mathbb N}}$ of i.i.d. $N(0,1)$ real random variables.
 We put
 $$
 \Vert  \Lambda \Vert_{{\mathrm R}({\mathrm K}, E)} =
\Big(\tilde{\mathbb{E}} \big\vert \sum_{j\in\mathbb{N}} \beta_j  \Lambda e_j\big\vert^2_E\Big)^{\frac12}.
 $$

 By the Kahane-Khintchin inequality
 and the It\^o-Nisio Theorem, for every Banach space $E$
 there exist a constant $C_p(E)$  such that  for every linear map $\Lambda :{\mathrm K}\to E$,
 \begin{equation}
\label{ineq-K+K}
C_p(E)^{-1} \Vert  \Lambda \Vert_{{\mathrm R}({\mathrm K}, E)} \leq
 \Big(\tilde{\mathbb{E}} \big\vert \sum_{j\in\mathbb{N}} \beta_j  \Lambda e_j\big\vert^p_E\Big)^{\frac1p}
 \leq C_p(E) \Vert  \Lambda \Vert_{{\mathrm R}({\mathrm K}, E)}.
\end{equation}
 Hence the condition of convergence in $L^2(\Omega,E)$ can be replaced by a condition of convergence in
$L^p(\Omega,E)$ for some (any) $p\in (1,\infty)$.
The space $R(H,E)$ was introduced by Neidhard in his PhD thesis \cite{Neidhardt_1978}
 and was then used  to study the existence and regularity of solutions to SPDEs % by the first named authour
 in \cite{Brz_1997} and  \cite{Brz+Pesz_2001}. Recently it has been widely used; %by Ondrej\`at,   van Neerven and coauthors;
 see for instance \cite{Ondrejat_2004}, \cite{Brz+vN_2003},  \cite{vNeerv+V+W_2007} and \cite{Bessaih_Millet}.

Furthermore,
 the Burkholder inequality holds in this framework;  see \cite{Brz_1997,Dettweiler_1985,Ondrejat_2004}.
 If $E$ is a martingale type $2$ Banach space,   for every $p\in (1,\infty)$ there exists a constant
$B_p(E)>0$  such that for
each accessible stopping time $\tau>0$ and  $R(\rK,E)$-valued progressively measurable process $\xi$,
\begin{equation}\label{ineq-Burkholder}
\mathbb{E}\sup_{0\le t\leq \tau}\Big\vert  \int_0^t \xi(s) \, dW(s) \Big\vert_E ^p\leq  B_p(E)\,
\mathbb{E}\,\Bigl( \int^\tau_0\Vert \xi (t)\Vert_{R(\rK,E)
} ^2\, dt \Bigr)^{p/2}.
\end{equation}

\begin{corollary}\label{remark-Burkholder2}
Let  $E$ be  a  martingale type $2$ Banach space and $p\in (1,\infty)$. Then
 there exists a constant $\hat B_p(E)$ depending on $E$  such that for every $T\in (0,\infty]$ and every
$L^p(0,T;E)$-valued progressively measurable process $(\Xi_s$, $s\in [0,T))$
\begin{equation}\label{ineq-Burkholder2}
\mathbb{E}\Big\vert  \int_0^T \Xi_s \, dW(s) \Big\vert_{L^p(0,T;E)}^p\leq  \hat B_p(E)\,
\mathbb{E}\,\Bigl( \int^T_0\Vert \Xi_s \Vert_{R(\rK,L^p(0,T; E))}^2\, ds \Bigr)^{p/2}.
\end{equation}
Moreover, for any $T>0$, the above inequality \eqref{ineq-Burkholder2} holds true also
for the space $L^p(0,T;E)$ and the integral over interval $(0,T)$  with the same constant $\hat B_p(E)$.
\end{corollary}
\begin{proof}
Since the space $L^p(\mathbb{R}_+;E)$ is a martingale type $2$, the above inequality holds with $T=\infty$.
The second half is a consequence of the fact that  $L^p(0,T;E)$ can be isometrically  identified
 with a closed subspace of $L^p(\mathbb{R}_+;E)$.
\end{proof}

Before we state the main result in this section, let us
 prove an  auxiliary result which has an  interest in its own.
Let us introduce the following notation.
 For an $R(\rK,H)$-valued  process $\xi $ and
  define a progressively measurable $R(\rK,L^p(0,T;E))$-valued process
 $(\Xi_r)_{r\in [0,T]}$  as follows:
 \begin{equation}
 \label{eqn-Xi}
 \Xi_r:=\big\{[0,T]\ni t\mapsto 1_{[r,T ]}(t)U_{t-r}\xi(r) \big\},\;\; r\in [0,T].
 \end{equation}

\begin{lemma}\label{lem-Strichartz-stoch} Let Assumption \ref{ass-spaces} be satisfied. Then   for each $r\in [0,T]$
\begin{equation}
   \Vert \Xi_r \Vert_{{\mathrm R}({\mathrm K}, L^p(0,T;\rE_0))}
\leq C \tilde{C}_p(T)C_p(\rE_0)C_p(\rH_0) \Vert \xi(r) \Vert_{{\mathrm R}(\rK, \rH_0)}.
\end{equation}
\end{lemma}
\begin{proof}%[Proof of Lemma \ref{lem-Strichartz-stoch}.]
Consider a  sequence  $\big(\beta_j ,  j\in\mathbb{N}\big) $ of i.i.d. $N(0,1)$
random variables defined on some auxiliary probability space $({\tilde \Omega},
 \tilde{\mathcal{F}}, \tilde{\mathbb{P}})$
and a sequence $\big(e_j , j\in\mathbb{N})$ which is an  ONB of the Hilbert space $K$.
Hence, by the  Kahane-Khintchin inequality \eqref{ineq-K+K},
 the  Fubini Theorem and  the time-homogenous  inequality \eqref{Strichartz-hom-0}, we infer that
\begin{align*}
\Vert \Xi_r & \Vert^p_{{\mathrm R}({\mathrm K}, L^p(0,T;\rE_0))}
 \leq   C^p_p(\rE_0)\tilde{\mathbb{E}}\!  \int_0^T  \!\! \Big\vert \sum_{j\in\mathbb{N}} \beta_j
\Xi_r( e_j)(t) \Big\vert^p_{\rE_0}\, dt \\
& = C^p_p(\rE_0) \tilde{\mathbb{E}} \! \int_r^T  \!\!
 \Big\vert \sum_{j\in\mathbb{N}} \beta_j U_{t-r}\xi(r)( e_j) \Big\vert^p_{\rE_0}\, dt
=C^p_p(\rE_0) \tilde{\mathbb{E}} \int_r^T   \Big\vert
 U_{t-r} \Big( \sum_{j\in\mathbb{N}} \beta_j\xi(r)( e_j)\Big) \Big\vert^p_{\rE_0}\, dt\\
& \leq  C^p_p(\rE_0)  \tilde{C}^p_p(T)\tilde{\mathbb{E}}
 \Big\vert U_{-r} \sum_{j\in\mathbb{N}} \beta_j\xi(r)( e_j) \Big\vert^p_{\rH_0}
\leq C  C^p_p(\rE_0)  \tilde{C}^p_p(T) \tilde{\mathbb{E}} \Big| \sum_j \beta_j \zeta(r) e_j|_{\rH_0}^p \\
&\leq
 C  C^p_p(\rE_0)  \tilde{C}^p_p(T) C^p_p(\rH_0)  \Vert \xi(r) \Vert_{R(\rK,\rH_0)}^p.
\end{align*}
The proof is complete.
\end{proof}

 We define local $p$-integrable martingales starting from the random time $T_0$ as follows.
\begin{definition}\label{def-shifted-filtration}
Let $T_0$ be a finite accessible $\mathbb{F}$-stopping time.  For $t\geq 0$ set ${\mathcal F}^{T_0}_t=
{\mathcal F}_{T_0+t}$ and set  $\mathbb{F}^{T_0}=\big({\mathcal F}^{T_0}_t\big)_{t\geq 0}$.
Let $ {\mathcal M}^p_{\textrm{loc}}([0,\infty),\mathbb{F}^{T_0},{\mathrm R}({\mathrm K}, \rH_0))$
denote the set of $\mathbb{F}^{T_0}$-predictable, ${\mathrm R}({\mathrm K}, \rH_0)$-valued processes $X=(X_t)_{t\geq 0}$ such that
there exists a sequence $(T_n)$ of finite accessible $\mathbb{F}^{T_0}$-stopping times such that $T_n\to \infty$ a.s. and
 $\E\int_0^{T_n} |X(t)|^p_{{\mathrm R}({\mathrm K}, \rH_0)} dt <\infty$ for every integer $n$.
 To ease notation, set ${\mathcal M}^p_{\textrm loc}([0,\infty),{\mathrm R}(K,H))=
 {\mathcal M}^p_{\textrm loc}([0,\infty),{\mathbb F},{\mathrm R}(K,H))$.
\end{definition}

The following theorem, which is the main result in this section,   extends \eqref{Strichartz-inhom}
 from the deterministic to the stochastic setting and will  be called the stochastic Strichartz   estimate. It is first stated on the time interval $[0,T]$ and  
generalizes the corresponding result of  De Bouard and Debussche from \cite{deBouard+Deb_1999,deBouard+Deb_2003}.

Informally, if $\zeta \in {\mathcal M}^p_{\textrm{loc}}([0,\infty), {\mathrm R}({\mathrm K}, H))$,
 then for $t\geq 0$  the It\^o integral
 $\int_0^t \zeta(r) dW(r)$  can be written as
\[
\int_0^t
\zeta(r)  dW(r)= \sum_{k\geq 1}
\int_0^t
\zeta(r)[e_k]\, dW_k(r),
\]
where $\{e_k\}_{k=1}^\infty$ is an ONB of $\mathrm{K}$ and $(W_k)_{k=1}^\infty$ is
a sequence of i.i.d. real-valued Wiener
processes defined for $  k\in\mathbb{N}^\ast$  and   $t\geq 0$  by
 $W_k(t):=W(t)[e_k]$

\begin{theorem}
\label{thm-Strichartz-stoch}
Assume that Assumption \ref{ass-spaces} is satisfied and $\rE_0$ is a martingale type $2$ Banach space.
For each $T>0$ set $ \hat{C}_p(T): = \tilde{C}^p_p(T)C^p_p(\rE_0)C^p_p(\rH_0)\hat{B}_p(\rE_0)$,
 where  $C_p(\rE_0)$, $C_p(\rH_0)$ and resp.  $ \hat B_p(\rE_0)$ are defined in
\eqref{ineq-K+K} and resp. \eqref{ineq-Burkholder2}. Then   for  every predictable process
$\xi \in {\mathcal M}^p_{\textrm{loc}}([0,\infty),{\mathrm R}({\mathrm K}, \rH_0))$ and every accessible stopping time $\tau$
 satisfying $\tau\leq T$
 and $\mathbb{E}\big(\int_0^\tau \|\xi(t)\|_{R(K,\rH_0)}^2 \,dt\big)^{\frac{p}2} <\infty$,  one has
\begin{equation}
\label{Strichartz-stoch}
\mathbb{E} \int_0^T \vert [J_{[0,\tau)} \xi](t)\vert_{\rE_0}^p\, dt \leq \hat{C}_p(T)
 \mathbb{E} \Big( \int_0^\tau \Vert \xi(t)\Vert^2_{{\mathrm R}({\mathrm K}, \rH_0)}\, dt \Big)^{p/2},
\end{equation}
 where   one puts
\[
[J_{[0,\tau)} \xi](t)= \int_0^t 1_{[0,\tau)}(r) U_{t-r} \xi(r)\, dW(r),\;\; t \geq 0.
\]
\end{theorem}

\begin{proof}
We use the following crucial equalities for $t\in [0,T]$:
$$
\int_0^t 1_{[0,\tau)}(r) U_{t-r} \xi(r)\, dW(r)=\int_0^T 1_{[0,\tau)}(r) \Xi_r(t) \, dW(r)
= \Big(\int_0^T 1_{[0,\tau)}(r) \Xi_r \, dW(r)\Big)(t),
$$
where $\Xi_r$ is defined by \eqref{eqn-Xi}.
Hence, with   $u=J_{[0,\tau)} \xi$, we have
\[
\int_0^T \vert u(t)\vert_{\rE_0}^p\, dt
 =\int_0^T \Big\vert\Big(  \int_0^T 1_{[0,\tau)}(r)  \Xi_r \, dW(r) \Big) (t)\Big\vert_{\rE_0}^p \,dt
=  \Big\vert \int_0^T  1_{[0,\tau)}(r)  \Xi_r \, dW(r)\Big\vert_{L^p(0,T;\rE_0)}^p.
\]
Next by Corollary \ref{remark-Burkholder2} and Lemma \ref{lem-Strichartz-stoch}, we deduce
\begin{eqnarray*}
 \mathbb{E} \int_0^T \vert J_{[0,\tau)} \xi(t)\vert_{E_0}^p\, dt    &\leq&
\hat B_p(E_0) \mathbb{E}\Big|\int_0^T 1_{[0,\tau)}(r) \Vert  \Xi_r\Vert_{R(\rK;L^p(0,T;E_0))}^2\,dr\Big|^{\frac{p}2}\\
&\leq&  \hat{B}_p(E_0) \tilde{C}^p_p(T)C^p_p(E_0)C^p_p(\rH_0)\mathbb{E}\Big[\int_0^\tau \Vert  \xi(r)\Vert_{R(\rK;\rH_0)}^2\,dr\Big]^{\frac{p}2}.
\end{eqnarray*}
This concludes the proof. % of the Theorem \ref{thm-Strichartz-stoch}.
\end{proof}
Next we  formulate a result which is related with  \eqref{Strichartz-stoch}
 as the inequality \eqref{ineq-maximal-H} is to  \eqref{Strichartz-inhom}.

\begin{proposition}\label{prop-maximal-H-stoch}
Asumme that the assumptions of Theorem \ref{thm-Strichartz-stoch} are satisfied; then there exists a constant $C_p>0$ such that for $\xi$ and $\tau$
as in Theorem \ref{thm-Strichartz-stoch} we have:
\begin{equation}
\label{ineq-maximal-H-stoch}
 \E\Big(\sup_{t\in [0,T]} \Big| [J_{[0,\tau)} \xi](t)  \Big|_{\rH_0}^p \Big)  \leq
 C_p \E\Big[ \int_{0}^{\tau}  \Vert \xi(t) \Vert_{R(K,\rH_0)}^2 d t \Big]^{\frac{p}{2}} .
\end{equation}
\end{proposition}
\begin{proof}
The proof of inequality %\eqref{majo_PsiT_s2}
\eqref{ineq-maximal-H-stoch}    is classical;  we include it for the sake of completeness.
Since   $(U_t, t\in {\mathbb R})$ is
a % unitary
${\mathcal C}_0$ group on $\rH_0$ with bounded norms on $[-T,T]$,  the Burkholder inequality yields
\begin{align*}
& \E\Big(\sup_{t\in [0,T]} \Big|\int_{0}^t 1_{[0,\tau)}(s) U_{t-s} \xi(s) dW (s) \Big|_{\rH_0}^p \Big) \leq C
\E\Big(\sup_{t\in [0,T]} \Big|\int_{0}^t  1_{[0,\tau)}(s) U_{-s} \xi(s) dW (s) \Big|_{\rH_0}^p \Big)\\
&\leq   C B_p(H_0)\E\Big[ \int_{0}^{\tau}  \Vert U_{- s} \xi(s) \Vert_{R(K,\rH_0)}^2 ds \Big]^{\frac{p}{2}}
= C_p \E\Big[ \int_{0}^{\tau}  \Vert \xi(s) \Vert_{R(K,\rH_0)}^2 \,ds \Big]^{\frac{p}{2}} .
\end{align*}
\end{proof}

 We next extend the above results replacing the starting time 0 by a random one. 
Let  $T_0$ be a finite accessible $\mathbb{F}$-stopping time, $\xi \in
{\mathcal M}^p_{\textrm{loc}}([0,\infty),\mathbb{F}^{T_0},{\mathrm R}({\mathrm K}, \rH_0))$,
$T>0$
and $\tau$ be a finite accessible $\mathbb{F}^{T_0}$ stopping time bounded from above by $T$  and  such that
${\mathbb E}\big( \int_0^\tau \|\xi(t)\|_{R(K,H_0)}^2 dt\big)^{p/2}<\infty$.  Then since the process $W^{T_0}$ defined by
$W^{T_0}(t):= W(T_0+t) - W(T_0)$, $t\geq 0$ is a $\mathbb{F}^{T_0}$ Wiener process, the operator
 $J^{T_0}_{[0,\tau)}\xi$ defined by
\begin{equation} \label{JT0}
  [J^{T_0}_{[0,\tau)} \xi](t)= \int_0^t 1_{[0,\tau)}(r) U_{t-r} \xi(r)\, dW^{T_0}(r),\;\; t \geq 0,
  \end{equation}
satisfies  the inequality \eqref{Strichartz-stoch}.
Informally, if one lets
$ u(T_0+t) = \xi(t)$ and  $[J_{[T_0,T_0+\tau)} u](t)=
 [J^{T_0}_{[0,\tau)} \xi](t)$ so that $[J_{[T_0,T_0+\tau)} u](t)=\int_{T_0}^{T_0+t} 1_{[T_0,T_0+\tau)}(s) U_{T_0+t-s} u(s)\, dW(s)$,
 Theorem \ref{thm-Strichartz-stoch} yields for $ \hat{C}_p(T): = \tilde{C}^p_p(T)C^p_p(E_0)C^p_p(\rH_0)\hat{B}_p(E_0)$
\begin{equation}
\label{Strichartz-stoch_2}
\mathbb{E} \int_{0}^T \vert J_{[T_0,T_0+\tau)} u(t)\vert_{E_0}^p\, dt \leq  \hat{C}_p(T)
  \mathbb{E} \Big( \int_{T_0}^{T_0+\tau} \Vert u(t)\Vert^2_{{\mathrm R}({\mathrm K}, \rH_0)}\, dt\Big)^{\frac{p}2} .
\end{equation}
Thus we have the following version of Theorem \ref{thm-Strichartz-stoch} using Lemma \ref{lem-spaces+maps}.
\begin{corollary}\label{cor-Strichartz-stoch-flat}
Assume that the assumptions of Lemma \ref{lem-spaces+maps} are satisfied.
Then
 for each $T>0$ there exists a constant $\hat{C}_p(T)$ such that:
 \begin{trivlist}
\item[(i)]  $\lim_{T\to 0} \hat{C}_{p}(T)=0$,
\item[(ii)]  For  every finite accessible $\mathbb{F}$-stopping time  $T_0$,
every $\mathbb{F}^{T_0}$ stopping time $\tau$ bounded by $T$ %, every $\sigma >0$
and every process
$\xi \in {\mathcal M}^p_{\textrm{loc}}([0,\infty),\mathbb{F}^{T_0},{\mathrm R}({\mathrm K}, \rH))$
 such that  $\mathbb{E}\big( \int_0^\tau \Vert \xi(t)
\Vert^2_{{\mathrm R}({\mathrm K}, H)} dt\big)^{p/2} <\infty$ % H^{\sigma + \frac{1}{p},2}(M))}\, dt<\infty$
one has
\begin{equation}  \label{Strichartz-stoch-M}
\mathbb{E} \int_{0}^{T} \vert J^{T_0}_{[0,\tau)} \xi(t)\vert_E^p\, dt \leq
 \hat{C}_p(T)    \mathbb{E} \Big( \int_{0}^\tau \Vert \xi(t)\Vert^2_{{\mathrm R}({\mathrm K}, H % ^{\sigma + \frac1p,2}(M)
)}\, dt
\Big)^{\frac{p}2}.
\end{equation}
where  $J^{T_0}_{[0,\tau)}$ be  defined by \eqref{JT0}.
\end{trivlist}
\end{corollary}

\subsection{Examples of the deterministic and the stochastic Strichartz estimates}\label{subsec-concreteSSE}

\label{rem-condition_C}
Let  $M$ be a compact Riemannian manifold $M$ of dimension $d\geq 2$.   According to Burq et all \cite{Burq+G+T_2004},   Assumption \ref{ass-spaces}
is satisfied by  the Hilbert space ${\mathcal H}_0=L^2(M)$, % where $M$ is a compact Riemannian manifold $M$ of dimension $d\geq 1$,
the ${\mathcal C}_0$-group of unitary operators $(U_t, t\in {\mathbb R})$ with  infinitesimal generator
 $iA$, where $A:=\Delta $ is the Laplace-Beltrami operator on $M$,
and the spaces $\rH_0=H^{\frac1p,2}(M)$ and $\rE_0=L^q(M)$,  provided the  parameters, $p\in [2,\infty)$ and $q\in (2,\infty)$
 satisfy the so called scaling admissible  condition
\begin{equation}
\label{eqn-scaling}
\frac2p+\frac{d}q=\frac{d}2.
\end{equation}
  Indeed, on $H_0=H^{\frac{1}{p},2}(M)$ we may consider either the norm $\|.\|_{H^{1/p,2}(M)}$ or, since $H_0=D((-\Delta)^{1/(2p)})$ the equivalent
norm $|\Delta^{1/(2p)}. |_{L^2(M)}$ for which $(U_t, t\in {\mathbb R})$ is a group of isometries.

It is well known that when $M$ is replaced by $\mathbb{R}^d$, then Assumption \ref{ass-spaces}
 is satisfied by  the ${\mathcal C}_0$  unitary group   generated by the operator $i \Delta$, and the spaces
${\mathcal H}_0=\rH_0=L^2(\mathbb{R}^d)$ and
 $\rE_0=L^q(\mathbb{R}^d)$ provided $p\in [2,\infty)$ and $q\in (2,\infty)$
 satisfy the  scaling admissible condition \eqref{eqn-scaling}.
In this setting, the identity  \eqref{eqn-scaling} is optimal with these spaces for \eqref{Strichartz-hom-0}
to hold true.

It is shown in \cite[Theorem 4]{Burq+G+T_2004} that when $M=S^2$ is the two-dimensional sphere,
 then    Assumption \ref{ass-spaces}
 is satisfied by the ${\mathcal C}_0$-group of unitary operators generated by the operator
 $i \Delta$ with the following choice parameters: $p=4$, $\rE_0=L^4(M)$
 and $\rH_0=H^{s,2}(M)$ for $s>s_0(2)=\frac18$,  which proves that \eqref{Strichartz-inhom} is not optimal for
$\rH_0=H^{\frac14,2}(M)$ and $\rE_0=L^4(M)$. Note also that \eqref{Strichartz-inhom} does not hold when $s<s_0(2)$.

The following result proves that on compact manifolds, the homogenous
  Strichartz estimates  \eqref{Strichartz-hom-0} and Lemma \ref{lem-inhomogenous-Strichartz}
 hold for the following spaces:  $\rH=H^{\sigma + \frac{1}{p},2}(M)$  and $\rE=W^{\sigma,q}(M)$, for $\sigma\geq 0$.
\begin{proposition}\label{prop-Strichartz-det}
 Let $M$ be compact Riemanian manifold, $\big(U(t)=e^{it\Delta}, t\in {\mathbb R}\big)$,
$(p,q)$ satisfy the scaling admissible condition \eqref{eqn-scaling}. Then
for each $\sigma \geq 0$ and  $T>0$ there exists a constant  $\bar{C}_q(T)>0$  such that
$\lim_{T\todown 0}\bar{C}_q(T)=0$ and
\begin{trivlist}
\item[(i)] For every $v_0\in H^{\sigma+\frac{1}{p},2}(M)$,
\begin{equation} %\label{Stri-hom-sigma}
\Big(\int_0^T \|U(t) v_0\|_{W^{\sigma ,q}}^p dt \Big)^{1/p}
\leq \bar{C}_q(T) \|v_0\|_{H^{\sigma+\frac{1}{p},2}}.
\end{equation}
\item[(ii)] For every  $g\in L^1(0,T;H^{\sigma + \frac{1}{p},2}(M))$,
\begin{equation}
\Big( \int_0^T
  \|(U\ast g)(t)\big\|_{W^{\sigma,q}}^p  dt \Big)^{1/p}
\leq \bar{C}_q(T) \int_0^T \|g(t)\|_{H^{\sigma + \frac{1}{p},2}} \,dt.
\end{equation}
\end{trivlist}
\end{proposition}
\begin{proof}
This result follows  Lemma \ref{lem-inhomogenous-Strichartz} and Lemma \ref{lem-spaces+maps}.   Set ${\mathcal H}_0=L^2(M)$.
 First let us notice that for $\sigma=0$ and
 $\rH_0=H^{\frac1p,2}(M)$ and $\rE_0=L^q(M)$ the above inequalities are satisfied by
 Lemma \ref{lem-inhomogenous-Strichartz} since in view of \cite{Burq+G+T_2004} Assumption \ref{ass-spaces} is satisfied for this choice of spaces.
Let us denote by $A_r$ the version of the operator $A$ on the space $L^r(M)$, $r\in [1,\infty)$.  Then
the space $\rH=H^{\frac1p+\sigma,2}(M)$ is equal to $D(A_2^{\frac1{2p}+\frac\sigma2})$. Moreover,
 $H^{\sigma,q}(M)=D(A_q^{\frac{\sigma}{2}})$;  since
 $q\in (2,\infty)$ we have $H^{\sigma,q}(M)\subset W^{\sigma,q}(M)=:\rE$.
 The proof  is complete.
\end{proof}

Similarly, Corollary \ref{cor-group}  has  the following particular formulation.
\begin{corollary}
 \label{prop-group} In the framework of Proposition \ref{prop-Strichartz-det}, if $\rH=H^{\sigma + \frac{1}{p},2}(M)$ and either $\rE=H^{\sigma,q}(M)$
or $\rE=W^{\sigma,q}(M)$, then  for every $u_0 \in L^p(\Omega, \rH)$ and every $T>0$ the trajectories of the process
$(U_tu_0 , t\in [0,T])$,  belong a.s. to $C([0,T];\rE) \cap L^p(0,T;\rH )$ and moreover
\begin{eqnarray}
\label{ineq-initial}
\mathbb{E} \Big( \int_0^T \vert U_t u_0\vert_{\rE}^p\, dt
 +\sup_{t \in [0,T]} \vert U_t u_0\vert^p_{\rH}\Big)  &\leq & \big(1+\bar{C}^q_p(T)\big) \mathbb{\rE} \vert u_0\vert^p_{\rH}
\end{eqnarray}
\end{corollary}

Finally note that Corollary \ref{cor-Strichartz-stoch-flat} holds for $\rH = H^{\sigma + \frac{1}{p},2}(M)$ and
$\rE= W^{\sigma,q}(M)$ for any $\sigma \geq 0$, and for $p\in [2,\infty)$, $q\in (2, \infty)$ satisfying the scaling admissible condition
 $\frac2p+\frac{d}q=\frac{d}2$.
 .

\section{Stochastic NSEs: abstract local existence result}
\label{sec-abstract-local}
The aim of this section is to  prove an abstract local existence result that
 will be used subsequently to prove the local existence  for certain nonlinear Schr\"odinger equations.
 This section is divided into five subsections.

\subsection{Asumptions and truncated equation}

We begin with a description of the main assumptions. The first one of them is just  Assumption \ref{ass-spaces}.

\begin{assumption}\label{ass-maps}
\begin{trivlist}
\item[(i)]
 Assume that the spaces % ${\mathcal H}_0$,
 $\rH_0$ and $\rE_0$  and the operator $A$ satisfy
Assumption \ref{ass-spaces} with some number $p\in (2,\infty)$.
Let  $\hat{s}\geq 0$ and put $\rH=D((-A)^{\frac{\hat{s}}2}))\subset \rH_0$.    Assume also that $\rE\subset E_0$
 is a separable Banach space such that  $\rE\supset D((-\tilde{A})^{\frac{\hat{s}}2})$. \\
\item[(ii)]
Assume that   $F$ is a locally Lipschitz map from $\rH \cap \rE$ to $\rH$ in the following sense. There exists positive constants $C$  and
 $\beta \in [ 1,p)$  such that   for all $u,v\in \rH \cap \rE $
\begin{align}
\label{ineqn-growth-F1}
\vert F(u)\vert_{\rH}\leq &\; C\, \big[ \big(1+ | u|_\rE^{\beta}\big) + \big(1+ | u|_\rE^{\beta-1}\big) \vert u\vert_{\rH}\big],
\\
\label{ineq-lip-F}
\big\vert F(u) -F(v) \big\vert_{\rH}
\leq & \;C\, \big[ 1+|u|_\rE ^{(\beta-2)^+}   +  |v|_\rE^{(\beta-2)^+} \big]  \big[ 1+|u|_{\rH}  +  |v|_{\rH} \big]
 |u-v|_{\rE}
\nonumber \\
& \; + C \,\big[ 1+|u|_\rE^{\beta-1}   +  |v|_\rE^{\beta-1} \big] |u-v|_{\rH}.
\end{align}
\item[(iii)] Assume that   $G$ is a locally Lipschitz map from $\rH \cap \rE$ to $R(K,\rH)$ in the following sense.
there exist positive constants $C$ and   $a \in [1, p/2)$
 such that   for all $u,v\in  \rH\cap \rE $
\begin{align}
\label{growth-G}
|G(u)|_{R(K,\rH)} \leq &\; C\, \big[ \big( 1+ \vert u\vert^a_{\rE}\big)+ \big( 1+ \vert u\vert^{a-1}_{\rE}\big) |u|_{\rH} \big],
\\\label{lip-G}
|G(u)-G(v)|_{R(K,\rH)} \leq &\;
C \,\big( 1+|u|_{\rE}^{a-1}   +  |v|_{\rE}^{a-1} \big) |u-v|_{\rH}
\nonumber \\
& \; + C\, \big( 1+|u|_{\rE}^{(a -2)^{+}}  +  |v|_{\rE}^{(a-2)^{+}} \big)
\big( 1+|u|_{\rH}  +  |v|_{\rH}\big)  |u-v|_{\rE} .
\end{align}
\end{trivlist}
\end{assumption}
We use the convention $x^0=1$.
Lemma \ref{lem-spaces+maps} implies that the spaces $\rH$ and $\rE$ satisfy the Assumption \ref{ass-spaces}.
 Although the above growth and local Lipschitz continuity conditions  are a bit unusual one can easily see that,
as in the more typical situations, \eqref{ineq-lip-F} implies  \eqref{ineqn-growth-F1}  and \eqref{lip-G}
implies \eqref{growth-G} if $\beta, a\geq 2$.

In this section we will  consider  the following stochastic
 It\^o nonlinear Schr{\"o}dinger Equation of the following form:
\begin{equation}\label{NLS-abstract}
idu(t) + A u(t)\, dt=F(u)\,dt +G(u)\,dW(t), \quad u(0)=u_0,
\end{equation}
where  the initial data  $u_0$ belongs to the Hilbert %Sobolev
space $\rH$.
For   $d\geq  c \geq 0$, let us denote
\begin{equation}
\label{eqn-Y_t-space}
\begin{aligned}
Y_{[c,d]}&:={\mathcal C}([c,d];\rH)\cap L^p(c,d;\rE),
\end{aligned}
\end{equation}

Obviously, $Y_{[c,d]}$ is a Banach space with  norm defined by:
\[
 \vert u\vert_{Y_{[c,d]}}^p:= \sup_{r\in [c,d]} \vert u(r)\vert_{\rH}^p+\int_c^d \vert u(r)\vert_{\rE}^p\, ds.
\]
Note that the $(Y_{[c,t]})_{t\geq c}$ is an increasing   family of Banach spaces. More precisely,  if $t>\tau>c$ and $u\in Y_{[c,t]}$,
then $u_{|[c,\tau]}\in Y_{[c,\tau]}$ and $\vert u_{|[c,\tau]}\vert_{Y_{[c,\tau]}} \leq \vert u\vert_{Y_{[c,t]}}$.
 To ease  notation, we will simply write
 $Y_t=Y_{[0,t]}$.

Let $\mathbb{M}^p(Y_{T_1},\mathbb{F}^{T_0}):= \mathbb{M}^p(Y_{[0,T_1]},\mathbb{F}^{T_0})$ denote the Banach space   of continuous  $\rH$-valued
$\mathbb{F}^{T_0}$-adapted local processes $(X_t, t\in [0,T_1])$ which  satisfy
\begin{equation} \label{LpYt}
\|X\|^p_{\mathbb{M}^p(Y_{T_1},\mathbb{F}^{T_0})} = \E\Big( \sup_{r\in [0, T_1]}
 |X(t)|_{\rH}^p + \int_0^{T_1} \|X(r)\|_{\rE}^p dr \Big)
<\infty.
\end{equation}
Similarly, let $\mathbb{M}^p_{\textrm{loc}}(Y_{[0,T_1)},\mathbb{F}^{T_0})$ denote the set of all  $\rH$-valued
$\mathbb{F}^{T_0}$-adapted and continuous local processes $(X_t, t\in [0,T_1))$ such that
$X \in  \mathbb{M}^p(Y_{\tau_n},\mathbb{F}^{T_0}) $
for any sequence of $\mathbb{F}^{T_0}$ stopping times $(\tau_n)$ approximating $T_1$.

Now we will  introduce  definitions of  local and  maximal local solutions;
 they  are modifications  of definitions  used earlier,
such as  in \cite{Brz_1997}, \cite{Brz+Elw_2000} and  \cite{Brz+Masl+S_2005}.

\begin{definition}\label{def-loc-sol-mod}
 Assume that  $T_0$   is   a  finite accessible $\mathbb{F}$-stopping time
 and   $u_0$ is a $\rH$-valued $\mathcal{F}_{T_0}$-measurable random variable.
 A \textbf{local mild solution} to equation \eqref{NLS-abstract} with initial condition $u_0$ at time $T_0$
 is a process $u$ defined as  $u(T_0+t) = X(t) $, $t\in [0, T_1)$, where\\
 %\begin{trivlist}
%\item[(i)]
(i)  $T_1$ is an accessible   $\mathbb{F}^{T_0}$ stopping time,\\
%\item[(ii)]
(ii) $X=(X(t), t\in [0, T_1))$ belongs to     $\mathbb{M}^p_{\textrm{loc}}(Y_{[0, T_1)},\mathbb{F}^{T_0})$,\\
%\item[(iii)]
(iii) for some approximating sequence $(\tau_n)$ of  $\mathbb{F}^{T_0}$ stopping times for
$T_1$,  one has,
%\end{trivlist}
  %
    \begin{align}\label{NLSlocalmild}
     X(t\wedge\tau_n)=U_{t\wedge\tau_n}{u}_0+\int_{0}^{t\wedge\tau_n}U_{t\wedge\tau_n-r}F(X(r))\,dr
    +I_{\tau_n}(G(X))( t),
  \end{align}
  for every  every $n=1, 2, \cdots$ and
$t\geq 0$,
  where $I_{\tau_n}(G(X))$ is  the process defined by
 \begin{align}\label{eqn-I_tau_n}
       I_{\tau_n}(G(X))(t) = %I_{\tau_n}(G(X(t\wedge \tau_n))) =
 \int_{0}^{\infty}1_{[0, t\wedge  \tau_n]}(r)U_{t-r}G(X(r))\,dW^{T_0}(r).
  \end{align}
A local mild solution $u=\big(u(T_0+ t)\, ,\, {0\leq t < T_1}\big)$  to problem
\eqref{NLS-abstract}
 is \textbf{pathwise unique} if for any other local mild solution
$\tilde{u}=\big(\tilde{u}(T_0+t)\, ,\, {0\leq t <\tilde{T}_1}\big)$ for this problem,
       $u(T_0+t,\omega)=\tilde{u}(T_0+t,\omega)$ for almost every $(t,\omega) \in [0,T_1\wedge\tilde{T_1})\times\Omega$.\\
 A local mild solution $u=(u(T_0+t)\, ,\, { t\in [0,T_1)})$ is called  \textbf{maximal} if for any other local mild
   solution $\tilde{u}=(\tilde{u}(T_0+t)\, ,\,  t\in [0,\tilde{T_1}))$
   satisfying $\tilde{T_1}\geq T_1$ a.s. and $\tilde{u}|_{[T_0,T_0+T_1)\times\Omega}\sim
   u$, one has $T_1=\tilde{T}_1$ a.s.
  The  $\mathbb{F}$-stopping time  $T_0+T_1$ will be called the life span of the maximal
local mild solution $u$.  Furthermore,
 a  maximal local mild solution $(u(T_0+t), \, \, { t\in [0,T_1)})$ is called global if
its lifespan is equal to $\infty$ a.s., i.e. $T_1=\infty$ a.s.
\end{definition}

The existence and uniqueness of a local maximal solution to \eqref{NLS-abstract}  will be proved in section \ref{sec-local-sol}.
We at first  prove the existence and the uniqueness of the solution when  its  norm is truncated.
Thus let   $\theta:\mathbb{R}_+\to [0,1]$  be a ${\mathcal C}^\infty_0$ non increasing function
 such that
\begin{equation}\label{eqn-theta} \inf_{x\in\mathbb{R}_+}\theta^\prime(x)\geq -1, \quad \theta(x)=1\;
\mbox{\rm  iff } x\in [0,1]\quad \mbox{\rm  and } \theta(x)=0
\; \mbox{\rm  iff } x\in [2,\infty).
\end{equation}
and for $n\geq 1$ set  $\theta_n(\cdot)=\theta(\frac{\cdot}{n})$.
Let us fix some finite  accessible $\mathbb{F}$-stopping time $T_0$, some constant $T>0$
and some accessible $\mathbb{F}^{T_0}$-stopping time $T_1$ such that $T_1\leq T$.
The rest of this section is devoted to prove existence and uniqueness of the solution $X^n$ to
 the following evolution equation for $t\in [0,T_1]$:
\begin{align}
 X^n(t)= &\, U_{t}u^n({T_0})+\int_{0}^t U_{t-r}\big[ \theta_n( \vert X^n\vert_{Y_{r}}) F(X^n(r)) \big]\, dr
\nonumber  \\
&\;   +\int_{0}^t U_{t-r}\big[ \theta_n( \vert X^n\vert_{Y_{r}}) G(X^n(r)) \big]\, dW^{T_0}(r).
\label{NLSlocal_n}
\end{align}
It is similar to that introduced by de Bouard and Debussche \cite{deBouard+Deb_1999} \cite{deBouard+Deb_2003}.
The first step consists in showing that for small $T_1$ the right handside of \eqref{NLSlocal_n} is a strict contraction.
Norm estimates of the corresponding three terms are studied in separate subsections.
We will often use the following straightforward inequalities.
\begin{lemma}\label{lem-Lip-theta}
If $h:\mathbb{R}_+\to\mathbb{R}_+$ is a non decreasing  function, then for every $x,y\in {\mathbb R}$,
\begin{equation}\label{ineq-theta}
 \theta_n(x)h(x) \leq h(2n),\quad
%\\\label{ineq-Lip-theta}
\vert \theta_n(x)-\theta_n(y)\vert \leq \frac1n |x-y| . %, \;\; \mbox{ for all } x_1,x_	2\in\mathbb{R}.
\end{equation}
\end{lemma}

\subsection{Estimates for the deterministic term}\label{subsec-ineq-deter}

The following results are modified and extended versions of the argument from \cite{Burq+G+T_2004}.
Since    $p>\beta $ by Assumption \ref{ass-maps} (i), we have
\begin{equation}\label{eqn-gamma}
\gamma:=1-\frac{\beta}{p}>0.
\end{equation}

Let us recall that the space $Y_{[T_0,T_0+T]}$
has  been defined in \eqref{eqn-Y_t-space}.
 For  $X\in Y_T:={Y}_{[0,T]}$ put
\begin{eqnarray}
\label{eqn-Phi_T-det}
[\Phi^{n}_T(X)](t)&=&\int_{0}^t U_{t-r}\big[ \theta_n( \vert X\vert_{Y_{r}}) F(X(r) \big]\, dr ,\;
 t\in [0,T].
\end{eqnarray}
The following two results are formulated and proved for $T_0=0$ but their
 generalization to any $T_0$ is straightforward since the integrals are deterministic.
\begin{lemma}\label{lem-det}
Assume that Assumption \ref{ass-spaces} is satisfied and that  the  map $F$ satisfies  Assumption  \ref{ass-maps}(ii).  Let $n>0$ and $T>0$.
Then the map $\Phi^{n}_T$ defined by \eqref{eqn-Phi_T-det}  maps the space
 $Y_{T}$ into itself.
 Moreover, there exists a generic constant $C>0$ such that for each $X\in {Y}_{T}$,
\begin{align}\Vert \Phi^{n}_T(X)\Vert_{ C([0,T];\rH)} &\leq
C\Big[ T+ \big(T^{\gamma} + T^{\gamma+\frac1p}\big)(2n)^\beta\Big],
 \label{ineq-Phi_T-stronger}\\
\label{ineq-Phi_T_2-stronger}
\Vert \Phi^{n}_T(X)\Vert_{ L^p(0,T;E)} &\leq C \tilde{C}_{p}(T) \Big[ T+ \big(T^{\gamma} + T^{\gamma+\frac1p}\big)(2n)^\beta\Big],
\end{align}
where $\tilde{C}_{p}(T)$ is the constant from Assumption \ref{ass-spaces} with $H$ and $E$ instead of $H_0$ and $E_0$ respectively.
\end{lemma}
\begin{proof}
 It is sufficient  to prove inequalities (\ref{ineq-Phi_T-stronger}-\ref{ineq-Phi_T_2-stronger}).
For this aim let  us fix  $X\in {Y}_{T}$.

%\textbf{ Step 1. Proof of inequality \eqref{ineq-Phi_T-stronger}.}
\textbf{ Step 1.} We at first prove  \eqref{ineq-Phi_T-stronger}. The inequality \eqref{ineq-maximal-H} from Lemma \ref{lem-inhomogenous-Strichartz}
 yields
\begin{eqnarray}\label{ineq-L^1-C}
\sup_{t\in [0,T]} \vert [\Phi^{n}_T(X)](t) \vert_{\rH} &\leq & \int_{0}^{T} \theta_n( \vert X\vert_{Y_{t}})
\vert  F(X(t)) \vert_{\rH}\, dt  .
\end{eqnarray}
 Thus,  it is enough to  estimate the $L^1(0,T;\rH)$-norm of %
$
\theta_n( \vert X\vert_{Y_{\cdot}}) F(u(\cdot))
$.
Let us define
$ T^\ast:=\inf\{t
\geq 0 :\vert X\vert_{Y_{t}}\geq 2n\} \wedge T
$ % \]
and note that $\theta_n(|X|_{Y_{t}})=0$ for $|X|_{Y_{t}} \geq 2n$.
Then since $\tau\to |X|_{Y_{t}}$ is non decreasing on $[0,T]$, $\theta_n(|X|_{Y_{t}})=0$ for $t \geq T^\ast$.
 Using   Assumption \eqref{ineqn-growth-F1} and   H\"older's inequality we infer that for some $C>0$:
\begin{align}
\nonumber
\int_{0}^{T}\!  & \theta_n( \vert X\vert_{Y_{t}}) \vert  F(X(t))\vert_{\rH}\, dt
 \leq
C\Big[ T^\ast+ \int_{0}^{T^\ast}\!\!
 \big[| X(t)|_{E}^{\beta} + | X(t)|_{E}^{\beta-1} \vert X(t)\vert_{\rH}\big]\, dt \Big]
\nonumber \\
&\leq C \Big[ T+ \int_{0}^{T^\ast} \!\! | X(t)|_{E}^{\beta} dt
  + \sup_{t \in [0,T^\ast]} \vert X(t)\vert_{\rH} \int_{0}^{T^\ast}\! \vert X(t)\vert_{E}^{\beta-1}  dt\Big]
\nonumber \\
 &\leq  C\Big[ T+ T^{\gamma}\Big( \int_{0}^{T^\ast}\!\! \vert X(t)\vert_{E}^p
 dt\Big)^{\frac{\beta}p}
 %\nonumber\\ &
+T^{\gamma+\frac1p}\sup_{t \in [0,T^\ast]}\vert X(t)\vert_{\rH} \Big(\int_{0}^{T^\ast}\!\! \vert X(t)\vert_{E}^p
 dt\Big)^{\frac{\beta-1}p}\Big]
 \nonumber\\
 &\leq  C\Big[T+ \big(T^{\gamma} + T^{\gamma+\frac1p}\big)(2n)^\beta\Big].
\label{ineq-L^1-H^s2}
\end{align}
The inequalities \eqref{ineq-L^1-C} and \eqref{ineq-L^1-H^s2} conclude the proof of \eqref{ineq-Phi_T-stronger}.

\noindent \textbf{Step 2. } %Proof of inequality \eqref{ineq-Phi_T_2-stronger}.}
We turn to the proof of  \eqref{ineq-Phi_T_2-stronger};
  Assumption \ref{ass-maps}(i) implies $\hat{s}+\frac1p=s$  and,  by \eqref{eqn-Phi_T-det},
%\begin{equation}\label{ineq-Phi_T-Psi_T}
$\Phi^{n}_T(X)=U\ast\big[ \theta_n(\vert X\vert_{Y_{\cdot}}) F(X(\cdot))\big]$.
%\end{equation}
 Hence   Corollary  \ref{cor-group} shows that it is enough to upper estimate   %show that
$\int_{0}^{T} \theta_n( \vert X\vert_{Y_{t}}) \vert F(X(t))\vert_{\rH}\, dt % <\infty,
$,
which has been done in  \eqref{ineq-L^1-H^s2}.
 This completes the proof.
\end{proof}

The next  result establishes  the Lipschitz properties of  $\Phi^{n}_T$ as a map acting on $Y_{T}$
with some explicit bound of the Lipschitz constant. This the main result of this subsection.

\begin{proposition}\label{prop-det-Lip} Assume that Assumption \ref{ass-spaces} is satisfied and that $F$ satisfies Assumption \ref{ass-maps}(ii). Then
the map   $\Phi^{n}_T$  defined in  \eqref{eqn-Phi_T-det}
is  Lipschitz from  the space ${Y}_T$ into itself. More precisely, for some generic constant $C>0$
 and all  $X_1,X_2\in {Y}_{T}$ we have
\begin{align*}
%\label{ineq-LipPhi-1}
\Vert \Phi^{n}_T(X_2)-\Phi^{n}_T(X_1)\Vert_{ C([0,T];\rH)} &\leq
C \, \big[ T+ nT^{\frac{p-1}{p}} + n^{\beta} (T^\gamma + T^{\gamma + 1/p})\big] |X_1-X_2|_{{Y}_{T}},\\
%\label{ineq-LipPhi-2}
\Vert \Phi^{n}_T(X_2)-\Phi^{n}_T(X_1)\Vert_{L^p(0,T;E)} &\leq
C \, \tilde{C}_{p}(T)\big[ T+ nT^{\frac{p-1}{p}} + n^{\beta} (T^\gamma + T^{\gamma + 1/p})\big] |X_1-X_2|_{Y_{T}} ,
\end{align*}
where $\tilde{C}_{p}(T)$ is  the constant from Assumption \ref{ass-spaces}. Furthermore, given any $T>0$
there exists a positive constant $L_{n}(T)$ such that  $L_n(.)$ is non decreasing,
$\lim_{T\to 0} L_{n}(T)=0 $ locally uniformly in $n$,  and such that
\begin{equation}\label{ineq-Phi_T}
\vert \Phi^{n}_T(X_2)-\Phi^{n}_T(X_1)\vert_{Y_{T}} \leq
L_{n}(T) |X_1-X_2|_{{Y}_{T}} .
\end{equation}
\end{proposition}
\begin{proof}
Let  $X_1,X_2\in {Y}_{T}$. We at first upper estimate the $C([0,T],\rH)$ and
$L^{p}(0,T;E)$-norms  of
the difference $\Phi^{n}_T(X_2)-\Phi^{n}_T(X_1)$ in terms of $A_T$ defined by
\[ %begin{equation}
%\label{eqn-proff-lem-detB-01}
A_T = \int_{0}^{T}  \big\vert \theta_n( \vert X_2\vert_{Y_{t}}) F(X_2(t))
- \theta_n( \vert X_1\vert_{Y_{t}}) F(X_1(t)) \big \vert_{\rH}\, dt .
\] %end{equation}
Indeed, arguing as in  the proof of Lemma \ref{lem-det},
we have
\begin{align} \label{eqn-LipCPhi}
\Vert \Phi^{n}_T(X_2)-\Phi^{n}_T(X_1)\Vert_{ {\mathcal C}([0,T];\rH)}  \leq A_{T},\quad
\vert \Phi^{n}_T(X_2)-\Phi^{n}_T(X_1)\vert_{L^p(0,T;E)} \leq C \tilde{C}_{p}(T) A_T,
\end{align}
where  $\tilde{C}_{p}(T)$ is the constant introduced in \eqref{Strichartz-hom-0}. For $i=1,2$ set  %the instants $T_i$ as follows:
%\begin{equation}
%\label{eqn-t_i}
$ T_i:=\inf\{t \geq 0  %\in [T_0,T_0+T]
:\vert X_i\vert_{Y_{t}}\geq 2n\} \wedge T$;
%\end{equation}
then  for $i=1,2$ we have
\[ %begin{equation}
%\label{eqn-t_i_2}
\sup_{t \in [0,T_i]} \vert X_i(t)\vert_{\rH}^p
+ \int_{0}^{T_i} \vert X_i(t)\vert_{E}^p  dt \leq (2n)^p .
\] %end{equation}
Without loss of generality we may assume that $T_1\leq T_2$.
Using once more the fact that the functions $[0,T]\ni t\mapsto |X_i|_{Y_t}$ are non decreasing,
%the properties \eqref{eqn-theta} of the function $\theta$ imply that
we deduce that
$\theta_n( \vert X_i\vert_{Y_{\tau}})=0$ for $\tau \geq T_i$, $i=1,2$, and hence
%In view of \eqref{eqn-theta} we have
\[ %begin{eqnarray*}
A_T \leq
\int_{0}^{T_2} \!\! \big\vert \theta_n( \vert X_1 \vert_{Y_{t}})-\theta_n( \vert X_2 \vert_{Y_{t}})  \big\vert
 \vert  F(X_2(t)) \vert_{\rH}  dt
 +    \int_{0}^{T_1}  \!\! \!\theta_n( \vert X_1 \vert_{Y_{t}}) \big\vert F(X_2(t)) -F(X_1(t)) \big\vert_{\rH}  dt.
\] %end{eqnarray*}
Therefore, in view of the conditions \eqref{ineqn-growth-F1} and  \eqref{ineq-lip-F}, since $2\leq \beta <p$,  H\"older's inequality
 yields the existence of a constant $C>0$ such that, for $\gamma$ defined by \eqref{eqn-gamma},  we have
\begin{align*}
 A_T &\leq     \int_{0}^{T_2}    \frac{C}{n}
  \big\vert \vert X_1\vert_{Y_{t}}-  \vert X_2\vert_{Y_{t}}  \big\vert
\big[ 1+|X_2(t)|_{E}^{\beta } + \big( 1+ |X_2(t)|_{E}^{\beta-1}\big) |X_2(t)|_{\rH}  \big]  dt
\nonumber \\
&\quad  +   C \int_{0}^{T_1} \!\!
\big( 1+ |X_1(t)|_{E}^{(\beta -2)^+} + |X_2(t)|_{E}^{(\beta -2)^+} \big)
 \vert X_1(t) -X_2(t)  \vert_{\rE} %\\
%&\qquad\qquad \times
 \big[1+ |X_1(t)|_{\rH} +  \vert X_2(t))  \vert_{\rH} \big] dt \nonumber  \\
&\quad +   C \int_{0}^{T_1} \!\!  %  1_{\{T_j \leq T_i\}}
 \big( 1+ |X_1(t)|_{\rE}^{\beta -1} +
 |X_2(t)|_{E}^{\beta -1} \big)
 \vert X_1(t)) -X_2(t)  \vert_{H} dt \nonumber \\
 &\leq  \frac{C}{n}  |X_1-X_2|_{{Y}_{T}} \Big\{
 T + T^\gamma \Big( \int_0^{T_2} |X_2(t)|_{E}^p dt \Big)^{\frac{\beta}{p}} \\
&\qquad \qquad \qquad +  \sup_{t\in [0,T_2]} |X_2(t)|_{H} \;  \Big[ T+ T^{\gamma + 1/p}\;
 \Big( \int_{0}^{T_2}    |X_2(t)|_{E}^p  dt  \Big)^{\frac{\beta-1}{p}}\Big] \Big\}
 \nonumber \\
 & \quad   + C \sup_{t \in [0, T_1]} \big( 1+ \big| X_1(t)\big|_{H} +
\big| X_2(t)\big|_{\rH} \big) \\
&\qquad\qquad \times \Big( \int_{0}^{ T_1}
 \Big[ 1+\sum_{i=1,2}  |X_i(t)|_{E}^{\frac{(\beta -2)^+ p}{p-1}}\Big] dt\Big)^{\frac{p-1}{p}}
\Big( \int_0^{T_1} |X_1(t) - X_2(t)|_{E}^p  dt \Big)^{\frac{1}{p}} \nonumber \\
 & \quad + C   \sup_{t\in [0,T_1]}  \big| X_1(t) - X_2(t)\big|_{\rH} \Big[ T+  T^\gamma \Big(
\int_0^{T_1} \sum_{i=1,2} |X_i(t)|_{E}^p dt\Big)^{\frac{\beta -1}{p}}\Big].
\end{align*}
Recall that $p>\beta $; thus
$p>\frac{p}{p-1}(\beta -2)$.
Hence, since $T_1\leq T_2$, H\"older's inequality yields
\begin{eqnarray}
 A_T &\leq &     \frac{C}{n}
|X_1-X_2|_{{Y}_{T}}    \big[ T+T^{\gamma} (2n)^\beta + (2n) \big\{ T+T^{\gamma + 1/p} (2n)^{\beta -1}\big\}\big]
\nonumber \\
&& + C (1+4n)  |X_1-X_2|_{L^p(0,T ; E)} \big( T^{\frac{p-1}{p}}  +  T^{\gamma + 1/p}(2n)^{(\beta -2)^+} \big)\nonumber \\
&& + C   \sup_{t\in [0,T_1]}  \big| X_1(t) - X_2(t)\big|_{\rH}
\big( T+  T^\gamma (2n)^{\beta -1}\big) \nonumber \\
&\leq & C \,|X_1-X_2|_{{Y}_{T}} \big[ T+ n T^{\frac{p-1}{p}} +  n^{\beta -1}  T^\gamma + n^\beta T^{\gamma+ 1/p} \big] .
\label{majoAT}
\end{eqnarray}
The inequalities \eqref{eqn-LipCPhi}-\eqref{majoAT} conclude the proof of the two first upper estimates of the difference
$\Phi^{n}_{T}(X_{2})- \Phi_{T}^{n}(X_{1})$ in the proposition.
Finally, let
\[ L_{n}(T)= C \big(1+\tilde{C}_p(T) \big)\big[ T+ nT^{\frac{p-1}{p}} + n^{\beta} (T^\gamma + T^{\gamma + 1/p})\big] .\]
Then since $\lim_{T\to 0} \tilde{C}_{p}(T)= 0$, it remains bounded for $T\in (0,1]$ and we deduce that
 $\lim_{T\to 0} L_{n}(T)= 0$ locally uniformly in $n$. This completes the proof of the proposition.
\end{proof}

\subsection{Estimates for  the stochastic term}\label{subsec-ineq-stoch}

Let us recall that according to Assumption \ref{ass-stochastic basis},  we assume that $K$ is a separable Hilbert space and $W=(W(t) ,t\geq 0)$ is a $K$-valued cylindrical Brownian motion on a filtered probability
space $(\Omega, {\mathcal F},\mathbb{F}, {\mathbb P})$.
We  fix   a finite accessible $\mathbb{F}$-stopping time $T_0$ and keep the notation introduced in Definition
\ref{def-shifted-filtration}.
Recall that
 ${\mathcal F}^{T_0}_t = {\mathcal F}_{T_0+t}$
for $t\geq 0$ and $(W^{T_0}(t), t\geq 0)$ is the $\mathbb{F}^{T_0}$ Brownian motion defined by $W^{T_0}(t)=W(T_0+t)-
W(T_0)$. Finally $T$ will  be some positive constant and $T_1$ a $\mathbb{F}^{T_0}$
accessible stopping time such that
$T_1\leq T$.  Recall that the space  $\mathbb{M}^p(Y_{T_1}, {\mathbb F}^{T_{0}})$ has been defined in \eqref{LpYt}.
As usual, we let $a^0=1$ for any $a\geq 0$.

Let $X\in {\mathbb M}^{p}(Y_{T}, {\mathbb F}^{T_{0}})$.
Then   for every $n\geq 1$   we set
\[ \xi^n(t)=\theta_n(|X|_{Y_{t}}) G(X(t)), \; t\in [0,T], \]
 and  put  $ \Psi^{T_0,n}_T(X)=J^{T_0}\xi^n$ with  $J^{T_0}$ defined by
 \begin{equation} \label{eqn-Psi_T}
 \big[J^{T_0}  \xi^n\big] (t)=
 \int_{0}^t U_{t-r} \xi^n(r) \, dW^{T_0}(r),
\quad t\in [0,T].
\end{equation}

We at first prove that $\Psi^{T_0,n}_T$ maps $\mathbb{M}^p(Y_{T},\mathbb{F}^{T_0})$ into itself.
More precisely, we have the following result.
\begin{lemma}  \label{PsiT_on_Lp} Assume that Assumption \ref{ass-spaces} is satisfied, $\rE$ is a martingale type $2$ Banach space  and that
 the map $G$ satisfies Assumption
\ref{ass-maps}(iii).
Then  $\Psi^{T_0,n}_T$  maps $\mathbb{M}^p(Y_{T},\mathbb{F}^{T_0})$ into itself.
Moreover, one can find a constant $C_p>0$ such that for every  $X\in \mathbb{M}^p(Y_{T},\mathbb{F}^{T_0})$
and the constant $\hat{C}_{p}(T)$ defined in \eqref{Strichartz-stoch-M}, we have:
\begin{align}
&\E \int_{0}^{T} \vert\Psi^{T_0,n}_T(X)(t)\vert_{E}^p  dt \leq C_p\, \hat{C}_{p}(T)\Big[ T^{p/2} +T^{p/2-a}  \big( 1 + T) n^{pa}
+ T^{p/2}  n^p   \Big] , \label{majo_PsiT_shatq} \\
&\E\Big(\sup_{t\in [0,T]} |\Psi^{T_0,n}_T(X)(t)|_{\rH}^p \Big)   \leq C_p \Big[  T^{p/2} +T^{p/2-a}  \big( 1 + T) n^{pa}
+ T^{p/2}  n^p     \Big] . \label{majo_PsiT_s2}
\end{align}
\end{lemma}

To ease notation, in the proof below, as well as in the other proofs in this section,
 we will omit the subscript ${T_0}$. For instance we will simply write
$\mathbb{F}$ and $J$ instead of $\mathbb{F}^{T_0}$ and $J^{T_0}$.
\begin{proof} Let us take
 $X\in \mathbb{M}^p(Y_{T},\mathbb{F})$; % it is enough to prove inequalities
 %(\ref{majo_PsiT_shatq}) and (\ref{majo_PsiT_s2}).
first we will   prove  \eqref{majo_PsiT_shatq}.
 Using inequality \eqref{Strichartz-stoch-M} from  Corollary \ref{cor-Strichartz-stoch-flat}  with $\tau =T$ we  deduce that
\[ \E \int_{0}^{T} \vert J \xi^n(t)\vert_{E}^p  dt \leq \hat C_p(T)
 \E \Big(\int_{0}^{T} \Vert\xi^n(t)\Vert_{R(K,\rH)}^2 \,dt\Big)^{p/2}.\]
Let
$T^\ast:=\inf\{ t\geq 0 : |X|_{Y_{t}} \geq 2n\}\wedge T$.
Then   $\theta_n(|X|_{Y_t})=0$  for $ t\in [T^\ast, T]$ and
 \[  \sup_{t\in [0,T^\ast]} |X(t)|_{\rH}^p + \int_{0}^{T^\ast}  \vert X(t)\vert_{E}^p dt
 \leq (2n)^p.
\]
 Hence the growth condition \eqref{growth-G} on $G$ and H\"older's inequality imply  that
   \begin{align} \label{majoxiRadon}
&\int_{0}^{T} \Vert\xi^n(t)\Vert_{R(K,\rH)}^2 \,dt \leq  C  \int_{0}^{T^\ast} \Big[ \big( 1+ \vert X(t)\vert^{2a}_{E}\big)+
 \big( 1+ \vert X(t)\vert^{2a-2}_{E}\big) |X(t)|^2_{\rH} \Big] \,
dt \nonumber \\
\nonumber
&\leq C\Big[ T + T^{1-\frac{2a}p} \Big( \int_{0}^{T^\ast}\!\!
 \vert X(t)\vert^{p}_{E} \, dt\Big)^{\frac{2 a}p}+  \sup_{t \in [0,T^\ast] } |X(t)|^2_{\rH}
\Big\{ T + T^{1-\frac{2a-2}p} \Big( \int_{0}^{T^\ast} \!\!\vert X(t)\vert^{p}_{E} \, dt\Big)^{\frac{2 a-2}p} \Big\}  \Big]
\\
&\leq C\Big[ T + T^{1-\frac{2a}p} (2n)^{2 a}+  (2n)^2\Big(T + T^{1-\frac{2a-2}p} (2n)^{2 a-2} \Big)  \Big].
\end{align}
This completes the proof of \eqref{majo_PsiT_shatq}.
To  prove  \eqref{majo_PsiT_s2} we apply inequality \eqref{ineq-maximal-H-stoch} to get
\begin{equation*}
 \E\Big(\sup_{t\in [0,T]} \Big|\int_{0}^t U_{t-r } \xi^n(r) dW (r) \Big|_{\rH}^p \Big)
\leq  C_p \E\Big[ \int_{0}^{T}  \Vert \xi^n(r) \Vert_{R(K,\rH)}^2 dr \Big]^{\frac{p}{2}} .
\end{equation*}
Combining the above with  inequality \eqref{majoxiRadon} we deduce \eqref{majo_PsiT_s2}. This completes the proof of the proposition.
\end{proof}

We next  result proves that  the map $\Psi^{T_0,n}_T$ is Lipschitz on $\mathbb{M}^p(Y_{T},\mathbb{F}^{T_0})$
and gives an upper bound of its Lipschitz constant.

\begin{proposition} \label{prop-Lip-Psi_T} Assume that Assumption \ref{ass-spaces} is satisfied and that
 the map $G$ satisfies Assumption \ref{ass-maps}(iii). Then for every $T>0$
there exists a constant $\hat{L}_n(T)>0$ such that $\hat{L}_n(.)$ is non decreasing, $\lim_{T\to 0} \hat{L}_n(T)=0$ locally uniformly in $n$, and
for  $X_1,X_2\in \mathbb{M}^p(Y_{T},\mathbb{F}^{T_0})$
and the constant $\hat{C}_{p}(T)$ defined in \eqref{Strichartz-stoch-M},
\begin{equation}
\Vert\Psi^{T_0,n}_T(X_2)-\Psi^{T_0,n}_T(X_1)\Vert_{\mathbb{M}^p(Y_{T},\mathbb{F}^{T_0})}
 \leq \big(1+\hat C_p(T)\big) \hat{L}_n(T) \, \Vert X_1-X_2\Vert_{\mathbb{M}^p(Y_{T},\mathbb{F}^{T_0})}.
\end{equation}
\end{proposition}
\begin{proof}
For $i=1,2$ set $\xi^n_i(t)=\theta_n(|X_i|_{Y_{t}}) G(X_i(t))$.
Using once again inequality \eqref{Strichartz-stoch-M} from  Corollary \ref{cor-Strichartz-stoch-flat},  we deduce that
\begin{equation} \label{Stri_BT}
\E\int_{0}^{T} \Big\vert \int_{0}^t U_{t-r}
\big[ J \xi_1^n(r) - J \xi_2^n(r) \big] dr\Big\vert_{E}^p dt
\leq  \hat C_p(T) B_T,
\end{equation}
where $ B_T=  \E \big( \int_{0}^{T} |\xi^n_2(t) - \xi^n_1(t)|_{R(K,\rH)}^2 dt\big)^{\frac{p}2}$.
Furthermore, the Burkholder-Davis-Gundy inequality \eqref{ineq-Burkholder2}  yields %the existence of a constant $\hat{B}_p$
% such that
\begin{equation} \label{BDG_BT}
\E\Big( \sup_{t\in [0,T]} \vert J \xi^n_1(t)- J\xi_2^n(t) \vert_{\rH}^p \Big)
\leq \hat{B}_p(H) B_T.
\end{equation}
 For $i=1,2$ let    $T_i=\inf\{ t\geq 0 : |X_{i}|_{Y_{t}} \geq 2n\}\wedge T$.
Then one has %\eqref{ineq-Lip-theta} and the Lipschitz condition  \eqref{lip-G} on the diffusion coefficient imply
\begin{eqnarray*}  %\label{majo1_BT}
B_T
&\leq &   C_p  \sum_{\{i,j\} = \{1,2\}}\!\!
\mathbb{E} \Big( \int_{0}^{T} \!\! \big\vert \theta_n( \vert X_i \vert_{Y_{t}})-\theta_n( \vert X_j \vert_{Y_{t}})  \big\vert^2  \;
 \Vert  G(X_i(t)) \Vert_{R(K,\rH)}^p \, 1_{\{ T_j\leq T_i\}} \; dt\Big)^{\frac{p}2}
\nonumber
\\
&& + C_p \sum_{\{i,j\} = \{1,2\}}  \mathbb{E} \Big(  \int_{0}^{T}   \theta^2_n( \vert X_j \vert_{Y_{t}})
\Vert G(X_i(t)) -G(X_j(t)) \Vert_{R(K,\rH)}^2\,
 1_{\{ T_j\leq T_i\}} \; dt  \Big)^{\frac{p}2} .
\end{eqnarray*}
Furthermore,
$\theta_n( \vert X_i\vert_{Y_{t}})=0$ provided that $t \geq T_i$, $i=1,2$. Therefore, using the property
\eqref{ineq-theta} on $\theta$, the growth and  Lipschitz conditions \eqref{growth-G} and \eqref{lip-G} on
$G$, we deduce that for some constant $C>0$ %such that $C_p(T)\to 0$ as $T\to 0$,
\begin{equation} \label{majo_BT}
 B_T \leq  C   \sum_{\{i,j\} = \{1,2\}} \sum_{l=1,2,3}  \big( B_T^l(i,j)\big)^{p/2},
\end{equation}
where
  \begin{align*}
% \label{majo_BT}
B^1_T(i,j)= & \frac1{n^2} \E\Big( \int_{0}^{T_i}  \!  1_{\{ T_j\leq T_i\}}
  \big\vert \vert X_i\vert_{Y_{t}}-  \vert X_j\vert_{Y_{t}}  \big\vert^2\\
  &\hspace{2truecm}
  \times
 \big[ \big( 1+ \vert X_i(t)\vert^{2a}_{E}\big)+ \big( 1+ \vert X_i(t)\vert^{2a-2}_{E}\big)
|X_i(t)|^2_{\rH} \big]  dt \Big),
 \\
B_T^2(i,j)= &   \E  \Big( \int_{0}^{T_i\wedge T_j} \big( 1+|X_i(t)|_{E}^{2a-2}
  +  |X_j(t)|_{E}^{2a-2} \big) |X_i(t)-X_j(t)|_{\rH}^2
  dt\Big),   \\
B^3_T(i,j)=&   \E  \Big( \int_{0}^{T_i\wedge T_j}
\big( 1+|X_i(t)|_{E}^{2(a -2)^+}  +  |X_j(t)|_{E}^{2(a-2)^+} \big)
\big( 1+|X_i(t)|_{\rH}^2  +  |X_j(t)|_{\rH}^2\big)\\
&\hspace{2truecm}
  \times  |X_i(t)-X_j(t)|_{E}^2
 dt \Big).
 \nonumber
  \end{align*}
%We at first upper estimate $B_T^1(i,j)$ u
Using  H\"older's inequality since $a< p/2$ we obtain
 \begin{align*}
 n^2 B_T^1(i,j)&
 \leq   \big\vert  X_1-  X_2  \big\vert_{Y_{T}}^2 \int_{0}^{T_i}
 \big[ \big( 1+ \vert X_i(t )\vert^{2a}_{E}\big)+
\big( 1+ \vert X_i(t)\vert^{2a-2}_{E}\big) |X_i(t)|^2_{\rH} \big]  dt\\
 &\leq  C
 \big\vert  X_1-  X_2  \big\vert_{Y_{T}}^2
\Big[ T + T^{1-\frac{2a}p} n^{2 a}+  n^2\Big(T + T^{1- \frac{2a-2}p} n^{2 a-2} \Big)  \Big].
 \end{align*}
Using once more H\"older's inequality we deduce
\begin{align*}
B_T^2(i,j) &\leq
    \sup_{t \in [0,T]}|X_i(t)-X_j(t)|_{\rH}^2 \int_{0}^{T_i\wedge T_j}\!\!
 \big( 1+|X_i(t)|_{E}^{2 a-2}   +  |X_j(t)|_{E}^{2 a-2} \big)  dt  \\
    &\leq  \sup_{t  \in [0,T]}|X_i(t)-X_j(t)|_{\rH}^2 \Big(T + 2 T^{1-\frac{2 a-2}p} n^{2 a-2} \Big),
\end{align*}
and
 \begin{align*}
 B_T^3(i,j)& \leq
 C  \sup_{t  \in [0,T_i\wedge T_j]} \big( 1+|X_i(t)|_{\rH}^2  +  |X_j(t)|_{\rH}^2\big)
\Big[\int_{0}^{T}   |X_i(t)-X_j(t)|_{E}^p \, dt \Big]^{\frac2p}\\
&\qquad \times  \Big[T^{1-\frac{2}{p}}  + T^{1-\frac{a}{p} } \Big(\int_{0}^{T_i\wedge T_j}
 \big( |X_i(t)|_{E}^p + |X_j(t)|_{E}^p\big) \,dt \Big)^{\frac{(a-2)^+}p}\Big] \\
 &\leq \Big(1+8n^2\Big)\Big(T^{1-\frac{2}p} + 2T^{1-\frac{2( a-1)}p} n^{2(a -2)^+}\Big)
 \Big[\int_{0}^{T}   |X_i(t)-X_j(t)|_{E}^p \, dt \Big]^{\frac2p}.
\end{align*}
Therefore, $B_T \leq  \hat{L}_n(T) \Vert X_1-X_2\Vert^p_{{\mathbb{M}^p(Y_{T},\mathbb{F}^{T_0}})}$
 if one lets
\[ \hat{L}_n(T):=  C\big[ T^{p/2} + T^{p/2-1} n^p +T^{p/2-a} (1+T) n^{(a-1)p} + n^{(a\vee 2) p/2} T^{p/2-a+1}  \big]   .\]
 Since $a<p/2$ and $p>2$ we have $\lim_{T\to 0} \hat{L}_n(T)=0$ locally uniformly in $n$; thus
 the inequalities \eqref{Stri_BT} and \eqref{BDG_BT} conclude the proof.
\end{proof}

\subsection{The estimates on the free term}\label{subsec-free term}

In order to prove the existence and uniqueness of a solution to the approximating equation \eqref{NLSlocal_n}
we have also to show that the first term on the RHS of that equation belongs to the space
$\mathbb{M}^p(Y_{T}, {\mathbb F}^{T_0})$. This in fact follows from Corollary \ref{cor-group} and Lemma \ref{lem-spaces+maps}. Thus we have.
\begin{lemma} \label{lem-initial}
Let $T_0$ be a finite accessible $\mathbb{F}$-stopping time and  $u(T_0) \in L^p(\Omega,\mathcal{F}_{T_0},\rH)$
 and $d>c>0$. Define a  process
${\mathcal U}_{[c,d]}(u_0)$  by $\big( [{\mathcal U}_{[c,d]}(u(T_0))](t)=U(t-c)(u(T_0))\, ,\, t\in [c,d]\big)$.
Then the process $X$ defined by $X(t)={\mathcal U}_{[c,d]}(u(T_0))(t+c)$, $t\in [0,d-c]$,
 belongs to  $\mathbb{M}^p(Y_{d-c},\mathbb{F}^{T_0})$.
\end{lemma}

\subsection{Existence and uniqueness of a global solutions to approximating equations}\label{subsec-global approx}
A direct consequence of all the results proved in sections \ref{subsec-ineq-deter}-\ref{subsec-free term}
(which follows by applying the Banach-Cacciopoli
 Fixed Point Theorem) is presented below.
 \begin{theorem}\label{thm-exis-approx}
Assume that the  Assumptions \ref{ass-stochastic basis}, \ref{ass-spaces} and  \ref{ass-maps} satisfied with  $ \beta \in [2,p)$
and $a\in [1,\frac{p}2)$.
 Assume also that $\rE$ is a martingale type $2$ Banach space. Let   $T_0$ be a   finite  and accessible
 $\mathbb{F}$-stopping time  $T_0$ and   $u(T_0)\in L^p(\Omega,\mathcal{F}_{T_0},\rH)$.
 Then   for every positive integer $n$,
there exists a unique process $u^n=(u^n(t) \, , t\, \in [T_0,\infty))$, such that $u^n(t+T_0)=X^n(t)$
for $t\geq 0$ and for every  $T>0$, $X^n$ belongs to the space $\mathbb{M}^p(Y_{T},\mathbb{F}^{T_0})$ and
  $X^n={\mathcal U}_{[0,T]}(u(T_0))+\Phi^{n}_{T}(X^n)+\Psi^{T_0,n}_{T}(X^n)$.
    Moreover, given any $T>0$ the process $X^n$  is the unique solution to the  evolution equation
 \eqref{NLSlocal_n} on the  time interval $[0,T]$.
 Moreover, if a local process $v=\big(v(T_0+t)=\tilde{X}(t)\, ,\, t\in [0,\tau)\big)$
 is a local solution to \eqref{NLSlocal_n},  then the processes $X^n$ and
$\bar{X}:=\tilde{X}_{\big\vert [0,\tau)\times \Omega}$ are equivalent.
\end{theorem}
\begin{proof}
The second statement is obvious in view of the definitions of the maps ${\mathcal U}_{[0,T]}$, $\Phi^{n}_{T}$ and $\Psi^{T_0,n}_{T}$.
 To prove the first part let us fix $T>0$.
It follows from Propositions \ref{prop-det-Lip} and \ref{prop-Lip-Psi_T} that for every   $t >0$, the map
\begin{equation}\label{eqn-Lambda}
\Lambda^{n}_{t}:\mathbb{M}^p(Y_{t},\mathbb{F}^{T_0})\ni X
\mapsto {\mathcal U}_{[0,t]}(u(T_0)) +\Phi^{n}_{t}(X)+\Psi^{T_0,n}_{t}(X)\in
   \mathbb{M}^p(Y_{t},\mathbb{F}^{T_0})
\end{equation}
is well defined and Lipschitz. Moreover, if $t$ is sufficiently small
 (depending on $n$),
this map is a strict contraction. Thus, there exists \label{page-delta} $\delta>0$ such that
 if $t \leq \delta$ then $\Lambda^{n}_{t}$ is a $\frac12$-contraction in  the space
 $\mathbb{M}^p(Y_{t},{\mathbb F}^{T_0})$.  Let us define a sequence  $(T_k)_{k=0}^\infty$
 by   $T_k=T_0+k\delta$, $k\in\mathbb{N}^\ast$.
By the previous conclusion there exists
a process $X^{(n,1)}\in \mathbb{M}^p(Y_{\tau}, {\mathbb F}^{T_0})$ such that
 $X^{(n,1)}={\mathcal U}_{[0,\delta]}(u(T_0)) +\Phi^{n}_{\delta}(X^{(n,1)})
+\Psi^{T_0,n}_{\delta}(X^{(n,1)})$. By the definition of the space $ \mathbb{M}^p(Y_{\delta},\mathbb{F}^{T_0})$,
 $X^{(n,1)}(\delta)$ belongs to the space $L^p(\Omega,\mathcal{F}_{T_0+\delta},\rH)$. By Propositions \ref{prop-det-Lip}
 and \ref{prop-Lip-Psi_T} we can find  a  unique process $X^{(n,2)}\in \mathbb{M}^p(Y_{\delta},\mathbb{F}^{T_1})$ such that
 $X^{(n,2)}={\mathcal U}_{[0,\delta]}(X^{(n,1)}(T_1)) +\Phi^{n}_{\delta}(X^{(n,2)})+\Psi^{T_1,n}_{\delta}(X^{(n,2)})$.
 By induction, we can find a sequence
 $(X^{(n,k)})_{k=1}^\infty$, such that $X^{(n,k)}\in \mathbb{M}^p(Y_{\delta},\mathbb{F}^{T_{k-1}})$ and
 $X^{(n,k)}={\mathcal U}_{[0,\delta)}(X^{(n,k-1)}(T_{k-1})) +\Phi^{n}_{\delta}(X^{(n,k)})+\Psi^{T_{k-1},n}_{\delta}(X^{(n,k)})$.
 Next, we define a process $u^n$ as follows: let  $u^n(T_0+t)=X^{(n,1)}(t)$, $t\in [ 0,\delta)$,
and
 for  $k= [\frac{t-T_0}{\delta} ]+1$ and  $0\leq t< \delta$, let  $u^n(T_k+t)=X^{(n,k)}(t)$.
The proof is concluded by observing that for every $T>0$, $u^n(T_0+t)=X^n(t)$
for some  $X^n\in \mathbb{M}^p(Y_{T}, {\mathbb F}^{T_0})$ and
 $X^n={\mathcal U}_{[0,T]}(u_0)+\Phi^{n}_{T}(X^n)+\Psi^{T_0,n}_{T}(X^n)$.

Finally,to prove the uniqueness,
let us choose $\delta>0$ as above % in the proof of Proposition \ref{prop-fixed point}
and put $\sigma_1=\tau \wedge \delta$.
Then the fixed point Theorem used above %in the proof of Proposition \ref{prop-fixed point}
implies that  the processes $X_{\big\vert [0,\sigma_1)\times \Omega}$  and $\tilde{X}_{\big\vert [0,\sigma_1)\times \Omega}$ are equivalent.
 By an inductive argument, if $k\in\mathbb{N}^\ast$ and $\sigma_k=\tau \wedge (k\delta)$,
then the processes $X^n_{\big\vert [0,\sigma_k)\times \Omega}$
 and $\tilde{X}_{\big\vert [0,\sigma_k)\times \Omega}$ are equivalent. Since $\sigma_k \toup \tau$, the result follows.

\end{proof}

\section{Existence of a local  maximal mild solution to problem \eqref{NLS-abstract}}
\label{sec-local-sol}
Our first aim is to prove the following result about the existence and uniqueness of a local maximal solution to problem to \eqref{NLS-abstract}.
Let $T_0$ be a finite accessible $\mathbb{F}$-stopping time
 and  suppose that the assumption \ref{ass-maps}  on  $F$ and $G$  are satisfied.
Let $u(T_0)$ be a $p$-integrable  $\rH$-valued $\mathcal{F}_{T_0}$ random variable.
In the previous section (see Theorem \ref{thm-exis-approx})  we proved  for every $n\in \mathbb{N}$ the existence of a unique solution
$u^n(T_0+.)=X^n(.)$ on $[0,\infty)$    to the problem  \eqref{NLSlocal_n}.
Let   $\tau_n$ and $\hat{\tau}_n$  denote the ${\mathbb F}^{T_0}$ stopping times defined by
\begin{equation}
 \label{eqn-tau_n}
 \tau_{n}= \inf \, \{ t>0: |X^{n}|_{Y_{t}}\geq n\}\wedge n,\qquad
 \hat{\tau}_{n}= \inf \, \{ t>0: |X^{n}|_{Y_{t}}\geq 2n\}\wedge n,
\end{equation}
The following result establishes the existence and uniqueness  of a local solution to \eqref{NLSlocal_n}.
\begin{proposition}\label{prop-mild} Let  $(X^n(t), t\geq 0)$ be the process introduced  in Theorem \ref{thm-exis-approx}.
Then the process $(X^n(t)\, ,\, t <  \tau_n)$
is a local mild solution to problem \eqref{NLS-abstract}   with the filtration is ${\mathbb F}^{T_0}$ and the Brownian Motion  $W^{T_0}$.
\end{proposition}

\begin{proof}
Obviously $\tau_n$ is an accessible ${\mathbb F}^{T_0}$ stopping time and the process $X^n$
 satisfies
for all $t\geq 0$, $\mathbb{P}$-a.s., %
\begin{equation} \label{NLSlocal_n2}
 X^n(t) -  \, U_t[ u(T_0)] -\int_{0}^t U_{t-r}\big[ \theta_n( \vert X^n\vert_{Y_{r}}) F(X^n(r)) \big]\, dr
 = I(t),
\end{equation}
where $I(t)=\int_{0}^t U_{t-r}\big[ \theta_n( \vert X^n\vert_{Y_{r}}) G(X^n(r)) \big]\, dW^{T_0}(r)$.
Since the processes on both sides of the above equality are continuous, the equality  still holds  when the fixed deterministic
 time $t$ is replaced by the random one $t\wedge \tau_n$. The definitions of
 $\theta_n$ and  $\tau_n$ imply   $\theta_n( \vert X^n\vert_{Y_{r\wedge \tau_n}})= 1$,
 and hence the  deterministic convolution above stopped at $t\wedge \tau_n$ is equal to
 $\int_{0}^t U_{t-r} 1_{[0,\tau_n)}(r) F(X^n(r)) \, dr$.
Moreover by \cite[\S 4.3]{Carroll_1999}  and \cite[Lemma A.1]{Brz+Masl+S_2005} we infer
\begin{align}
I(t\wedge \tau_n)&=
\int_{0}^t U_{t-r}\big[ 1_{[0,\tau_n)}(r) \theta_n( \vert X^n\vert_{Y_{r\wedge \tau_n}})
 G(X^n(r\wedge \tau_n)) \big]\, dW^{T_0}(r).
\label{eqn_conv_stopped}
\end{align}
Finally, as above
$1_{[0,\tau_n)}(s) G(X^n(s\wedge \tau_n))=G (X^n(s)) $, so that
%\[
%\int_{0}^t U_{t-r} 1_{[0,\tau_n)}(r) G(X^n(r)) \, dW^{T_0}(r) = I_{\tau_n}(G(X^n))( t).
%\]
$I(t)= I_{\tau_n}(G(X^n))( t)$ where  $I_{\tau_n}(G(X^n))( t)$ is defined in \eqref{eqn-I_tau_n}.
This concludes the proof.
\end{proof}

The  previous existence result
will be supplemented  by   the following
local uniqueness result in which we will use the notation introduced in \cite{Elw_1982} (see Theorem VI.5).
%compare with Theorem VI.5 in  \cite{Elw_1982} and   Lemma 4.11 in \cite{Brz_1997}.

\begin{lemma}\label{Lem:loc-uniq}
Suppose that  the assumptions of Proposition  \ref{prop-mild} are satisfied.
Assume that $k,n$ are two natural numbers such that $n\leq k$.
 Then
$
\tau_n \leq \tau_k$  a.s.
and the processes
$u^n|_{[T_0,T_0+\tau_n)\times\Omega}$ and $ u^k|_{[T_0,T_0+\tau_n)\times\Omega}$
are equivalent.
\end{lemma}

\begin{proof} %[Proof  of Lemma \ref{Lem:loc-uniq}.]
Let us fix $k,n\in\mathbb{N}$ such that $n\leq k$ and put
\[
 %\label{eqn-tau_nk}
 \tau_{k,n}= \inf \, \{ t>0: |u^{k}|_{Y_{[0,t]}}\geq n\} \wedge k.
\]
Then obviously $\tau_{k,n} \leq \tau_{k}$ and  by repeating the argument from the proof of Proposition \ref{prop-mild}
the process  $\big(u^k(t)\, ,\, t < T_0+ \tau_{k,n}\big)$  is a local solution of the problem \eqref{NLSlocal_n2}.
 But by Proposition \ref{prop-mild}, the process   $\big(u^n(t)\, , \, t\leq T_0+ \tau_{n}\big)$
is also a local solution of the problem \eqref{NLSlocal_n}.
The proof is thus completed by applying the uniqueness part of Theorem \ref{thm-exis-approx}.
\end{proof}

In the following Theorem we will prove   the existence of a unique local mild solution to \eqref{NLS-abstract}.
An important   feature  of this result is that we can estimate from below the length of the existence time
interval with  a lower bound, which depends on the $p$-th moment of the $\rH$-norm of the initial data, on a "large" subset of
 $\Omega$ whose probability does not
depend on this moment.
 This property is used later on in proving that the $\rH$-norm of solution converges
to $\infty$ as the time converges to the lifespan, provided it is finite.

\begin{theorem}\label{thm_local} Let us assume that  Assumptions \ref{ass-spaces} and \ref{ass-maps} be satisfied.
Then for every    $\mathcal{F}_{T_0}$-measurable $\rH$-valued $p$-integrable
random variable $u(T_0)$    there exits  a  local  process $X=\big(X(t), t\in[0,T_1) \big) $  which is the
  unique local mild solution to the  problem \eqref{NLS-abstract} with the filtration ${\mathbb F}^{T_0}$ and the Brownian
$W^{T_0}$.
Moreover,  given $R>0$ and  $\varepsilon >0$ there exists $\tau(\varepsilon,R)>0$,  such that for every   $\mathcal{F}_{T_0}$-measurable $\rH$-valued
random variable $u(T_0)$ satisfying  $\mathbb{E}|u(T_0) |^{p}_{\rH} \leq R^{p}$, one has
${\mathbb P}\big(T_1\geq \tau(\varepsilon,R)\big) \geq
1-\varepsilon$.
 \end{theorem}

\begin{proof}
The first part
follows from  Proposition \ref{prop-mild}.
The proof of the second part seems to be new with respect to the existing literature.

 Let us fix  $\varepsilon >0$ ; choose  $N$
 such that $N\geq 2 \varepsilon^{-1/p}$.

Fix $\omega\in \Omega$ such that $|u(T_0)(\omega)|_{\rH}<\infty$. Then
for $C:=\sup_{t\in {\mathbb R}} |U(t)|_{L(H)}$, we have $|U_t u(T_0)(\omega)|_{\rH} \leq C |u(T_0)(\omega)|_{\rH}$
and by Assumption \ref{ass-spaces}  we deduce that
$\big( \int_0^T \|U_t u(T_0)(\omega) \|_E^p dt\big)^{1/p} \leq C^p\, \tilde{C}_p(T)\;  |u(T_0)(\omega)|_{\rH}$. This implies that
\[ \mathbb{E} \Big( \sup_{t\in [0,T]} |U_t u(T_0)|_{\rH}^p +
\int_0^T \|U_t u(T_0)\|_{\rE}^p \,dt \Big) \leq C^p\, [1+\tilde{C}_p(T)^p]\;  \mathbb{E}|u(T_0)|_{\rH}^p.\]

 Moreover,  Lemmas \ref{lem-det} and  \ref{PsiT_on_Lp}
imply that given $X\in {\mathbb M}^p(Y_T, {\mathbb F}^{T_0})$, one has
  $\vert \Phi_T^n(X)+\Psi_T^{T_0,n}(X) \vert_{{\mathbb M}^p(Y_T, {\mathbb F}^{T_0})} \leq K_n(T) $,
with  $\sup_{n\leq n_0} K_n(T)\to 0$  as $T\to 0$ for every $n_0 \in\mathbb{N}^\ast$.
Furthermore, the map $\tilde{C}_p(.)$ is non decreasing and   $\tilde{C}_p(T)\to 0$  as $T\to 0$. Hence
one may choose $\delta_{1}(\varepsilon)$ such that
$[1+\tilde{C}_p(\delta_{1}(\varepsilon))^p]^{1/p} \leq \frac32$.

 Let us put $n=N R$    for some "large" $N$ to be chosen later  and
 choose $\delta_{2}(\varepsilon) >0$ such that
$K_{n}(\delta_{2}(\varepsilon)) \leq \frac12 R $.
 Let  $\Lambda_T^{n}$ be the map defined by \eqref{eqn-Lambda}. Since  $ \mathbb{E}(|u(T_0)|_{\rH}^p)^{1/p}\leq R$,
we deduce that for $T\leq \delta_{1}(\varepsilon) \wedge \delta_{2}(\varepsilon)$,
 the range of $\Lambda_T^{n}$ is included in the ball of
${\mathbb M}^p(Y_T, {\mathbb F}^{T_0})$ of radius  $(3C+1)R/2$.
Furthermore, Propositions \ref{prop-det-Lip} and \ref{prop-Lip-Psi_T}  show that there exists
 $\delta_{3}(\varepsilon)$ such that $\Lambda^{n}_T$
is a strict contraction of ${\mathbb M}^p(Y_T, {\mathbb F}^{T_0})$ if $T \leq \delta_{3}(\varepsilon)$. Hence, if one lets
$\tau(\varepsilon,R)= \delta_{1}(\varepsilon) \wedge\delta_{2}(\varepsilon) \wedge \delta_{3}(\varepsilon)\wedge \frac{n}{2}$,
 the unique fixed point
$X^{n}$ of the map $\Lambda^{n}_{\tau(\varepsilon,R)}$
is such that $ \mathbb{E}\big( |X^{n}|_{Y_{\tau(\varepsilon,R)}}^p\big)^{1/p} =
 \|X^{n}\|_{M^p(Y_{\tau(\varepsilon,R)},
 {\mathbb F}^{T_0})} \leq (3C+1) R/2 $.
Proposition \ref{prop-mild} shows that the process $(X^{n}(t) , t\leq \tau_{n})$ is a local mild solution
to problem \eqref{NLS-abstract}. Furthermore, by the definition of the stopping time $\tau_n$,
 the set $\{\tau_{n} <\tau(\varepsilon,R)\}$ is contained in the set $\{\|X^{n} \|_{Y_{\tau(\varepsilon,R)}} \geq n\}$.
 Since by the Chebyshev's inequality
${\mathbb P}\big( \|X^{n} \|_{Y_{\tau(\varepsilon,R)}} \geq n\big) \leq
 \big((3C+1)/2\big)^p N^{-p} \leq \varepsilon$ provided that $N$ is chosen large enough, we infer that
$ {\mathbb P}\big( \tau_{n} <\tau(\varepsilon,R) \big) \leq \varepsilon $.

Therefore,  ${\mathbb P}\big(\tau_{n}\geq \tau(\varepsilon,R)\big) \geq 1-\varepsilon$
 and the stopping time $T_1=\tau_{n}$ satisfies the requirements of the Theorem; this  concludes the proof.
\end{proof}

 Assume that $T_0=0$, $u_0\in L^p(\Omega,\mathcal{F}_0,\rH)$  and let $\tau_n$ be
 defined by \eqref{eqn-tau_n}. Set
\[ \tau_{\infty}(\omega) := \lim_{n\toup \infty}\tau_n(\omega), \; \omega \in \Omega,\]
Then $U^n=X^n$ and
Lemma \ref{Lem:loc-uniq}
 implies that the following identity uniquely defines a local process
$(u(t)\, ,\, t< \tau_{\infty})$ as follows:
\begin{equation} \label{eqmax}
 u(t,\omega):=
  u^n(t,\omega), \; \mbox{ if } t < \tau_n(\omega),\; \omega \in \Omega.
\end{equation}

We can now prove the existence and uniqueness of a maximal solution to our abstract evolution equation \eqref{NLS-abstract};
this is the main result of this section.
\begin{theorem}\label{thm_maximal-abstract}
 Assume that the  Assumptions \ref{ass-stochastic basis}, \ref{ass-spaces} and  \ref{ass-maps}  are
satisfied with  $2\leq \beta<p$
and $a\in [1,\frac{p}2)$.
 Assume also that $\rE$ is a martingale type $2$ Banach space. Then for every    finite  and accessible
 $\mathbb{F}$-stopping time  $T_0$  and every  $u_0\in L^p(\Omega,\mathcal{F}_0,\rH)$,
 the process $u=(u(t)\, ,\, t<  \tau_\infty) $ defined by \eqref{eqmax}
   is the unique local maximal solution to equation \eqref{NLS-abstract}.
  Moreover,
$ {\mathbb P }\big(\{\tau_\infty <\infty\} \cap \{\sup_{ t<\tau_\infty} |u(t)|_{\rH}<\infty \}\big)=0$
and on $\{ \tau_\infty<\infty \}$, $\limsup_{t\to \tau_\infty} |u(t)|_{\rH} = +\infty$ a.s.
\end{theorem}

\begin{proof} %[Proof of Theorem \ref{thm_maximal-abstract}]
We will follow some ideas from  \cite[Theorem 4.10]{Brz_1997}.
The process $(u(t), t<\tau_\infty)$ is such that   $\mathbb{P}$-a.s. we have $\vert u\vert_{Y_{t}}\to \infty $ as $t\toup \tau_\infty$
on the set $\{\tau_\infty <\infty\}$ and this solution is hence maximal.
We now prove the last part which is not obvious since it requires to prove that the $\rH$ norm of $u(t)$
does not remain bounded.

   Let us argue by contradiction and assume
that ${\mathbb P}\big( \{\tau_\infty <\infty\} \cap \{\sup_{ t<\tau_\infty} |u(t)|_{\rH}<\infty \}\big) =4\varepsilon >0$.
Choosing $R$ large enough, we may and do assume that
\[ {\mathbb P}\big( \{\tau_\infty <\infty\} \cap \{\sup_{ t<\tau_\infty} |u(t)|_{\rH}\leq R \}\big) =3\varepsilon >0.\]
Let $\sigma_{R}= \inf\{ t\geq 0 : |u(t)|_{\rH}\geq R\} \wedge \tau_{\infty}$; then $\sigma_{R}$ is
a predictable stopping time and  the ${\mathcal F}_{\sigma_{R}}$-measurable set
 $\tilde{\Omega}=\{ \sigma_{R}= \tau_{\infty} <\infty\}$ is such that ${\mathbb P}(\tilde{\Omega}) \geq 3\eps$.\\
 Let $v_{0}$ denote the ${\mathcal F}_{\sigma_{R}}$-measurable variable defined by
 $v_{0}= u(\sigma_{R})$ on $\tilde{\Omega}$ and $v_{0}=0$ otherwise.
 Then Theorem \ref{thm_local} implies the existence of a positive time $\tau_{\eps,R}$ such that the
 ${\mathcal F}^{T_{0}}_{t}$ solution $X(t)$
 to the evolution equation $i dX(t)+ \Delta X(t) dt = F(X(t)) dt + G(X(t)) dW^{T_{0}}(t)$ with initial
 condition $v_{0}$ has a solution on some time interval $[0,T_{1})$ with ${\mathbb P}(T_{1}\geq \tau_{\eps,R})
 \geq 1-\eps$.\\
 Let $v$ be the predictable process defined by:
 \[ v(t,\cdot)= \left\{
 \begin{array}{lll}
 u(t,\cdot) &\mbox{\rm if} & t \leq  \sigma_{R} (\cdot)<  \tau_{\infty}(\cdot),\\
  u(t,\cdot)  & \mbox{\rm on}&  \{ t>\sigma_{R}(\cdot)\} \cap \tilde{\Omega}^{c}, \\
   X(t-\sigma_{R}(\cdot)) & \mbox{\rm on}&
 \{ t>\sigma_{R}(\cdot)\} \cap \tilde{\Omega}  \end{array} \right.  .
 \]
 Therefore,  $0< \mathbb{E} \big( |v|_{Y_{\tau_{\infty}+ \frac{1}{2} \delta_{\eps,R}}} 1_{\tilde{\Omega}}\big)<\infty$,
 which contradicts the definition of $\tau_{\infty}$; this concludes the proof.
\end{proof}

\section{Abstract Stochastic NLS in the Stratonovich form}\label{sec-Stratonovich}

Multiplying equation \eqref{NLS-abstract} by $- i$, we obtain the following form of it:
\[ du(t)=i \, \Big[ A u(t)- F(u)\Big]\,dt + (-i)\, G(u) \, dW(t), \;t\geq 0.
\]
Now we suppose that the stochastic term is in the Stratonovich form, i.e. formally

\begin{equation}\label{NLS-abstract-Stranonovich}
 du(t)=i \, \Big[ A u(t)- F(u)\Big]\,dt + (-i)\, G(u)\,\circ \, dW(t),  \;t\geq 0.
\end{equation}
Below we will present a rigorous approach to equation \eqref{NLS-abstract-Stranonovich}.
 We assume that the assumptions of Theorem \ref{thm_maximal-abstract} are satisfied.
 In order for this problem to make sense,  we need to make stronger assumption on the map $G$.
To be precise, we require the following assumptions.
\begin{assumption}\label{ass-K} The Hilbert space $\rK$ is such that
\[ \rK \subset \mathcal{R}:=\rH \cap \rE,\]
and  the natural embedding $\Lambda: \rK \embed \mathcal{R}$ is $\gamma$-radonifying.
 \end{assumption}
\begin{assumption}\label{ass-G}
 The map
$G : \mathcal{R}  \to \mathcal{L}(\mathcal{R} ,\mathcal{R} )$
 is   of real ${\mathcal C}^1$-class.
 \end{assumption}

  Note that the above  Assumptions \ref{ass-K} and \ref{ass-G}  imply that the naturally induced map
\[{G}:\cR \ni u \mapsto G(u)\circ \Lambda \in  \rR(\rK,\cR),\]
which can be identified with   the original   map $G : \mathcal{R}  \to \mathcal{L}(\mathcal{R} ,\mathcal{R} )$,
is of real ${\mathcal C}^1$-class and satisfies
\[
|{G}(u)|_{R(\rK,\cR)}\leq |G(u)|_{{\mathcal L}(\cR,\cR)}
\, |\Lambda|_{R(K,\cR)}.\]
 Furthermore,   for $u\in \cR$,
the Fr{\'e}chet derivative $G'(u) = d_uG \in \mathcal{L}\big(
\cR , {\mathcal L}(\cR, \cR)\big)$
 is  $\mathbb{R}$-linear.
% Hence $d_u G$ does not
% in general commute with the multiplication by $i$ on ${\mathbb C}$ which can be identified with the linear map
% $ {\mathcal I}$ on $\RR^2$ defined by $ {\mathcal I}(x,y)=(-y,x)$. See however Lemma \ref{lem-i-multiplication}.

By the Kwapie\'n-Szyma\'nski Theorem \cite{Kwapien+Szym_1980} we can assume  that an
 ONB $\{e_j\}_{j \geq 1}$ of $\rK$ can be  chosen (and fixed for the remainder of the article) in such a way   that
\begin{equation}\label{eqn-Kwapien+Szym-0}
\sum_{j\geq 1}  \vert   \Lambda e_j \vert^2_{\mathcal{R}} <\infty.
\end{equation}
Let us recall that  for a bilinear $\phi \in \mathbb{L}_2(\cR ;\cR )$, we put
\begin{equation}\label{eqn-trace}
\mathrm{tr}_{\rK}(\phi):= \sum_{j\geq 1}  \  \phi (\Lambda e_j, \Lambda e_j) \in \cR.
\end{equation}

Following \cite{Brz+Elw_2000} we can define the Stratonovich differential  $-i G(u)\circ\,dW(t)$  as follows:
\begin{eqnarray}\nonumber
-i G(u)\circ\,dW(t)&=&-i G(u)\,dW(t)+\frac{1}2 \mathrm{tr}_{\rK}\big(-i
G^\prime(u)\big)\big(-iG(u)\big)\, dt\\
&=&-i G(u)\,dW(t)+\frac{1}2 \mathrm{tr}_{\rK}\big(i
G^\prime(u)\big)\big(iG(u)\big)\, dt.
\end{eqnarray}
If for $u\in \cR$, we denote by ${\mathcal M}(u)=(iG^\prime(u)) (iG(u))$ the element
of $  \mathcal{L}\big(\cR ,\mathcal{L}(\cR,\cR)\big)\equiv  \mathbb{L}_2(\cR ;\cR )$ defined by
\begin{equation}\label{eqn-M} {\mathcal M}(u)(h_1,h_2)=(iG^\prime(u)) (iG(u))(h_1,h_2)=(iG'(u)(iG(u) h_1) h_2, \quad h_1,h_2\in \cR,
\end{equation}
then equation \eqref{NLS-abstract-Stranonovich} can be reformulated in the following way:
\begin{equation}\label{NLS-abstract-Stranonovich-Ito}
d u =\Big[i A u - i F(u) + \frac{1}{2}  \mathrm{tr}_{\rK}\big({\mathcal M}(u)\big) \Big] \, dt
+ i G(u) dW(t).
\end{equation}

Given real-valued maps  $\phi,\psi$ we write $\phi \lesssim \psi$
if there  exists a constant $c$ such that $\phi\leq c \psi$.
We write $\phi \eqsim \psi$ to
express that $\phi \lesssim \psi$ and $\psi \lesssim \phi$.

 Let us recall that  although $\mathcal{R}$ is a complex Banach space,  below we will treat it as a real Banach space.
The following result states the equivalence of the $\mathbb{L}_2(\cR ;\cR )$ norm of ${\mathcal M}(u)$ and of $G'(u)G(u)$. Its proof, which is
is straightforward, is omitted.
\begin{lemma}\label{lem-i-multiplication}
Assume that
the multiplication by $i$ is a bounded real linear map in the real Banach space $\mathcal{R}$. Then  we have, for all $u,v\in\mathcal{R}$,
\begin{eqnarray*}
\vert {\mathcal M}(u)\vert_{\mathbb{L}_2(\cR ;\cR )} &\eqsim& \vert G^\prime(u)(G(u))\vert_{\mathbb{L}_2(\cR ;\cR )},\\
\vert {\mathcal M}(u)-{\mathcal M}(v)\vert_{\mathbb{L}_2(\cR ;\cR )} &
\eqsim& \vert G^\prime(u)(G(u))-G^\prime(v)(G(v))\vert_{\mathbb{L}_2(\cR ;\cR )},
\end{eqnarray*}
where $G^\prime(u)(G(u)) (h_1,h_2)= G'(u) \big( G(u) h_1\big) h_2$ for $h_1,h_2\in {\mathcal R}$.
\end{lemma}
 Since $\cR  \embed \rH$ continuously and  the trace
\[\mathrm{tr}_{\rK} : \phi\in {\mathbb L}_2(\cR ;\cR) \to \mathrm{tr}_{\rK}\big(\phi\big)\in \cR\]
is a linear and bounded map,  we can find $C>0$ such that
  \[
\big|  \mathrm{tr}_{\rK} {\mathcal M}(u)-   \mathrm{tr}_{\rK} {\mathcal M}(v) \big|_{H}\leq C
\vert {\mathcal M}(u) -{\mathcal M}(v) \vert_{{\mathbb L}_2(\cR ;\cR)},\;\; u,v \in \cR.
\]

\begin{definition}\label{def-sol-Stratonovich}
We say that a process $u$ is a  local (resp. local maximal, global) solution to equation \eqref{NLS-abstract-Stranonovich}
 if and only if it is a local (resp. local maximal, global) solution to the It\^o equation \eqref{NLS-abstract-Stranonovich-Ito},
that is equation \eqref{NLS-abstract}, with  the map $F$ being  replaced by  $F_1=F+\frac{i}{2}  \mathrm{tr}_{\rK}\Big[ \mathcal{M}\Big]$.
\end{definition}

We now state the stronger version of Assumption \ref{ass-maps}(iii).
\begin{assumption}\label{ass-maps-b}
Let $p\in (2,\infty)$, $a\in [1, \frac{p}{2})$ and $\gamma \in [1,p)$.
The  map $G:  \cR \to   \mathcal{L}(\cR,\cR)$  is of class ${\mathcal C}^1$ and such that for some  positive constant $C$,
and  all $u,v\in  \cR  $, with $M$ having been defined in \eqref{eqn-M}
\begin{align}
\label{growth-G_H^12}
&|G(u)|_{\mathcal{L}(\cR,\rH)}
\leq C  \Big[1+ \| u\|_{\rE}^{a}+ \big(1+\| u\|_{\rE}^{a-1}\big) \vert u\vert_{\rH}\Big]
\\
\label{lip-G_H^12}
&|G(u)-G(v)|_{\mathcal{L}(\cR ,\rH)} \leq
C \big( 1+\|u\|_{\rE}^{a-1}   +  \|v\|_{\rE}^{a-1} \big) |u-v|_{\rH}  \nonumber
\\
&\qquad \qquad  + C \big( 1+\|u\|_{\rE}^{(a -2)^+}  +  \|v\|_{\rE}^{(a-2)^+} \big)
\big( 1+|u|_{\rH}  +  |v|_{\rH}\big)
 \|u-v\|_{\rE} , \\
\label{growth-G'G_H^12}
&\Big| {\mathcal M}(u)\Big|_{\mathbb{L}_2(\cR ;\rH)} \leq C
\Big[1+ \| u\|_{\rE}^{\gamma}+ \big(1+\| u\|_{\rE}^{\gamma-1}\big) \vert u\vert_{\rH}\Big] ,
\\\nonumber
&\Big|
{\mathcal M}(u)-{\mathcal M}(v)\Big|_{\mathbb{L}_2(\cR;\rH)} \leq
C   \big( 1+\|u\|_{\rE}^{\gamma-1}   +  \|v\|_{\rE}^{\gamma-1} \big) |u-v|_{\rH}
\\
&\qquad\qquad + \;\;C \big( 1+\|u\|_{\rE}^{(\gamma -2)^+}  +  \|v\|_{\rE}^{(\gamma-2)^+} \big)
\big( 1+|u|_{\rH}  +  |v|_{\rH}\big)
 \|u-v\|_{\rE} .
\label{lip-G'G_H^12}
\end{align}
\end{assumption}

Theorem \ref{thm_maximal-abstract} immediately yields the following result.
\begin{theorem}\label{thm_maximal-abstract-Stratonovich}
 Assume that the  Assumptions \ref{ass-stochastic basis}, \ref{ass-spaces} and  \ref{ass-maps}  are
satisfied with  $2\leq \beta<p$
and $a\in [1,\frac{p}2)$ and also Assumptions  \ref{ass-K},  \ref{ass-G}  and
\ref{ass-maps-b} are  satisfied.
 Assume also that $\rE$ is a martingale type $2$ Banach space and
  $W=(W_t)_{t\geq 0}$ is an $\mathcal{R}=\rH\cap\rE$-valued Wiener process defined on some filtered probability space
$(\Omega, \mathcal{F},\mathbb{F},\mathbb{P})$ satisfying the usual assumptions.
 Then  for every  $u_0 \in L^p(\Omega,\mathcal{F}_0,\rH)$ there exists a unique
  local  process $u=\big(u(t)\, , \, t<  \tau_\infty \big)$
 which is the  local maximal solution to equation  \eqref{NLS-abstract-Stranonovich}. Moreover,
$ {\mathbb P }\big(\{\tau_\infty <\infty\} \cap \{\sup_{ t<\tau_\infty} |u(t)|_{\rH}<\infty \}\big)=0$
and on $\{ \tau_\infty<\infty \}$, $\limsup_{t\to \tau_\infty} |u(t)|_{\rH} = +\infty$ a.s.
\end{theorem}

\section{Stochastic NSEs: the local existence}
\label{sec-local}

In this section we will formulate  results about the existence and the uniqueness of solutions to the stochastic NLS equation,
  about an  equation in the It\^o (respectively Stratonovich) formulation. The former one will be based on Theorem \ref{thm_maximal-abstract}
 and the latter on Theorem \ref{thm_maximal-abstract-Stratonovich}.  For simplicity we formulate it for $d=2$.
 One can also prove a similar result for $d >2$ but since in the latter case we do not know whether the solution is global
or blows  up in finite time, we have decided to leave it out.
Thus we  assume that $M$ is a $2$-dimensional,  compact riemannian manifold and  $\Delta$
is the Laplace Beltrami operator on $M$. We assume that some numbers $p,q$ satisfy the scaling admissible condition  (with $d=2$)
\begin{equation}\label{eqn-compatibility}
\frac2p+\frac2q=1.
\end{equation}
We choose  $s\in (1-\frac1p,1]$ and put $\hat{s}:=s-\frac1p$. Since $s-\frac1p>\frac{d}q$,   we infer that the Sobolev space $W^{\hat{s},q}(M)$
is embedded into the space ${\mathcal C}(M)$ of  continuous (and hence bounded) functions on $M$;
the latter is a Banach space equipped with the $L^\infty$-norm.

  In this section we let $\rH=H^{s,2}(M):=H^{s,2}(M,{\mathbb C}) $ and
$\rE=W^{\hat{s},q}(M):=W^{\hat{s},q}(M,{\mathbb C}) $,  where ${\mathbb C}$ is identified with ${\mathbb R}^2$. Finally,  as in Assumption \ref{ass-K}, we denote by $\mathcal{R}$ the following real Banach space
\[  \mathcal{R}=H^{s,2}(M) \cap W^{\hat{s},q}(M) = \rH\cap \rE.\]
In order to study the diffusion operator $G$ we need the following result which follows from Corollary \ref{cor-algebra} and
Theorem \ref{thm-Nemytski-Lip}.
\begin{lemma}\label{lem-R} Under the above assumptions  the pointwise multiplication map
\begin{equation}\label{eqn-Pi}
\Pi: \cR \times \cR \ni (u,h)\mapsto uh \in \cR
\end{equation}
is   bilinear and continuous. The same assertion holds for $\mathcal{R}=H^{s,2}(M) \cap L^\infty(M)$.
\end{lemma}

The following assumption will play an essential r\^ole in the next section as well as in the last part of
 Theorem \ref{thm_maximal-manifold}, which is the main result of the current section.
\begin{assumption}\label{ass-tildeg} There exists a function  $\tilde{g}:[0,+\infty)\to \RR$   of class ${\mathcal C}^1$
 such that
\begin{equation}\label{eqn_global-tg}
 g(z)=\tg(|z|^2) z, \; z\in \mathbb{C}.
\end{equation}
\end{assumption}
We will consider the  "generalized" Nemytski map $\tilde{G}$ associated with $g$ (see \cite{Brz+Elw_2000}), that is with $\cR=\rH\cap \rE$,
  \begin{equation} \label{generalizedG}
 \tilde{ G}:\cR \ni u\mapsto \big\{ h\mapsto \Pi(g(u),h)  \big\} \in \mathcal{L}(\cR,\cR).
\end{equation}
The aim of this section is to prove the existence and uniqueness of a maximal solution to
 \begin{equation}\label{NLS}
idu(t) + \Delta u(t)\, dt=f(u)\,dt +g(u)\,dW(t), \quad u(0)=u_0,
\end{equation}
or to its Stratonovich formulation
 \begin{equation}\label{NLS-Stratonovich}
idu(t) + \Delta u(t)\, dt=f(u)\,dt +g(u)\circ dW(t), \quad u(0)=u_0.
\end{equation}
%\begin{remark}\label{rem-def-sol}
 By a solution of equation \eqref{NLS} (resp.  \eqref{NLS-Stratonovich}), we mean a solution to
its abstract version  \eqref{NLS-abstract} (resp. \eqref{NLS-abstract-Stranonovich-Ito}) defined in terms of the  Nemytski map $\tilde{G}$.
%\end{remark}
 As proved in the previous section, the Stratonovich formulation requires to identify $\mathcal{M}$ as the Nemytski map corresponding to function  $z\mapsto (ig'(z))(ig(z))$.
This will be a consequence of the following general result.
\begin{lemma}
\label{lem-j}
Assume  that $(H,\vert \cdot \vert)$ is a real separable Hilbert space and that
   ${\mathcal I}:H\to H$ is  a bounded linear operator such that
${\mathcal I}^2=-Id,\; \langle z,  {\mathcal I}z \rangle=0, z\in H$.
Let $\varphi :H \to H$ be  function of the form $\varphi(z)=\tilde{\varphi}(|z|^2) {\mathcal I} z$, $z\in H$,
where $\tilde{\varphi}:\mathbb{R} \to \mathbb{R}$ is differentiable.   Then, $\varphi$ is $\mathbb{R}$-differentiable and
\begin{equation}\label{eqn-B03}
 \big(\varphi^\prime (z)\big)\big(\varphi(z))=\big(d_z \varphi \big)\big(\varphi(z))=-|\tilde{\varphi}(|z|^2)|^2\, z,\; \forall
z\in H.
\end{equation}
\end{lemma}
\begin{proof}%[Proof of formula \eqref{eqn-useful}.]
Let $y,z \in H$; then we have $[\varphi^\prime (z)](y)=2\tilde{\varphi}^\prime(|z|^2)\langle z,y\rangle  {\mathcal I}z +
\tilde{\varphi}(|z|^2) {\mathcal I} y$. Therefore, given $z\in H$, we deduce
\begin{align*}
[\varphi^\prime(z)](\varphi(z))=&
  %=&  2\tilde{h}^\prime(|z|^2)\langle z,h(z)\rangle jz +\tilde{h}(|z|^2)j h(z)\\
%=&
2\tilde{\varphi}^\prime(|z|^2)\langle z,\tilde{\varphi}(|z|^2)  {\mathcal I}z\rangle  {\mathcal I}z
 +\tilde{\varphi}(|z|^2){\mathcal I} \varphi(z)\\
=&2\tilde{\varphi}^\prime(|z|^2) \tilde{\varphi}(|z|^2)\langle z,  {\mathcal I}z\rangle  {\mathcal I}z
 |\tilde{\varphi}(|z|^2)|^2  {\mathcal I}^2z  =-|\tilde{\varphi}(|z|^2)|^2 z .
\end{align*}
This completes the proof of \eqref{eqn-B03}.
\end{proof}
Since we identify $\mathbb{C}$ with $\mathbb{R}^2$   and the operator of multiplication   by $i$
 with the operator ${\mathcal I}:\mathbb{R}^2 \ni (x,y)\mapsto (-y,x)\in \mathbb{R}^2$, we deduce the following result.
\begin{corollary}
\label{cor-g^prime}
Assume that  a function
$g:\mathbb{C} \to \mathbb{C}$ satisfies Assumption \ref{ass-tildeg}
for a  differentiable function $\tilde{g}:\mathbb{R}\to \mathbb{R}$.
Then  $g$ is $\mathbb{R}$-differentiable and, with  $\langle\cdot,\cdot\rangle$  (resp.  $\vert \cdot \vert$)
 denoting the scalar product (resp. the euclidian norm),
in $\mathbb{C}\cong\mathbb{R}^2$, we have for all $z\in \mathbb{C}$,
\begin{equation}\label{eqn-B02}
m(z):=\big((ig)^\prime (z)\big)\big(ig(z)\big)=-\big( \tilde{g}(\vert z\vert ^2) \big)^2 z,\quad
\langle \big((ig)^\prime (z)\big)\big(ig(z)\big),z\rangle=-\vert g(z)\vert^2.
\end{equation}
\end{corollary}
In  particular we get  the  formulation of $ \mathrm{tr}_{\rK}\big({\mathcal M}(u)\big)$, where ${\mathcal M}$ is defined by \eqref{eqn-M}.
 Let $\Pi$ be the bilinear map  defined in \eqref{eqn-Pi} and    $\Lambda : K \to \cR$
denotes the natural embedding (which is  a gamma-radonifying operator) and $\big(e_j\big)_{ j\geq 1}$ is a complete orthonormal
system of $K$ satisfying \eqref{eqn-Kwapien+Szym-0} and consisting of real valued functions.
Then it follows from the definition \eqref{eqn-trace}  of the trace  and the Kwapie\'n-Szyma\'nski result \eqref{eqn-Kwapien+Szym-0} that
\begin{equation}\label{eqn-trPi}
\mathfrak{p}:=\mathrm{tr}_{\rK}(\Pi) = \sum_{j\geq 1} (\Lambda e_j)^2 \in H^{s,2}(M,\mathbb{R})\cap W^{\hat{s},q}(M,\mathbb{R}) \subset \cR .
\end{equation}

Let us make another  useful observation.
Let $m$ be defined in \eqref{eqn-B02} and $\mathbb{M}$  be   the Nemytski map corresponding to the function $m:\mathbb{C}\to\mathbb{C}$,
that is $\mathbb{M}(u)=  m \circ u,  \;\;   u\in \cR$. Then
\begin{equation}%\label{eqn-Phi}
{\mathcal M}(u)(h_1,h_2)=\mathbb{M}(u)h_1h_2, \;\; \mbox{ for }u,h_1,h_2\in {\mathcal R}.
\end{equation}
 Furthermore, we have the following:
\begin{lemma}\label{lem-B03}
Assume that  a function
$g:\mathbb{C} \to \mathbb{C}$ satisfies Assumption \ref{ass-tildeg}
for a differentiable function $\tilde{g}:\mathbb{R}\to \mathbb{R}$. Then the maps $G$ and $\mathbb{M}$ transform  $\cR$
 to $\mathcal{L}(\cR,\cR)$ and  $\cR $ respectively,  and,
for every $u\in \cR$,
\begin{align}\label{eqn-B06}
  \tr_{\rK}({\mathcal M}(u))
&= \Pi\big(  \mathfrak{p}, \mathbb{M}(u)\big)  ,\\
\label{eqn-B07}
 \Re \langle \tr_{\rK} {\mathcal M}(u) ,u\rangle_{L^2(M)}&=
 - \int_M \vert g(u(x))\vert^2 \mathfrak{p}(x)\, dx
= - \Vert G(u)\Lambda \Vert^2_{R(\rK,L^2(M))} , \\
\label{eqn-B08}
 \Re \langle \nabla \tr_{\rK} {\mathcal M}(u),\nabla u\rangle  & =
-\int_M \vert \tilde{g}(\vert u(x)\vert^2)\vert^2 \Re \langle u(x) \nabla \mathfrak{p}(x), \nabla u(x)\rangle \, dx
\\
 -   4 \int_M (\tilde{g}^\prime \tilde{g})(|u(x)|^2)  &
 \Re \langle u(x) \nabla u(x), \nabla u\rangle \mathfrak{p}(x) \, dx
-2 \int_M  \tilde{g}(|u(x)|^2)^2  \mathfrak{p}(x) \vert \nabla u(x)\vert^2 \, dx.\nonumber
\end{align}
\end{lemma}
\begin{proof}[Proof of Lemma \ref{lem-B03}] Since both $g$ and $m$ are functions of ${\mathcal C}^1$-class and  the point-wise multiplication
in $\cR$ is a bounded bilinear map,  Proposition \ref{prop-Nemytski} implies that ${G}$ and $\mathbb{M}$
 are well defined maps from $\cR$  to $\mathcal{L}(\cR,\cR)$ and  $\cR $.
Using \eqref{eqn-M}  we deduce that for $u\in \cR$,
\[ \tr_{\rK}\Big[\mathcal{M}(u)\Big]=\sum_{j\geq 1}
\big[\big((ig)^\prime (ig)\big)\circ u\big](\Lambda e_j)^2\in \cR.\]
Then the definition \eqref{eqn-B02} of the function $m$  concludes the proof of identity \eqref{eqn-B06}.
 Moreover   the second identity in    \eqref{eqn-B02} yields
\begin{align*} %\label{eqn-useful4}
 \Re& \langle \tr_{\rK}  \mathcal{M}(u),u\rangle_{L^2(M)}
 %\\
%&
=\sum_{j} \int_M \!   \Re \langle (ig)^\prime (u(x))\big(ig(u(x))\big)\big( (\Lambda e_j)(x)\big)^2,u(x)\rangle dx
 \\
& =-\sum_{j} \int_M  \big( (\Lambda e_j)(x)\big)^2 |g(u(x))|^2 \, dx
=-\sum_{j} |\tilde{G}(u)\Lambda e_j|_{L^2(M)}^2 .
\end{align*}
This concludes the proof of \eqref{eqn-B07}; that of \eqref{eqn-B08} is similar.
\end{proof}

Recall that $m$ is defined by \eqref{eqn-B02}. The above results show that the Stratonovich equation \eqref{NLS-Stratonovich}
can be written in the following It\^o form:
 \begin{equation}\label{NLS-Stratonovich-Ito2}
du(t) = \Big[ i A u(t) - i f(u(t)) + \frac{1}{2} \mathfrak{p}\, m(u(t))  \Big]\,dt  - i g(u(t))  \,dW(t).
\end{equation}
We now prove the existence and uniqueness of a maximal solution to equations \eqref{NLS} and  \eqref{NLS-Stratonovich}
- or \eqref{NLS-Stratonovich-Ito2}.
This is the main result of this section. \\
\begin{theorem}\label{thm_maximal-manifold}
Assume that $M$ is a compact riemannian manifold of dimension $d=2$.
 Assume that $f: \mathbb{C}\to \mathbb{R}$ is of real ${\mathcal C}^1$-class satisfying, for some $\beta \geq 2$ and all $y,z\in \mathbb{C}$,
\begin{equation}\label{ineq-growth}
\vert f(y)\vert \leq C(1+\vert y\vert^{\beta}),\; \vert f^\prime(y)\vert \leq C(1+\vert y\vert^{\beta -1}),\;
\vert f^{\prime}(y)- f^{\prime}(z)\vert \leq C(1+\vert y\vert^{\beta -2 }+ \vert z\vert^{\beta -2}) |y-z|.
\end{equation}
Assume that $g: \mathbb{C}\to \mathbb{R}$ is of real ${\mathcal C}^1$-class satisfying, for some $a \geq 1$ and all $y,z\in \mathbb{C}$,
\begin{equation}\label{ineq-growth-g}
\vert g(y)\vert \leq C(1+\vert y\vert^{a}),\;
\vert g^\prime(y)\vert \leq C(1+\vert y\vert^{a -1}),\; \vert g^{\prime}(y) - g^{\prime}(z)\vert \leq C(1+\vert y\vert^{(a -2)^+}+
|z|^{(a-2)^+}) |y-z|.
\end{equation}
Assume that $p>\beta \vee (2a) $ and $q>2$ satisfy the scaling admissible  condition \eqref{eqn-compatibility}.
Assume that $s\in (1-\frac1p,1]$
 and let  $W=(W(t)\, ,\, t\geq 0)$ be an $H^{s,2}(M)\cap W^{s-\frac1p,q}(M)$-valued Wiener process defined on some filtered probability space
$(\Omega, \mathcal{F},\mathbb{F},\mathbb{P})$ satisfying the usual assumptions.
 \\
 Then for every $u_0\in L^p(\Omega,\mathcal{F}_0,H^{s,2}(M))$   there exist a local  process
$u=(u(t)\, ,\, t<  \tau_\infty) $ whose trajectories are  $H^{s,2}(M)$-valued continuous  and locally $p$-integrable with values in
$W^{s-\frac1p,q}(M)$,
 that  is the unique local maximal solution to \eqref{NLS}
  Moreover,
$ {\mathbb P }\big(\{\tau_\infty <\infty\} \cap \{\sup_{ t<\tau_\infty} |u(t)|_{H^{s,2}(M)}<\infty \}\big)=0$
and
\[\limsup_{t\to \tau_\infty} |u(t)|_{H^{s,2}(M)} = +\infty \mbox{ a.s. on } \{ \tau_\infty<\infty \}.
\]
Suppose furthermore that $p>\beta \vee (2a)\vee \gamma $, that  $g$ satisfies Assumption \ref{ass-tildeg} and that
$m(z)=-\tilde{g}(|z|^2)^2 z$ satisfies the following  condition for some $\gamma \geq 2$ and all $y,z\in \mathbb{C}$:
\begin{equation}\label{ineq-growth-m}
\vert m(y)\vert \leq C(1+\vert y\vert^{\gamma}),\;
\vert m^\prime(y)\vert \leq C(1+\vert y\vert^{\gamma -1}),\; \vert m^{\prime}(y) - m^{\prime}(z)\vert \leq C(1+\vert y\vert^{(\gamma -2)}+
|z|^{(\gamma-2)}) |y-z|.
\end{equation}
 Then the same conclusion as above holds for the equation \eqref{NLS-Stratonovich} in the Stratonovich form.
\end{theorem}
\begin{remark}\label{example-tg}
 Let $\tilde{g}$ be of class ${\mathcal C}^2$ such that for some  constants $C>0$,   $\alpha_0\geq 0$, and some constants $\alpha_i$,
$i=1,2$ one has for all $r\geq 0$:
\[ |\tg(r)|\leq C(1+r^{\alpha_0}), \quad |\tg'(r)|\leq C(1+r^{\alpha_1}), \quad |\tg''(r)|\leq C(1+r^{\alpha_2}).\]
Then if $g(z)=\tg(|z|^2)z$ and $m(z)=\tg(|z|^2)^2z$, the function $g$ satisfies condition \eqref{ineq-growth-g} with $a\geq (2\alpha_0+1)\vee
(2\alpha_1+3)\geq 1 $ and $a\geq 2\alpha_2+5$ if $a>2$. Furthermore, the function $m$ defined by  $m(z)=-\tilde{g}(|z|^2)^2 z$ satisfies
\eqref{ineq-growth-m} with $\gamma = ( 2a-1) \vee 2$.
\end{remark}
\begin{proof}[Proof of Theorem \ref{thm_maximal-manifold}]
 We only consider the case $s\in (1-\frac1p,1)$. The case $s=1$  can be dealt with  analogously.\\
 We put ${\mathcal H}_0=L^2(M)$,
$\rE_0=L^q(M)$, $\rH_0=H^{\frac1p,2}(M)$, $A=\Delta$ with $D(A)=H^{2,2}(M)$ and
 $\mathbf{U}=\big(U_t\big)_{t\in \mathbb{R}}$, the unitary group on
 $L^2(M)$ generated by $iA$. Then Assumption \ref{ass-spaces} is satisfied. Next we put $\hat{s}=s-\frac1p$.
Since by the assumptions $s-\frac1p>1-\frac2p=\frac2q>0$, we infer that $\hat{s}>0$. Finally we  put
 ${\mathcal H}=A^{-\hat{s}}({\mathcal H}_0)=H^{s,2}(M)$,
$\rH=A^{-\hat{s}}(\rH_0)=H^{s+\frac1p,2}(M)$ and $\rE=W^{\hat{s},q}(M)$.
Then Assumption \ref{ass-maps}(i) is satisfied and by Lemma \ref{lem-spaces+maps},
 Assumption \ref{ass-spaces} is satisfied as well.
 Note that since $q>2$, the space $W^{\hat{s},q}(M)$ is bigger that $H^{\hat{s},2}(M)$.
 Moreover, since again $\hat{s}>\frac2q$, in view of  Theorem \ref{thm-Nemytski-Lip},
 the Nemytski maps $F$ satisfies Assumption  \ref{ass-maps}(ii).\\
 As above   we infer that $G$ satisfies Assumption \ref{ass-maps}(iii);
see also Proposition 6.4 in \cite{Brz+Elw_2000}, where a weaker version of our results from section 2 was used. \\
 We conclude that the problem \eqref{NLS} is a special case of the problem \eqref{NLS-abstract} for the
 above choice of  spaces $\rH$ and $\rE$. Therefore, the first result follows by applying Theorem \ref{thm_maximal-abstract}.\\

We now turn to the Stratonovich evolution equation \eqref{NLS-Stratonovich} written in terms of an It\^o integral in \eqref{NLS-Stratonovich-Ito2}.
We only need to show that the function ${\mathcal M}$ defined in \eqref{eqn-M}
satisfies Assumption \ref{ass-maps-b},
 and in particular inequalities \eqref{growth-G'G_H^12} and \eqref{lip-G'G_H^12}.
First notice that if $g$ satisfies \eqref{ineq-growth-g}
then $G$ satisfies the first part of Assumption \ref{ass-maps-b} with the same spaces $\rE$ and $\rH$.
If Assumption \ref{ass-tildeg} is satisfied, the map $\mathbb{M}:\cR\to\cR$ is Lipschitz on balls and
 since $m$  satisfies the assumption \eqref{ineq-growth}
 with some parameter  $\gamma=2a-1$, we deduce that  $\mathbb{M}$ satisfies the second  part of Assumption \ref{ass-maps-b} with
the same  choice of  spaces $\rE$ and $\rH$.
This completes the proof.
\end{proof}
\begin{remark}\label{rem-incompatibility}
 Although $q>2$ and $s>\hat{s}$, since $s-1< \hat{s}-\frac{2}q$, we cannot deduce
that $H^{1,2}(M) \subset W^{\hat{s},q}(M) $; see e.g. Theorem \cite[Theorem 4.6.1]{Triebel_1978}.
 In fact,  in view of the Sobolev embedding Theorem, $H^{s,2}(M) $ is not a subset of $H^{\hat{s},q}(M) $.
 On the other hand, we believe  that although $ \hat{s}-\frac{d}q> s-\frac{d}2$,
 but $q>2$ and $s>\hat{s}$, it is not true that $H^{\hat{s},q}(M)  \subset H^{s,2}(M) $.
 Hence, the two Banach spaces  $H^{s,2}(M)$ and $ H^{\hat{s},q}(M) $ are not included in one another.
\end{remark}

\section{Existence of a global solution $H^{1,2}$-valued solution to the Stochastic NLS in the Stratonovich form}
\label{sec-global}

\subsection{Preliminaries}\label{sec-global-prel}

As in  section \ref{sec-local},  we  assume below that $M$ is a $2$-dimensional,  compact riemannian manifold and  $\Delta$
is the Laplace Beltrami operator on $M$.
In the previous sections we considered the stochastic nonlinear Schr\"odinger equation \eqref{NLS}
with the initial data $u_0$
belonging to the Sobolev space $H^{s,2}(M)$ for some $s\leq 1$. In this section we will consider the problem
of global existence for $s=1$.
We at first rewrite the non linear Schr\"odinger equation \eqref{NLS} with a Stratonovich integral
and then prove that  the $L^2(M)$ norm of the solution is preserved.
We finally conclude by means of the Khashmiski Theorem with the energy function playing the r\^ole of the  Lyapunov function.

The following notations already, used in section \ref{sec-local} for any $s\in (1-\frac{1}{p}, 1]$, will be used in the entire section for $s=1$.
Given $\theta\in (0,1]$ and $r\in [1,+\infty)$, we put $W^{\theta,r}(M):= W^{\theta,r}(M,{\mathbb C})$
where  $\mathbb{C}$ is identified to ${\mathbb R}^2$. Let $(p,q)$ be a pair of positive numbers which satisfies the scaling admissible condition \eqref{eqn-compatibility},  that  is
$\frac{2}{p}+\frac{2}{q}=1$.   We set $s=1$,  $\hat{s}=1-\frac{1}{p}$, $\rH=H^{1,2}(M)$, $\rE=W^{\hat{s},q}(M)$, $\|u\|_{\hat{s},q}=
\|u\|_{W^{\hat{s},q}(M)}$ and $\mathcal{R}= \rH\cap \rE$. Since $q>2$,  $\hat{s}q > 2$ and hence we deduce that $W^{\hat{s},q}(M)\subset L^\infty(M)$.
We use the notation for the scalar product in $L^2(M):=L^2(M;{\mathbb C})$:
\[ \langle u,v\rangle = \int_M \Re\, \big( u(x) \overline{v(x)}\big) dx, \;\; u,v\in L^2(M).\]

We will consider the stochastic NLS equation  in Stratonovich form, that is for $u_0\in H^{1,2}(M)$,
\begin{equation} \label{NLS-Strat_01}
idu(t) + \Delta u(t) dt = f(u) dt + g(u)\circ dW(t),\quad  u(0)=u_0.
\end{equation}

To prove the global existence of the solution to the the NLS equation \eqref{NLS-Strat_01}, we need to impose conditions on the noise
$W$, on the diffusion coefficient $g$ and on the non-linearity $f$ stonger that those made in the previous section.
 \begin{assumption} \label{ass-W}
Thus we suppose that $(W(t), t\geq 0)$  is  a real  $W^{1, 2 s_0}(M,{\mathbb R})\cap W^{\hat{s},q}(M,{\mathbb R})$-valued Wiener process,
for some $s_0> 1$.
\end{assumption}
Let $\rK$ be  the reproducing kernel Hilbert space of the law of the $H^{1, 2 s_0}(M,{\mathbb R})\cap W^{\hat{s},q}(M,{\mathbb R})$-valued
random variable ${W}(1)$.
Then  the embedding $ \Lambda : \rK \to H^{1, 2 s_0}(M,{\mathbb R})\cap W^{\hat{s},q}(M,{\mathbb R})$ is
 $\gamma$ radonifying.
\begin{lemma} \label{lem-W}
For $x\in M$ let $\mathfrak{p}(x)={\rm tr}_K(\Pi)(x) = \sum_j |\Lambda e_j(x)|^2$ and $
\mathfrak{q}(x)= \sum_{j\geq 1} \vert \nabla \Lambda e_j(x) \vert^2$.
Then $ \mathfrak{p} \in L^\infty(M)$ and $ \mathfrak{q}\in L^1(M)$. Furthermore, $\sum_{j\geq 1}  \|\nabla \Lambda e_j\|_{L^{2 s_0}}^2 <\infty$.
\end{lemma}
\begin{proof}
By the  Kwapie{\'n}-Szyma{\'n}ski Theorem
\cite{Kwapien+Szym_1980}, we can assume that the ONB $\{e_j\}_{j=1}^\infty$ is chosen in such a way that
$\sum_{j\geq 1}  \|  \Lambda e_j \|^2_{H^{1, 2 s_0}(M) \cap W^{\hat{s},q}(M)} <\infty$. Since $ W^{\hat{s},q}(M)\subset L^\infty(M)$
we deduce that $\mathfrak{p}$ is bounded. Furthermore,
 $\sum_{j\geq 1}  \|  \nabla \Lambda e_j \|^2_{L^1(M)}\leq
\sum_{j\geq 1}  \|  \nabla \Lambda e_j \|^2_{L^{s_0}(M)} <\infty$ and therefore, the series
$\sum_{j\geq 1}   |\nabla \Lambda e_j| ^2$ is absolutely convergent in $L^1(M)$ as claimed; this concludes the proof.
\end{proof}

In this section, we suppose that $g$ satsfies the following stronger version of Assumption \ref{ass-tildeg}.
\begin{assumption}\label{ass-tildeg-final} There exists a bounded function  $\tilde{g}:[0,+\infty)\to \RR$
 of class ${\mathcal C}^1$  such that
\begin{equation}\label{eqn_global-tg-final}
 g(z)=\tg(|z|^2) z, \; z\in \mathbb{C},\\
%\sup_{r\geq 0} &\big[ |\tg(r)| + r\,|\tg'(r)|\big] <\infty. \label{growth-tg}
\end{equation}
Furthermore, we assume that the function $g$ satisfies the conditions \eqref{ineq-growth-g} with $a=1$ and the function $m:\mathbb{C}\to \mathbb{C}$
defined by $m(z)=-\tilde{g}(|z|^2)^2 \, z$ satisfies condition \eqref{ineq-growth-m} with $\gamma \geq 2$.
\end{assumption}

An example of function  $\tilde{g}$ such that the function $g$ defined by \eqref{eqn_global-tg-final} satisfies Assumption \ref{ass-tildeg-final} is a bounded function
of class ${\mathcal C}^2$ such that $\sup_{r>0} (1+r) |\tg'(r)|<\infty$
and $\sup_{r>0} r^{\frac{3}{2}} |\tg''(r)|<\infty$, for instance,   $\tilde{g}(r)=\frac{\ln(1+r)}{C+\ln(1+r)}$ for $r>0$ and $C>0$. Indeed,
the conditions in Remark \ref{example-tg} are satisfied with $\alpha_0=0$, $\alpha_1=-1$ and $\alpha_2=-\frac{3}{2}$, which yields \eqref{ineq-growth-g}
for $g$ with $a=1$ while $m$ satisfies \eqref{ineq-growth-m} with $\gamma =2$.
\bigskip

We study the global existence for two types of equation \eqref{NLS-Strat_01} depending on the non-linear term $f$, which is defocusing or focusing.
Precise assumptions will be described below, but let us mention that a typical example of the former is when $f(u)=\vert u\vert^2u$
while  a typical example of the latter
is when $f(u)=-\vert u\vert u$.
\begin{assumption}\label{ass-focusing}
We assume that   $f: \C \to \R$ is of the  form
\begin{equation}\label{eqn_global-tf}
 f(z)=\tilde{f}(|z|^2) z,\; z\in \mathbb{C},
\end{equation}
where the function  $\tilde{f}:[0,+\infty)\to \RR$ satisfies one of the following two  cases.
\begin{trivlist}
\item[\textbf{Case 1: defocusing nonlinearity.}] The function $\tilde{f}$ satisfies either (a) or (b):\\
 \textbf{(a)} There exist a natural number $N$  and real number $a_k$, $k=0,\cdots,N$,  with $a_N>0$,
such that   $\tilde{f}(r)=\sum_{k=0}^N a_k r^k$ for every $r\in {\mathbb R}$.\\
 \textbf{(b)}
There exist $C>0$ and $\sigma \in  [\frac{1}{2},\infty)$ such that  $\tilde{f}(r)=C r^\sigma$ for every $r\in {\mathbb R}$.
\item[\textbf{Case 2: focusing nonlinearity.}] There exist $C>0$ and
 $\sigma \in [ \frac{1}{2},1)$ such that
 for every $r\in {\mathbb R}$, $\tilde{f}(r)=-C r^\sigma$.
 \end{trivlist}
\end{assumption}
This assumption yields the following result, whose straightforward proof is omitted.
\begin{lemma}\label{lem-f} If Assumption \ref{ass-focusing} is satisfied, then function $f$ satisfies inequalities \eqref{ineq-growth} with
$\beta=2N+1\geq 2$ in the defocusing case 1(a) and with $\beta=2\sigma+1\geq 2$ in the focusing case or the defocusing case 1(b).
\end{lemma}

Thus, Lemma \ref{lem-B03} and equation \eqref{NLS-abstract-Stranonovich-Ito} imply that
 in this framework we can reformulate equation \eqref{NLS-Strat_01} as
\begin{equation}\label{NLS-Strat_03}
\displaystyle
du(t)=i \Big[ \Delta u(t)- F(u)\Big]\,dt  + \frac{1}2
\mathfrak{p} \mathbb{M}(u)  \,dt +(-i)\tilde{G}(u)\,dW(t),
\end{equation}
where $\tilde{G}$ is the generamized Nemytski map defined by \eqref{generalizedG},
 $\mathbb{M}(u)=m\circ u$ is the Nemytski map associatied with $m(z)=-\tilde{g}(|z|^2)^2 z$ and
$\mathfrak{p}= \sum_j (\Lambda e_j)^2$ is defined in Lemma \ref{lem-W}.

 Let us fix  $u_0\in L^p(\Omega, {\mathcal F}_0, H^{1,2}(M))$ and observe that the  assumptions of Theorem \ref{th_global}
 imply  that  Theorem \ref{thm_maximal-manifold}
can be applied.
Hence   equation  \eqref{NLS-Strat_03} has a unique local maximal solution
 $u=\big(u(t), t\in [0, \tau_\infty)\big)$.
We will  show that our assumptions  on $f$, $g$ and on the noise $W$ are  sufficient  to ensure that the explosion time  $\tau_\infty$
 is a.s. infinite. This will be achieved by proving some conservation laws in the next  two subsections.\\
 Let us recall that according to Theorem \ref{thm_maximal-manifold} $\lim_{t \toup  \tau_\infty} \vert u(t)\vert_{H^{1,2}} =\infty$
 $\mathbb{P}$-a.s. on $\{ \tau_\infty <\infty\}$. Hence, the following stopping times are well defined (and finite on $\{ \tau_\infty <\infty\}$):
 \begin{equation}\label{eqn-tau_k-explosion}
 \tilde{\tau}_k:=\inf\big\{ t \in [0, \tau_\infty):  \vert u(t)\vert_{H^{1,2}} \geq k\big\}.
 \end{equation}

The aim of this section is to prove that $\tau_\infty =\infty$ a.s.
The proof of this result will be given in several steps. Recall that $\rH=H^{1,2}(M)$ and $\rE=W^{1-\frac{1}{p},q}(M)$.
 The main two steps are described in the following two sections. The first one is the a.s.  conservation of the $L^2(M)$
norm due to the Stratonovich integral and the deterministic conservation law. The second one is the use of a Lyapounov function. Note that
unlike the deterministic case, the It\^o-Stratonovich correction term implies that the expected value of this Lyapounov function
is not preserved.  However, it remains bounded  and this  implies that the expected value of the
$\rH$-norm of the maximal solution remains bounded, which  forbids the
explosion time to be finite.

\subsection{Preservation of the $L^2$-norm}
\label{subsec-L^2norm}

We at first prove that the $L^2(M)$-norm of this solution is almost surely constant in time
This extends classical results for the deterministic NLS equation (see \cite{Burq+G+T_2004} for the case of compact manifolds)
as well as \cite{deBouard+Deb_2003} deBouard Debussche for the flat stochastic  NLS equation.

\begin{lemma}\label{lem-L^2} Assume that $f$ and $g$  satisfy the Assumptions \ref{ass-focusing} and \ref{ass-tildeg-final}  respectively,
$p$ and $q$ satisfy the scaling admissibility condition $\frac{2}{p} + \frac{2}{q}=1$.
Let $(W(t), t\geq 0)$ be an $\rH\cap \rE$-valued Wiener process  and $u_0\in H$. Then
 $\vert u(t)\vert_{L^2(M)}=\vert u_0\vert_{L^2(M)}$, for all $t\in [0,\tau_\infty)$, $\mathbb{P}$-almost surely.
\end{lemma}
\begin{proof}
Let $\big(\tilde{\tau}_k\big)_k$ denote the approximating  sequence  of the stopping  time $\tau_\infty$ defined by \eqref{eqn-tau_k-explosion}.
 Suppose that we have proved that  for each $t\geq 0$ and $k\in {\mathbb N}$,
 $\vert u(t\wedge \tilde{\tau}_k) \vert_{L^2(M)}=\vert u_0\vert_{L^2(M)}$ $\mathbb{P}$-almost surely.
 Then it follows that there exists a set
 $\hat{\Omega}\subset \Omega$ of full $\mathbb{P}$-measure such that for each $\omega \in \hat{\Omega}$,
$\vert u(t,\omega)\vert_{L^2(M)}=\vert u_0\vert_{L^2(M)}$ for all $t \in \mathbb{Q}\cap [0,\tau(\omega))$. Thus, since for all
$\omega\in\hat{\Omega}$ the map $[0,\tau(\omega)) \ni t \mapsto u(t,\omega) \in L^2(M)$ is continuous,  the result will follow.

To prove the conservation of the $L^2(M)$-norm, let us consider  the  functional
\[
\varPhi: L^2(M) \ni u \mapsto   \frac12 \vert u \vert_{L^2(M)}^2= \frac12\int_M u(x)\overline{u(x)}\, dx \in \mathbb{R}.
\]
where  $dx$  denotes the integration with respect to the riemannian volume measure on $M$.
The function $\varPsi$ is of real-${\mathcal C}^\infty$ class and for all  $u , v,v_1,v_2 \in L^2(M)$, we have
\begin{eqnarray*}
\varPhi^\prime(u)(v)&=&d_u\varPhi(v)=\Re \langle u,v \rangle_{L^2}
 = \int_M \Re \big( u(x) \overline{v (x)} \big)\, dx,\\
\varPhi^{\prime\prime}(u)(v_1,v_2)&=&d^2_u\varPhi(v_1,v_2)
=\Re \langle v_1,v_2 \rangle_{L^2}= \int_M \Re \big(  v_1(x)   \overline{v_2 (x)} \big) \, dx.
\end{eqnarray*}
Let us now assume, for purely pedagogical reasons, that $u$ is a strong solution.
Applying the It\^o formula we obtain  for each $t\in \mathbb{R}_+$ and every $k\in\mathbb{N}$, $\mathbb{P}$-almost surely,
\begin{align*}
\varPhi&(u(t \wedge \tilde{\tau}_k))-\varPhi(u_0)=-\int_0^t 1_{[0,\tilde{\tau}_k)}(s) \langle \varPhi^\prime(u(s)), iG(u(s)) \rangle_{L^2}\,dW(s) \\
&\quad +\int_0^t  1_{[0,\tilde{\tau}_k)}(s) \Big\langle \varPhi^\prime(u(s)),  i \big[ \Delta u(s)- F(u(s))\big]
  +  \frac12 \mathrm{tr}_{K}(\Pi)  {\mathbb M}(u(s))  \Big\rangle_{L^2} \,ds
\\ &\quad +
\frac12 \int_0^t 1_{[0,\tilde{\tau}_k)}(s) \mathrm{tr}_{K}\Big[ \varPhi^{\prime\prime}(u(s))\big(iG(u(s)), iG(u(s))\big)\Big]\,      ds
\\
&= \int_0^t 1_{[0,\tilde{\tau}_k)}(s) \Re \langle u(s),  i  \Delta u(s)\rangle_{L^2}\,ds
 - \int_0^t 1_{[0,\tilde{\tau}_k)}(s) \Re \langle u(s),  i  F(u(s) \rangle_{L^2}\,ds
\\ &\qquad  + \frac12   \mathfrak{p}\int_0^t 1_{[0,\tilde{\tau}_k)}(s) \Re \langle u(s),
 {\mathbb M}(u(s))  \rangle_{L^2}\,ds
\\
&\qquad - \int_0^t 1_{[0,\tilde{\tau}_k)}(s) \Re \langle u(s), iG(u(s)) \rangle_{L^2}\,dW(s)
+ \frac12 \int_0^t 1_{[0,\tilde{\tau}_k)}(s)\sum_{j\geq 1} \vert G(u(s))\Lambda e_j\vert_{L^2}^2\,ds.
\end{align*}
Next we make the following three observations.
\begin{enumerate}
\item
Since $\Delta $ is self-adjoint in $L^2(M)$, we have  $\Re \langle u(s),  i  \Delta u(s)\rangle_{L^2(M)}=0$.
\item If $H$ be the Nemytski map associated with $h$ of the form  $h(z)=\tilde{h}(|z|^2) z$, where  $\tilde{h}:{\mathbb R}\to {\mathbb R}$,    then
\[
\Re \langle u(s),  i  H(u(s) \rangle_{L^2}
=
\int_M \tilde{h}(|u(s,x)|^2) \Re \big[ u(s,x)  \overline{i u(s,x)}\big] \, dx=0.
\]
\item Lemma \ref{lem-B03} implies that
  $ \mathfrak{p} \Re \langle u(s),  {\mathbb M}(u(s))  \rangle_{L^2}= -\sum_{j\geq 1} \vert G(u(s))\Lambda e_j\vert_{L^2}^2 $.
\end{enumerate}

 Therefore, we infer that
 that for each $t\geq 0$ and every $k\in\mathbb{N}$, $\mathbb{P}$-almost surely,
$  \varPhi(u(t\wedge \tau_k))-\varPhi(u_0)=0 $, that is
  $\vert u(t\wedge \tilde{\tau}_k) \vert_{L^2(M)}=\vert u_0\vert_{L^2(M)}$ $\mathbb{P}$-almost surely
 and the result follows.

A full proof can be made by replacing $u$ by its Yosida approximation as it has been done for instance in \cite{Brz+Masl+S_2005};
see also \cite{deBouard+Deb_1999} and \cite{deBouard+Deb_2003} for a similar approach.
\end{proof}

\subsection{The Lyapounov function}\label{subsec-global-Lyapunov}

As in the deterministic case, we will use some   Lyapunov function. Let $\tilde{F}$ denote the antiderivative of
$\tilde{f}$ such that $\tilde{F}(0)=0$. In this section, we will consider two cases as in Assumption \ref{ass-focusing}. \\

\textbf{Case 1(a).} We assume that  $\tilde{f}$ is a polynomial of degree  $ N$ with a positive leading coefficient.
 Hence  $\tilde{F}(r)=a_{N+1} r^{N+1} +Q(r)$, where $Q$ is a polynomial function
of degree at most  $ N$ such that $Q(0)=0$ and $a_{N+1}>0$.  We have the following result.
\begin{lemma}\label{lemQ}
Let $\tilde{F}$ and $Q$ be polynomial functions  as above.
 Then there exists a constant $C>0$, and for every $\varepsilon >0$, there exists a constant $C(\varepsilon)>0$
such that for all  $u\in L^{2N+2}(M)\supset {\mathcal R}$,
\begin{align}
\label{majotildeF} \int_M |u(x)|^{2N+2} dx &\leq C \int_M \tilde{F}(|u(x)|^2 dx + C  \int_M |u(x)|^2 dx,
\end{align}
\begin{align} \label{majoQ}
\Big\vert \int_M  Q(|u(x)|^2) \, dx \Big\vert &\leq \varepsilon \int_M |u(x)|^{2N+2} dx + C(\varepsilon) \int_M |u(x)|^2 dx.
\end{align}
Finally for $u\in H^{1,2}(M)$ we have
\begin{equation}\label{ineq-tildef-tildeF}
\int_M \vert \tilde{f}(\vert u(x)\vert^2) \vert u(x)\vert^2  \, dx \leq C \int_M \vert \tilde{F}(\vert u(x)\vert^2) \, dx
+ C \int_M \vert u(x)\vert^2 \, dx.
\end{equation}
\end{lemma}
\begin{proof}
Since $\dim(M)=2$,    the Gagliardo-Nirenberg inequality implies
$H^{1,2}(M)\subset \bigcap_{r\geq 2}L^r(M)$ and hence ${\mathcal R}\subset L^{2N+2}(M)$.
  Let us fix $\alpha >0$ and  $u\in L^{2N+2}(M)$. Then  for $k=2, \cdots, N$ by  the H\"older and
Young inequalities yield
\begin{align*}
\int_M |u(x)|^{2k} dx &\leq \Big( \int_M |u(x)|^{2(N+1)} dx \Big)^{\frac{k-1}{N}} \;  \Big( \int_M |u(x)|^{2} dx \Big)^{\frac{N+1-k}{N}}\\
&\leq \alpha \; \frac{k-1}{N} \int_M |u(x)|^{2(N+1)} dx + \frac{N+1-k}{N} \; \alpha^{-\frac{N+1-k}{N}}  \int_M |u(x)|^{2} dx.
\end{align*}
This concludes the proof of \eqref{majoQ}. Since  $a_N$ is positive,   we have
\[ \int_M  |u(x)|^{2N+2} dx \leq \frac{1}{a_{N+1}}\int_M  \tilde{F}(|u(x)|^2) dx -\frac{1}{a_{N+1}} \int_M Q(|u(x)|^2) dx.\]
Thus applying  \eqref{majoQ} with  $\varepsilon= \frac{1}{2a_{N+1}} $ concludes the proof of \eqref{majotildeF}.
Finally, $\tilde{f}(r)r=a_N r^{N+1} + \tilde{Q}(r)$ where $\tilde{Q}$ is a polynomial of degree $N$. Hence \eqref{majoQ} and \eqref{majotildeF}
yield \eqref{ineq-tildef-tildeF}.
\end{proof}
 I have put the corollary inside the lemma

\textbf{Case 1(b).} We assume that $\tilde{f}(r)=C r^\sigma$ for $C>0$ and $\sigma\geq \frac12$.
 Then  $\tilde{F} \geq 0$ and thus
 \[\int_M \tilde{F}(|u(x)|^2) dx  \geq 0, \mbox{ for any } u\in \cR  .\]
Furthermore,  since $r\tilde{f}(r) = C \tilde{F}(r)$, there exists some positive constant $C$ such that
\begin{equation}\label{ineq-tildef-tildeF_1b}
\int_M  \tilde{f}(\vert u(x)\vert^2) \vert u(x)\vert^2  \, dx \leq C \int_M \vert \tilde{F}(\vert u(x)\vert^2)\vert \, dx, \quad \forall u\in H^{1,2}(M).
\end{equation}

\textbf{Case 2.} We assume that $\tilde{f}(r)=-C r^\sigma$ for $C>0$ and $\sigma \in [\frac12,1)$.
Then $\tilde{F}(r)=-\frac{C}{\sigma+1} r^{\sigma+1}$. The following lemma will be used to deal with this case.
\begin{lemma}  \label{lemGagliardo}
Assume that  $\alpha \in (1,3)$. Then for any $\varepsilon >0$ there exists $C(\varepsilon)>0$ such that for any
 $u\in H^{1,2}(M)$ (and for any  $u\in \cR $),
\begin{equation} \label{Gagliardo}
\int_M |u(x)|^{\alpha+1} dx \leq \varepsilon |\nabla u|_{L^2}^2 + C(\varepsilon) |u|_{L^2}^{\frac{4}{3-\alpha}}.
\end{equation}
Furthermore, for $\sigma \in [\frac{1}{2},1)$ there exists  $\tilde{C}>0$  such that for  every $u\in H^{1,2}(M)$,
\begin{equation}\label{ineq-tildef-tildeF_2}
\int_M  \vert \tilde{f}(\vert u(x)\vert^2)\vert  \vert u(x)\vert^2  \, dx \leq
 \frac{1}{2} |\nabla u|_{L^2}^2 + \frac{1}{2} \int_M \tilde{F}(|u(x)|^2) dx+ \tilde{C}  \vert u\vert_{L^2}^{\frac2{1-\sigma}} .
\end{equation}
\end{lemma}
\begin{proof}
The Gagliardo-Nirenberg and Young inequalities imply that for $u\in H^{1,2}(M)$ and $\varepsilon >0$,
\[
\int_M |u(x)|^{\alpha+1} dx \leq C \, |\nabla u|_{L^2}^{\alpha-1} \, |u|_{L^2}^2
\leq \varepsilon |\nabla u|_{L^2}^2 + C \varepsilon^{- \frac{2(\alpha-1)}{3-\alpha}}\, |u|_{L^2}^{\frac{4}{3-a}}.
\]
This concludes the proof of \eqref{Gagliardo}. Finally, $\vert \tilde{f} (r)\vert r=c r^{\sigma +1}$, $r\geq 0$ and since
$\sigma <1$ we have
\[ \int_M  \vert \tilde{f}(\vert u(x)\vert^2)\vert  \vert u(x)\vert^2  \, dx   - \frac{1}{2} \int_M \tilde{F}(|u(x)|^2) dx
 =C\Big(1+\frac{1}{2(\sigma+1)}\Big)\int_M \vert u(x)\vert^{2+2\sigma}\, dx.
\]
Hence using \eqref{Gagliardo} with $\eps=\frac{1+\sigma}{2C(2\sigma+3)}$ concludes the proof.
\end{proof}

Let us assume that $\tilde{f} $ satisfy Assumption \ref{ass-focusing}, hence either the conditions of Case \textbf{1(a)}, \textbf{1(b)}  or
 \textbf{2} above.  Let us define  the  map
\begin{equation}  \label{Lyapunov}
\varPsi: \cR \ni u \mapsto \frac{1}{2} |\nabla u|_{L^2}^2 + \frac{1}{2} \int_M \tilde{F}(|u(x)|^2) dx \in \mathbb{R}.
\end{equation}
Using Lemmas \ref{lemQ} or \ref{lemGagliardo}, it is easy to see that there exists a constant $C>0$ such that
$\varPsi(u) + C |u|_{L^2}^2 \geq 0$   for all  $u\in \cR $. This proves the following
\begin{corollary}\label{cor-Psi-H^1} There exists a constant $c \geq 0$ such that
\[\vert u\vert^2_{H^{1,2}} \leq 2 \Psi(u)+ c\vert u\vert^2_{L^{2}}, \;\; u\in \cR .
\]
\end{corollary}
We will need the following result about the regularity of  $\Psi$ and some of its properties.
\begin{lemma}\label{lem-Psi} The function $\Psi$ defined by \eqref{Lyapunov} is of real ${\mathcal C}^2$-class with the second derivative bounded on balls;
 for all $u, v_1, v_2 \in  \cR $, we have
\begin{eqnarray*}
\varPsi^\prime(u)(v)
&=& \,\mathrm{Re}\, \int_M   \nabla u(x) \overline{\nabla v (x)} \, dx
 + \int_M  \tilde{f} (|u(x)|^2) \,\mathrm{Re}\, [ u(x) \overline{v (x)}] \, dx , \\
\varPsi^{\prime\prime}(u)(v,v)
&=&  \int_M   \vert \nabla v(x)\vert^2 \, dx  +
 \int_M  \tilde{f} (|u(x)|^2) \vert v(x) \vert^2 dx   \nonumber
\\
&&  +   2\int_M  \tilde{f}^{\prime} (|u(x)|^2) \Big( \,\Re\, [ u(x) \overline{v(x)}]\Big)^2\, dx.
\end{eqnarray*}
Moreover,
\begin{eqnarray} \label{eqn-Psiprime-drift}
 \langle \varPsi^\prime ( u), i[\Delta u - F(u)] \rangle &=& 0, \;\;  u\in H^{2,2}(M),\\
\label{eqn-Psiprime-diff} \langle \varPsi^\prime ( u), i G(u) \rangle& = &\int_M \Re\, \big( \nabla u(x) \overline{\nabla i g(u(x))}\big) dx, \;\;  u\in\cR,\\
 \label{eqn-Phibis-3}
 \mathrm{tr}_{K} \varPsi^{\prime\prime}(u)\big(iG(u), iG(u)\big)
&\leq & 2 \int_M   \vert g^\prime(u(x)) \nabla u(x)  \big) \vert^2  \mathfrak{p} (x)\, dx
+ 2 \int_M   \vert g(u(x)) \vert^2 \mathfrak{q}(x)  \, dx
\nonumber\\
&& +
\int_M  \tilde{f} (|u(x)|^2)  \vert g(u(x)) \vert^2 \mathfrak{p} (x)\, dx \;\;  u\in\cR ,
\end{eqnarray}
where
\end{lemma}
\begin{proof}
 The regularity of $\Psi$ and the explicit expressions of its first and second derivatives follow from section \ref{sec-Nemytski},
in particular from Remark  \ref{focusing}.
Note that if $\phi : {\mathbb R}\to {\mathbb R}$, we have
\begin{equation} \label{zero}
 \Re \big( z \; \overline{ i \phi(|z|^2) z}\, \big) =0, \quad \forall z\in {\mathbb C}.
\end{equation}
Integration by parts implies that for every $u\in H^{2,2}(M)$, we have
\[ \Re \int_M \nabla u(x) \overline{\nabla [i \Delta u(x)]} \, dx = - \Re \int_M \Delta u(x) \overline{ i \Delta u(x)}\, dx =0.\]
Thus \eqref{zero} applied to $\phi=\tilde{f}$ yields  \eqref{eqn-Psiprime-drift}.
Equality \eqref{eqn-Psiprime-diff} is a consequence of  \eqref{zero} applied with $\phi=\tilde{g}$.

We now prove the last assertion.
The inclusion   $\cR \subset L^\infty(M)$ implies that for every $u\in \cR$, $g\circ u\in L^\infty(M)$,  and the integral
$\int_M   \vert g(u(x)) \vert^2 \mathfrak{q}(x)  \, dx $ exists since $\mathfrak{q}\in L^1(M)$ as proved in Lemma \ref{lem-W}.\\

 By the definition of the trace, since $\Lambda e_j(x)\in  {\mathbb R}$, using the Cauchy-Schwarz inequality and
\eqref{zero} with $\phi = \tilde{g}$, we deduce
\begin{eqnarray}\nonumber
&&\hspace{-1truecm}\lefteqn{ \mathrm{tr}_{K} \varPsi^{\prime\prime}(u)\big(iG(u), iG(u)\big)
= \sum_{j\geq 1} \; \varPsi^{\prime\prime}(u)\big( i(g\circ u) e_j, i (g\circ u) \Lambda e_j\big)}
\\
\nonumber
&=&  \sum_{j\geq 1} \;\int_M   \Big| \nabla\Big( g(u(x)) \Lambda e_j(x)\Big) \Big|^2 \, dx  +
\sum_{j\geq 1} \;  \int_M  \tilde{f} (|u(x)|^2) \, \vert   g(u(x)) \Lambda e_j(x) \vert^2\, dx   \\
&&+   2 \sum_{j\geq 1} \; \int_M  \tilde{f}^{\prime} (|u(x)|^2) \,\Big(\Re\, [ u(x) \overline{i u (x)\tilde{g}(\vert u(x)\vert^2)} \Lambda e_j(x)] \Big)^2\, dx.
\nonumber\\
\nonumber
&\leq &  2 \int_M   \vert g^\prime(u(x)) \nabla u(x) \vert^2 \Big(\sum_{j\geq 1} \vert  \Lambda e_j(x) \vert^2 \Big) \, dx
+  2\int_M   \vert g(u(x)) \vert^2 \Big( \sum_{j\geq 1} \vert \nabla \Lambda e_j(x) \vert^2 \Big)\, dx \\
\nonumber
&& +
 \int_M  \tilde{f} (|u(x)|^2)  \vert g(u(x)) \vert^2 \Big( \sum_{j\geq 1}\vert  \Lambda e_j(x) \vert^2 \Big)\, dx.
\end{eqnarray}
This proves  \eqref{eqn-Phibis-3} and concludes the proof of the Lemma.
\end{proof}

The following lemma gives an explicit expression of $\varPsi(u(t))$, where  $\big(u(t), t\in [0,\tau_\infty)\big)$ denotes
 the local maximal solution to \eqref{NLS-Strat_03}.
Note that, unlike in the deterministic case, the It\^o-Stratonovich correction term yields that
 $\mathbb{E} ( \varPsi(u(t)))$ is not time invariant.

\begin{lemma}\label{lem-energy}
 Assume that $(W(t), t\geq 0)$ is an $\cR$-valued Wiener process. Then in the framework above,
  for every $t\geq 0$ and every $k\in \mathbb{N}^\ast$,  we have
\begin{equation} \label{ItoPsi}
 \varPsi(u(t \wedge \tilde{\tau}_k))  =  \varPsi(u_0)  -\int_0^{t \wedge \tilde{\tau}_k} \int_M \Re\,
\big( \nabla u(s,x) \overline{\nabla i g(u(s,x))}\big)\, dx \,dW(s) + T(t\wedge \tilde{\tau}_k),
\end{equation}
where
  \begin{align} \label{majo-ItoPsi}
  T(t\wedge \tilde{\tau}_k) \leq  &
\int_0^{t \wedge \tilde{\tau}_k} \!\! \int_M \vert g^\prime(u(s,x))\vert^2 \, |\nabla u(s,x)|^2 \,
\mathfrak{p}(x)
\, dx\, ds %\nonumber \\
+ \int_0^{t \wedge \tilde{\tau}_k} \!\! \int_M \vert g(u(s,x))\vert^2 \, \mathfrak{q}(x) \,dx \, ds \nonumber \\
& + \frac{1}{2}
 \int_0^{t \wedge \tilde{\tau}_k} \!\! \int_M  \tilde{f} (|u(s,x)|^2)  \vert g(u(s,x)) \vert^2 \mathfrak{p}(x)\, dx\, ds
\end{align}
\end{lemma}
\begin{proof}
According to Lemma \ref{lem-Psi} $\Psi$ is of class ${\mathcal C}^2$. Thus the proof of  \eqref{ItoPsi}
is done using  the It\^o Lemma for Yosida approximations of the solution and passing to the limit
(see e.g. \cite{Brz+Masl+S_2005} for a more detailed justification).
The upper estimate \eqref{majo-ItoPsi} of the "quadratic variation" is deduced from
 using \eqref{eqn-Phibis-3}.
\end{proof}

\subsection{Existence of a global solution}

We can now state the main result of this section, proving that the non linear stochastic Schr\"odinger equation
\eqref{NLS-Strat_01}, or \eqref{NLS-Strat_03}, has a unique global solution.
\begin{theorem}\label{th_global} Assume that the function $\tilde{f}$ satisfy  Assumption \ref{ass-focusing},
 that $g$   satisties
  Assumption  \ref{ass-tildeg-final} and that the Wiener process  $(W(t), t\geq 0)$ satisfies Assumption \ref{ass-W}.
Let $\beta$ be defined in Lemma  \ref{lem-f},  $p>\beta \vee \gamma $ where $\gamma$ is defined in Assumption \ref{ass-tildeg-final},
 and let $q$ be such that $(p,q)$ satisfy the scaling admissible condition $\frac{2}{p}+ \frac{2}{q}=1$.
Suppose furthermore that $u_0\in H^{1,2}(M)$.
Then the stochastic NLS equation  \eqref{NLS-Strat_01} has a unique global solution whose trajectories  belong a.s.  to
 ${\mathcal C}([0,\infty),H^{1,2}(M))$.
\end{theorem}
\begin{proof}
Let $u=\big(u(t)\, , \, t<  \tau_\infty \big)$
 belonging to $\mathbb{M}^p_{\textrm{loc}}(Y_{[0,\tau_\infty)})$, be the unique local maximal solution to the  problem
\eqref{NLS-Strat_01}.
 Note  that $\limsup_{t\to \tau_\infty} |u(t)|_{H^{1,2}} = +\infty$ a.s. on $\{ \tau_\infty<\infty \}$; for an integer $k\geq 1$
recall that  $\tilde{\tau}_k=\inf\{t\geq 0 : |u(t)|_{H^{1,2}}\geq k\}$.
Using the Khashminskii test for non-explosions (see \cite[Theorem
III.4.1]{Kh_1980} for the finite-dimensional case) and arguing as in \cite[page 7]{Brz+Masl+S_2005} it is sufficient to show that
  each $t>0$, there exists  a constant $C_t>0$ such that
\begin{equation} \label{normfinite_2} \mathbb{E}\Big( |u(t\wedge \tilde{\tau}_k)|_{H^{1,2}}^2 \Big) \leq C_t, \;\;  \mbox{ for every } k\in\mathbb{N}^\ast.
\end{equation}
In view of Corollary \ref{cor-Psi-H^1} and  Lemma \ref{lem-L^2}  it is sufficient to find, for each $t>0$, a constant  $C_t>0$ such that
\begin{equation} \label{Psifinite} \mathbb{E}\Big( \Psi(u(t \wedge \tilde{\tau}_k )) \Big) \leq C_t, \;\; \mbox{for every }   k\in\mathbb{N}^\ast.
\end{equation}
Since $W^{1, 2 s_0}(M,{\mathbb R})\cap W^{\hat{s},q}(M,{\mathbb R})\subset \cR$, the assumptions of Lemma \ref{lem-energy} are satisfied.
Hence, for each $t\in \mathbb{R}_+$ and every $k\in\mathbb{N}$, $\mathbb{P}$-almost surely,
\begin{align} \label{BDG_Psi}
\mathbb{E}  \Psi(t\wedge \tilde{\tau}_k) \leq & \mathbb{E} \Psi(u_0)
 + C  |\mathfrak{p}|_\infty\mathbb{E} \int_0^{t\wedge \tilde{\tau}_k} \int_M |\nabla u(s,x)|^2 \,dx \, ds \\
\nonumber
&  + C \sum_{j\geq 1} \mathbb{E} \int_0^{t\wedge \tilde{\tau}_k} \int_M  |u(s,x)|^2 \vert \nabla \Lambda e_j(x) \vert^2    \,dx \,ds
\\
\nonumber
& +
C |\mathfrak{p}|_\infty \, \mathbb{E}  \int_0^{t \wedge \tilde{\tau}_k} \; \int_M  \tilde{f} (|u(s,x)|^2)  \vert u(s,x) \vert^2 \, dx\, ds.
\end{align}
Let $s_0^\ast$ denote the conjugate exponent to $s_0$. The Gagliardo-Nirenberg  inequality proves that $H^{1,2}(M) \subset L^{2 s_0^*}(M)$;
 H\"older's inequality,  Lemma \ref{lem-W} and Corollary \ref{cor-Psi-H^1} imply that for $u\in H^{1,2}(M)$,
\[
\sum_{j\geq 1} \int_M  |u(x)|^2   \vert \nabla \Lambda e_j(x) \vert^2  \,dx \leq  \|u\|_{L^{2s_0^\ast}}^2
\sum_j \|\nabla \Lambda e_j\|_{L^{2 s_0}}^2
\leq C \big[  \Psi(u) + |u|_{L^2(M)} \big] .
\]
Next, the inequalities  \eqref{ineq-tildef-tildeF}, \eqref{ineq-tildef-tildeF_1b} and \eqref{ineq-tildef-tildeF_2},
imply the existence of positive constants $C$ and $\delta $ such that for all $u\in H^{1,2}(M)$,
\[ \int_M  \tilde{f} (|u(x)|^2)  \vert u(x) \vert^2 \, dx \leq C \big[ \Psi(u)+   \vert u \vert_{L^2(M)}^{\delta}\big] .
\]
Therefore, the conservation of energy proved in Lemma \ref{lem-L^2} and the above estimates  imply the
 existence of  an increasing  function $\phi:\mathbb{R}_+\to \mathbb{R}_+$ and a constant $C_t>0$ such that
\begin{align} \label{BDG_Psi_2}
\mathbb{E}  \Psi(t\wedge \tilde{\tau}_k)
\leq & \mathbb{E} \Psi(u_0)  + C \mathbb{E} \int_0^{t\wedge \tilde{\tau}_k} \Psi (u(s)) \, ds
+ C \mathbb{E} \int_0^{t\wedge \tilde{\tau}_k} \phi (\vert u(s\wedge \tilde{\tau}_k)\vert_{L^2(M)}) \, ds
\\
\nonumber
\leq &\mathbb{E} \Psi(u_0)  + C \mathbb{E} \int_0^{t} \Psi (u(s\wedge \tilde{\tau}_k)) \, ds + C\phi (\vert u_0\vert_{L^2(M)}).
\end{align}
The  Gronwall Lemma yields that for some constant $C>0$ the upper estimate
\begin{align} \label{BDG_Psi_3}
\mathbb{E}  \Psi(t\wedge \tilde{\tau}_k) &\leq \big[ \mathbb{E} \Psi(u_0) + Ct\phi (\vert u_0\vert_{L^2(M)})\big] e^{Ct},\;\; t\geq 0,
\end{align}
holds for every integer $k\geq 1$. This concludes the proof of \eqref{Psifinite} and hence that of the Theorem.
\end{proof}

\noindent\textbf{Acknowledgements:} This paper was initiated during a visit of the first named authour to the Universit\'e  Paris 1 in April 2009
and  partly written in 2010 while both authors were attending the semester
Stochastic Partial Differential Equations  at the Isaac Newton Institute for Mathematical Sciences (Cambridge).
They would like to thank their host institutions for the warm hospitality, the excellent working conditions and the financial support.

\end{document}